\documentclass[11pt]{amsart}
\usepackage{amsmath,amssymb,amsthm}

\RequirePackage{color}
\RequirePackage[colorlinks,urlcolor=my-blue,linkcolor=my-red,citecolor=my-green]{hyperref}
\definecolor{my-blue}{rgb}{0.0,0.0,0.6}
\definecolor{my-red}{rgb}{0.5,0.0,0.0}
\definecolor{my-green}{rgb}{0.0,0.5,0.0}

\newtheorem{theorem}{\sc Theorem}[section]
\newtheorem{lemma}[theorem]{\sc Lemma}
\newtheorem{proposition}[theorem]{\sc Proposition}
\newtheorem{corollary}[theorem]{\sc Corollary}

\numberwithin{equation}{section}
\theoremstyle{remark}
\newtheorem{remark}[theorem]{Remark}

\newcommand{\be}{\begin{equation}}
\newcommand{\ee}{\end{equation}}
\newcommand{\nn}{\nonumber}
\providecommand{\abs}[1]{\vert#1\vert}
\providecommand{\norm}[1]{\Vert#1\Vert}

\newcommand{\fl}[1]{\lfloor{#1}\rfloor}

\def\cZ{\mathcal{Z}}
\def\cI{\mathcal{I}}

\def\bE{\mathbb{E}}
\def\bN{\mathbb{N}}
\def\bP{\mathbb{P}}

\def\bR{\mathbb{R}}
\def\bZ{\mathbb{Z}}

 \def\Util{\wt{U}} \def\Vtil{\wt{V}}

\def\w{\omega}
\def\om{\omega}

\def\e{\varepsilon}
\def\ind{\mathbf{1}}

\def\mE{\mathbf{E}}   

\def\mP{\mathbf{P}}
\def\m1{\mathbf{1}}

\DeclareMathOperator{\Var}{Var}   \DeclareMathOperator{\Cov}{Cov}  
 \def\Vvv{{\rm\mathbb{V}ar}}  \def\Cvv{{\rm\mathbb{C}ov}}

 \def\wt{\widetilde}   
 
\def\E{\bE}
\def\P{\bP} 
\def\xexit{{\xi_x}}   
\def\yexit{{\xi_y}}   
\newcommand{\xexitk}[2]{{\xi_x^{(#1,#2)}}}   
\newcommand{\yexitk}[2]{{\xi_y^{(#1,#2)}}} 
\def\dxexit{{\xi^*_x}}  \def\dyexit{{\xi^*_y}}    
\def\vva{v_0} \def\vvb{v_1}   
\def\wwa{w_0} \def\wwb{w_1} 
\def\north{S_{\mathcal N}}  \def\south{S_{\mathcal S}}  \def\east{S_{\mathcal E}}  
\def\west{S_{\mathcal W}}

\def\backpi{\overleftarrow{\pi}} 

\def\Zalt{Z^{\scriptscriptstyle\,\square}}

\def\digamf{\Psi_0} 
\def\trigamf{\Psi_1} 
\def\functaa{L} 
\def\functaabar{\bar L} 
\def\Pann{P} \def\Eann{E}  

\def\freee{f} 
\def\mnc{\kappa} 
\allowdisplaybreaks[1]
\begin{document}

\title[Scaling for a polymer]
{Scaling for a one-dimensional directed polymer with  boundary conditions (revised)}
\author[T.~Sepp\"al\"ainen]{Timo Sepp\"al\"ainen}
\address{Timo Sepp\"al\"ainen\\ University of Wisconsin-Madison\\ 
Mathematics Department\\ Van Vleck Hall\\ 480 Lincoln Dr.\\  
Madison WI 53706-1388\\ USA.}
\email{seppalai@math.wisc.edu}
\urladdr{http://www.math.wisc.edu/~seppalai}
\thanks{The author was partially supported by 
National Science Foundation grant DMS-0701091 and by the
Wisconsin Alumni Research Foundation.} 
\date{Last updated August 26, 2015.}  

\begin{abstract} 
We study a 1+1-dimensional directed polymer in a random
environment on 
the integer lattice with  log-gamma distributed weights.   
Among directed polymers this model is special in the same way 
as the last-passage percolation model with exponential or geometric
weights is special among growth models, namely,  both permit explicit
calculations.  With  appropriate boundary conditions
 the polymer with log-gamma weights satisfies 
an analogue of Burke's theorem for queues. Building on this 
we prove  the conjectured values for the  fluctuation exponents of the free energy and
the polymer path,  in the case where the boundary conditions are present and 
 both endpoints of the polymer path are fixed. 
 For the polymer 
without boundary conditions and 
with either fixed or free endpoint we get 
the expected upper bounds on the exponents. 

This is a corrected and improved version of the paper published in Ann.~Probab.~40 (2012) 19--73.   The differences between this version and the published version are explained at the end of the Introduction. 

\end{abstract}
\maketitle


\section{Introduction}
The {\sl directed polymer in a random environment} 
represents  a polymer (a long chain of molecules) by a
 random walk path that interacts with a random 
environment.  Let $x_\centerdot=(x_k)_{k\ge 0}$ denote a nearest-neighbor 
path in $\bZ^d$ started at the origin: $x_k\in\bZ^d$,  $x_0=0$,  and 
$\abs{x_k-x_{k-1}}=1$.   The environment $\om=(\om(s,u):s\in\bN, u\in\bZ^d)$
puts a real-valued weight 
$\om(s,u)$ at space-time point $(u,s)\in\bZ^d\times\bN$.
 For a  path segment $x_{0,n}=(x_0,\dotsc, x_n)$,   $H_n(x_{0,n})$
  is the total weight collected by 
the walk up to time $n$:  $H_n(x_{0,n})=\sum_{s=1}^n \om(s, x_s)$.
 The quenched polymer distribution on paths,  in
environment $\om$ and 
at inverse temperature $\beta>0$,  is the probability measure   defined by
\be  Q^\om_n(dx_\centerdot)= \frac1{Z^\om_n} \exp\{\beta H_n(x_{0,n}) \} 
\label{poly-meas}\ee
with normalization factor (partition function) 
$ Z^\om_n= \sum_{x_{0,n}} e^{\beta H_n(x_{0,n})}.$
The environment $\w$ is taken as random with   probability distribution
 $\bP$, typically   
  such that the  weights $\{\om(s,u)\}$ are i.i.d.\ random variables. 

At $\beta=0$ the model is standard simple random walk. 
 The general objective  is  to understand how the model behaves as
 $\beta>0$ and the dimension $d$ varies.  
A key question is whether the diffusive behavior of the walk is 
affected.  ``Diffusive behavior'' refers to  the fluctuation behavior 
of standard  random walk,  characterized by $n^{-1}E(x_n^2)\to c$ 
and convergence of 
diffusively rescaled walks  $n^{-1/2}x_{\fl{nt}}$   to 
Brownian motion. 

The directed polymer model was introduced in the statistical 
physics literature by Huse and Henley in 1985 \cite{huse-henl}.
The first rigorous mathematical work was by Imbrie and Spencer
\cite{imbr-spen} in 1988. They  proved with an elaborate
expansion  that in dimensions $d\ge 3$
and with small enough $\beta$, 
the walk is diffusive in the sense that, for a.e.\ environment $\w$,  
$n^{-1}E^{Q^\om}(\abs{x_n}^2)\to c$.  Bolthausen  \cite{bolt-cmp-89} strengthened
the result to a central limit theorem for the endpoint of the walk,
 still $d\ge 3$, small $\beta$ and for a.e.\ $\w$,
  through the observation that
 $W_n=Z_n/\bE(Z_n)$ is a martingale. 
   Since then martingale techniques
  have been a standard fixture in much of the work on 
directed polymers.  

The limit  $W_\infty=\lim W_n$ is either almost surely $0$ or almost surely $>0$. 
The case $W_\infty=0$ has been termed {\sl strong disorder} and 
$W_\infty>0$ {\sl weak disorder}.   There is a critical 
value $\beta_c$ such that weak disorder holds for $\beta<\beta_c$ and 
strong for $\beta>\beta_c$.  It is known that
 $\beta_c=0$ for $d\in\{1,2\}$ and   
$0<\beta_c\le\infty$ for $d\ge 3$. 
In $d\ge 3$ and weak disorder the walk converges to a Brownian
motion, and the limiting  diffusion matrix is the same
 as for standard random walk 
\cite{come-yosh-aop-06}. 
There is a further refinement of strong disorder into  
 strong and very strong  disorder.  Sharp  recent results appear in 
\cite{laco-10}.

  One way to phrase questions about the polymer model  is to ask about two 
scaling exponents,  $\zeta$ and $\chi$, defined somewhat informally 
as follows: 
\be  \text{ fluctuations of the path
 $x_{0,n}$ are  of order $n^\zeta$ } \label{poly-xi}\ee
and  
\be  \text{  fluctuations of  
 $\log Z_n $ are of order $n^\chi$. } \label{poly-chi}\ee

Let us   restrict ourselves  to the case $d=1$ for the remainder
of the paper.  By the results mentioned above 
the model is in strong disorder for all $\beta>0$.  
It is expected that the 1-dimensional exponents
are  $\chi=1/3$ and $\zeta=2/3$
\cite{krug-spoh}.  
Precise values  have not been obtained in the past, but during the last decade
and a half 
nontrivial rigorous bounds have appeared in the literature 
 for some   models with Gaussian ingredients. 
 For  a Gaussian random walk    in a Gaussian potential  
   Petermann \cite{petermann}  proved the 
lower  bound $\zeta\ge 3/5$ 
and Mejane \cite{meja-04} provided the upper bound $\zeta\le 3/4$.  
Petermann's proof was adapted to a certain continuous setting
in  \cite{beze-tind-vien}. 
For an undirected  Brownian motion in a Poissonian potential 
   W\"uthrich 
 obtained $3/5\le \zeta\le 3/4$  and $\chi\ge 1/8$   \cite{ wuth-98aihp,wuth-98aop}. 
For a directed  Brownian motion in a Poissonian potential  Comets and 
Yoshida derived $\zeta\le 3/4$ and $\chi\ge 1/8$   \cite{come-yosh-05}. 
  
  Piza \cite{piza-97} showed generally that the
fluctuations of $\log Z_n$ diverge at least logarithmically, and bounds on
exponents under curvature assumptions on the limiting free energy.  
Related results for first passage percolation appeared in 
\cite{lice-newm-piza,newm-piza-95}.  

Exact exponents and even limit distributions have recently been derived for the
so-called {\sl continuum directed random polymer}.  The partition function 
  $\cZ(t,x)$  is the solution  of a stochastic heat equation 
$\cZ_t=\tfrac12 \cZ_{xx} -\cZ\dot W$ where $\dot W$ is space-time white noise.  
In \cite{bala-quas-sepp} the exact scaling exponent is determined for initial
data $\cZ(0,x)=e^{-B(x)}$  where $B$ is a two-sided Brownian motion:
$\Var(\log \cZ(t,0))$ is of order $t^{2/3}$.  The result comes from corresponding
bounds for the current of the weakly asymmetric simple exclusion process (WASEP). The
techniques are related to the ones used in the present paper.  
The link from WASEP to $\log\cZ$ that enables this transfer of estimates 
 is originally due to \cite{bert-giac-97}.  
\cite{amir-corw-quas}  obtains the probability distribution of $\log \cZ$ 
for an initial delta function $\cZ(0,x)=\delta_0(x)$ and 
 proves a Tracy-Widom limit under the appropriate scaling.   
   The WASEP connection is used again 
in \cite{amir-corw-quas}, together with asymptotic analysis of a determinantal
formula from \cite{trac-wido-09cmp}.  There is no methodological overlap 
between \cite{amir-corw-quas} and the present paper.

\begin{figure}[ht]
\begin{center}
\begin{picture}(220,140)(5,20)
\put(40,20){\vector(1,0){150}}
\put(40,20){\vector(0,1){130}}
\put(40,40){\line(1,0){138}}
\put(63,20){\line(0,1){120}}
\multiput(40,60)(0,20){5}{\line(1,0){138}}
\multiput(86,20)(23,0){5}{\line(0,1){120}}
\multiput(49,27)(23,0){3}{\large$\bullet$}
\multiput(95,47)(0,20){3}{\large$\bullet$}
\multiput(118,87)(23,0){3}{\large$\bullet$}
\multiput(164,107)(0,20){2}{\large$\bullet$} 
\multiput(56.5,30)(23,0){2}{\vector(1,0){14}}
\multiput(98.1,34.1)(0,20){3}{\vector(0,1){12}}
\multiput(102.5,90)(23,0){3}{\vector(1,0){14}}
\multiput(167.2,93.8)(0,20){2}{\vector(0,1){12.2}}
\put(194.7,17){$x$} \put(30,146.5){$y$}
\put(50,10){\small 0} \put(73,10){\small 1}
\put(96,10){\small 2} \put(119,10){\small 3}
\put(142,10){\small 4}\put(165,10){\small 5}
\put(29,27){\small 0} \put(29,47){\small 1} \put(29,67){\small 2}
\put(29,87){\small 3} \put(29,107){\small 4} \put(29,127){\small 5}
\end{picture}
\end{center}
\caption{ \small 
An up-right path from $(0,0)$ to $(5,5)$ in $\bZ_+^2$.}
\label{fig1}\end{figure}
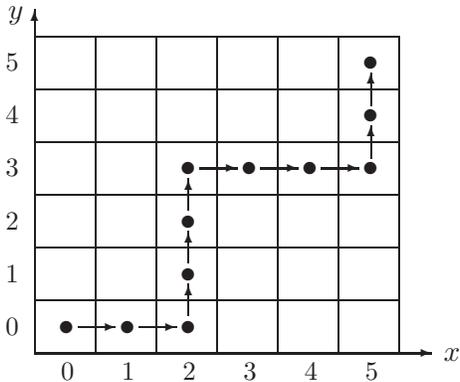%

Let us return to the  1+1-dimensional lattice polymer.  
For the rest of the discussion we   turn 
the picture 45 degrees clockwise so that   the 
model lives in the nonnegative quadrant 
$\bZ_+^2$ of the plane,  instead of the space-time  wedge
$\{(u,s)\in\bZ\times\bN: \abs{u}\le s\}$. 
The weights are i.i.d.\ variables $\{\om(i,j): i,j\ge 0\}$.   
The polymer $x_\centerdot$ becomes
a nearest-neighbor {\sl up-right} path (see Figure \ref{fig1}). 
We also fix both endpoints of the path.  So, given 
the endpoint $(m,n)$,  the partition function is 
\be  Z^\om_{m,n} = \sum_{x_{0,m+n}}  
\exp\Bigl\{  \beta \sum_{k=1}^{m+n}\om(x_k)\Bigr\}  
\label{defZmn}\ee
where the sum is over paths $x_{0,m+n}$ 
 that satisfy    $x_0=(0,0)$, $x_{m+n}=(m,n)$ and 
 $x_k-x_{k-1}=(1,0)$ or $(0,1)$. The polymer measure
of such a path is  
\be Q^\om_{m,n}(x_{0,m+n})  = \frac1{Z^\om_{m,n}}
\exp\Bigl\{  \beta \sum_{k=1}^{m+n}\om(x_k)\Bigr\}. 
\label{defQmn}\ee

  If we take
 the ``zero temperature limit''   $\beta\nearrow\infty$
in \eqref{defQmn} then the measure $Q^\om_{m,n}$ concentrates on the
paths $x_{0,m+n}$ that maximize the sum $\sum_{k=1}^{m+n} \om(x_k)$.  
Thus    the polymer model has become a 
{\sl last-passage percolation model}, also called the 
{\sl corner growth model}. 
The quantity that corresponds to $\log Z_{m,n}$ 
is  the {\sl passage time} 
\be G_{m,n}= \max_{x_{0,m+n}}  
 \sum_{k=1}^{m+n}\om(x_k). \label{Gmn}\ee 
  For certain last-passage growth  models, notably for \eqref{Gmn}
with exponential or geometric weights $\w(i,j)$,
 not only have the predicted
 exponents  been confirmed but also limiting
Tracy-Widom fluctuations for $G_{m,n}$ have been proved
 \cite{baik-deif-joha-99, bala-cato-sepp, cato-groe-06,
ferr-spoh-06, joha, joha-ptrf-00}.
The recent article \cite{bena-corw} verifies a complete
picture proposed in 
\cite{prah-spoh} that characterizes the scaling limits of 
$G_{m,n}$ with exponential weights as a function of the 
parameters of the boundary weights and the ratio $m/n$.

In the present paper  we study the polymer model 
\eqref{defZmn}--\eqref{defQmn} 
with fixed endpoints,  with fixed $\beta=1$,
and  for a particular choice of weight
distribution.  Namely, the weights $\{\w(i,j)\}$ are independent random
variables with log-gamma distributions. Precise definitions
follow in the next section.  This particular polymer
model turns out to be amenable to explicit computation,
similarly  to the case of exponential or geometric  weights 
among the corner growth models \eqref{Gmn}.  

 We introduce a polymer model with boundary conditions that 
possesses  a  
two-dimensional stationarity property.
By boundary conditions we mean that the weights on the boundaries
of $\bZ_+^2$ are distributionally different from the weights in
the interior, or bulk. 
 For the model with
boundary conditions we prove that the fluctuation  exponents take
exactly their conjectured  values  $\chi=1/3$ and $\zeta=2/3$ when
the endpoint $(m,n)$ is taken to infinity along a 
 characterictic direction. This characteristic direction
 is a function of the parameters
of the weight distributions.  
In other directions $\log Z_{m,n}$ satisfies a central limit
theorem in the model with boundary conditions.  
 As a corollary
we get the correct upper bounds for the exponents in the model
without boundary and with either fixed or free endpoint, but still with 
 i.i.d.\ log-gamma weights $\{\w(i,j)\}$.

In addition to the  $\beta\nearrow\infty$ limit, there is another
formal connection between the polymer model and the corner growth
model. Namely, the definitions of $Z_{m,n}$ and $G_{m,n}$ 
imply the equations 
\be  Z_{m,n}=e^{\beta\om(m,n)} (Z_{m-1,n}+Z_{m,n-1}) \label{Zeq1}\ee
and 
\be  G_{m,n}=\om(m,n) + \max(G_{m-1,n}\,,\, G_{m,n-1}).  \label{Geq1}\ee
These equations can be paraphrased by saying that 
$G_{m,n}$ obeys max-plus algebra, while $Z_{m,n}$ obeys 
the familiar algebra of addition and multiplication. 

This observation informs the  approach of the paper. It is not that we can  
  convert results for $G$  into
results for $Z$. Rather, after the proofs have been found,
one can detect a kinship with the arguments of 
\cite{bala-cato-sepp}, but transformed from $(\max, +)$ 
to $(+\,,\cdot\,)$.   The ideas
in \cite{bala-cato-sepp}  were originally adapted from  
the seminal paper \cite{cato-groe-06}.  The purpose was  to give an alternative
proof of  the scaling exponents of the corner growth model, without 
the asymptotic analysis of Fredholm determinants utilized in \cite{joha}. 

\medskip

 {\sl Frequently used notation.}   $\bN=\{1,2,3,\dotsc\}$ and $\bZ_+=\{0,1,2,\dotsc\}$.
Rectangles on the planar integer lattice are denoted by 
$\Lambda_{m,n}=\{0,\dotsc,m\}\times\{0,\dotsc,n\}$ and more
generally 
$\Lambda_{(k,\ell), (m,n)}=\{k,\dotsc,m\}\times\{\ell,\dotsc,n\}$.  
 $\bP$ is the probability distribution on the random environments
 or weights $\w$, 
 and  under $\bP$ the expectation of a random variable  $X$ is
 $\bE(X)$ and variance $\Vvv(X)$.
 Overline means centering:  $\overline{X}=X-\E X$.  
$Q^\w$ is the quenched polymer measure in a rectangle. The annealed measure
is $\Pann(\cdot)= \E Q^\w(\cdot)$ with expectation $\Eann(\cdot)$.  
$\mP$ is used for a generic probability measure that is not part
of the polymer model.  
Paths can  be written
$x_{k,\ell}=(x_k,x_{k+1},\dotsc, x_\ell)$ but also $x_\centerdot$ when 
$k,\ell$
are understood. Occasionally $A$ and $B$  denote gamma-distributed  random variables. 
The more usual random variable symbols  $X$, $Y,$ $Z$ and $W$ have 
specific  meanings in the polymer model.  
\medskip

{\sl Acknowledgements.}  The author thanks M\'arton Bal\'azs for 
pointing out that the gamma distribution solves the equations 
of Lemma \ref{UVY-lm},  Persi Diaconis for literature suggestions, and an 
anonymous referee for valuable suggestions. 

\medskip

{\sl Changes made to this version.}
This version differs from the published version (Ann.~Pro\-bab.~40 (2012) 19--73) in the following  respects:  two mistakes in  proofs have been fixed,   one theorem has been strengthened, and terminology has been altered to better agree with common usage. 

The published version   
has a mistake on lines 3--5 of page 52. Namely, the reversal mapping has to be applied in a fixed rectangle, but here the rectangle varies with $k$.     In this corrected version the required inequality is derived from the coupling given in Lemma \ref{coupl-lm}.   
The changes made to fix this mistake  are confined to the proof of Step 1 of the
proof of Proposition \ref{lb-prop1}.   

In the published version part of the proof of Lemma 5.4(ii) is missing.  The proof is complete for $1\le b\le CN^{2/3}$, but a separate argument is needed for $b\ge CN^{2/3}$.  In the present version the lemma has become Lemma \ref{lb-aux-lm3} and the proof is complete.  

In Theorem \ref{tot-f-thm} of the published version parameter $N_0$ depends on $b$.  This dependence has been lifted in the present version.

\section{The model and results}

We begin with the definition of the polymer model with boundaries and then state the results.  As stated in the introduction, 
relative to the standard description of the polymer model 
we turn
the picture 45 degrees clockwise so that the polymer lives in the nonnegative quadrant 
$\bZ_+^2$ of the planar lattice.  The inverse 
temperature parameter $\beta=1$ throughout.  
We replace the exponentiated weights 
  with multiplicative   weights $Y_{i,j}=e^{\om(i,j)}$, $(i,j)\in\bZ_+^2$.  
Then the partition function for paths whose endpoint is 
constrained to lie  at 
  $(m,n)$   is given by 
\be
Z_{m,n}=\sum_{x_\centerdot \in\Pi_{m,n}} \prod_{k=1}^{m+n} Y_{x_k}
\label{Z0.0}\ee  
where $\Pi_{m,n}$ denotes the collection of up-right paths $x_\centerdot=(x_k)_{0\le k\le m+n}$
  inside the rectangle $\Lambda_{m,n}=\{0,\dotsc,m\}\times\{0,\dotsc,n\}$ 
that go from  $(0,0)$   to $(m,n)$:    $x_0=(0,0)$, $x_{m+n}=(m,n)$ and 
 $x_k-x_{k-1}=(1,0)$ or $(0,1)$.   We adopt the convention that $Z_{m,n}$ does not
include  the weight at the origin, and if a value is needed then set  $Z_{0,0}=Y_{0,0}=1$. 
The symbol $\om$ will denote the entire random environment:
$\om=(Y_{i,j}:(i,j)\in\bZ_+^2)$.  When necessary the dependence of 
$Z_{m,n}$ on $\om$ will be expressed by $Z_{m,n}^\om$, with a similar
convention for other $\w$-dependent quantities.  
  
 We assign
distinct weight distributions  
on the boundaries $(\bN\times\{0\})\cup(\{0\}\times\bN)$ and in  the bulk $\bN^2$. 
  To highlight this   the 
  symbols $U$ and  $V$  will denote   weights on the horizontal
and vertical boundaries:  
\be  U_{i,0}=Y_{i,0}  \quad\text{and}\quad  V_{0,j}=Y_{0,j} \quad\text{for $i,j\in\bN$.} \label{UV1}\ee
However, in formulas such as \eqref{Z0.0}  it is obviously 
  convenient to use a single symbol $Y_{i,j}$ for all the  weights.  

Our results rest on the assumption  that the weights are reciprocals
 of gamma variables.  Let us  recall some basics. 
 The gamma function is $\Gamma(s)=\int_0^\infty x^{s-1}e^{-x}\,dx$. We 
 shall need it only for positive real $s$.   
The Gamma$(\theta, r)$ distribution  has density 
$\Gamma(\theta)^{-1}r^\theta x^{\theta-1}e^{-rx}$ on $\bR_+$,
mean $\theta/r$ and variance $\theta/r^2$. 
  
 The logarithm $\log \Gamma(s)$  
 is convex and  infinitely differentiable on $(0,\infty)$.
The derivatives are the polygamma functions
$\Psi_n(s)=(d^{n+1}/ds^{n+1}) \log \Gamma(s)$,
$n\in\bZ_+$ \cite[Section 6.4]{handbook}.   
For $n\ge 1$,  $\Psi_n$ 
is nonzero and has sign $(-1)^{n-1}$ throughout $(0,\infty)$
\cite[Thm.~7.71]{stro-ca}.  
Throughout the paper we make use
of  the digamma  and 
 trigamma functions  $\digamf$ and  $\trigamf$, on account of the 
connections 
\be
\digamf(\theta)=\E(\log A)\quad\text{and}\quad 
 \trigamf(\theta)=\Vvv(\log A) \quad
\text{for} \quad  A\sim \,\text{Gamma}(\theta, 1). 
\label{ditri1}\ee 
 
 Here is the assumption on the distributions.   Let $0<\theta<\mu<\infty$.
 \be \begin{aligned} &\text{Weights } \ 
\{U_{i,0}, V_{0,j}, Y_{i,j}:  i,j\in\bN\} \ \text{ are   independent with
distributions} \\
&\text{$U_{i,0}^{-1}$ $\sim$ Gamma$(\theta, 1)$,  $V_{0,j}^{-1}$ $\sim$ Gamma$(\mu-\theta, 1)$,   and 
$Y_{i,j}^{-1}$ $\sim$ Gamma$(\mu, 1)$.} 
\end{aligned}\label{distr4}\ee
We fixed the scale parameter $r=1$ in the gamma distributions 
above for the sake of convenience.   We could equally well fix
it to any value and our results would not change,  as long as   all three gamma distributions
above have the same scale parameter.   

A key technical result will be that under \eqref{distr4}
each  ratio  $U_{m,n}=Z_{m,n}/Z_{m-1,n}$
and $V_{m,n}=Z_{m,n}/Z_{m,n-1}$ has the same marginal distribution as $U$ and $V$ in \eqref{distr4}.
This is a Burke's Theorem of sorts, and appears as Theorem \ref{burkethm} below.  
From this we can compute the mean exactly: for $m,n\ge 0$, 
\be \E\bigl[\log Z_{m,n}\bigr] = m\E(\log U)+n\E(\log V)= -m\digamf(\theta)-n\digamf(\mu-\theta).  
\label{ElogZ}\ee

Together with the choice of the parameters $\theta, \mu$ goes a choice of 
``characteristic direction''  $( \trigamf(\mu-\theta),  \trigamf(\theta))$     for 
the   polymer.   Let $N$ denote the scaling   parameter  
we take to $\infty$.  We assume that the coordinates $(m,n)$  of the endpoint 
of the polymer satisfy 
\be \abs{\, m-{N\trigamf(\mu-\theta)}\, }\le  \gamma N^{2/3} 
\quad\text{and}\quad 
 \abs{\, n- {N\trigamf(\theta)}\,}
\le \gamma N^{2/3}  \label{mnN}\ee
for some fixed constant $\gamma$. 
Now we can state the variance bounds for the free energy. 

\begin{theorem} Assume \eqref{distr4} and let $(m,n)$ be as in \eqref{mnN}. Then
there exist constants $0<C_1,C_2<\infty$ such that, for $N\ge 1$, 
\[  C_1N^{2/3} \le \Vvv(\log Z_{m,n}) \le C_2N^{2/3} . \]
\label{var-bd-thm}\end{theorem}

The constants $C_1,C_2$  in the theorem depend on $0<\theta<\mu$
and on $\gamma$ of \eqref{mnN},  and they can be taken
the same for $(\theta,\mu, \gamma)$ that vary in a compact set.  This holds for  all the
constants in the theorems of this section: they depend on the parameters
of the assumptions, but for parameter values in a compact set the constants 
themselves can be fixed. 

The upper bound on the variance is good enough for Borel-Cantelli to give
the strong law of large numbers:   with $(m,n)$ as in \eqref{mnN}, 
\be   \lim_{N\to\infty} N^{-1}\log Z_{m,n}  = 
-\digamf(\theta)\trigamf(\mu-\theta)-\digamf(\mu-\theta)\trigamf(\theta)  
\quad \text{$\P$-a.s.}  
\label{Zllnbd}\ee

As a further corollary we deduce that if the direction of the polymer deviates   from the  
characteristic one by a larger power of $N$  than allowed by \eqref{mnN},
 then $\log Z$ satisfies a 
  central  limit theorem.  For the sake of concreteness we treat the case where
the  horizontal direction is too large. 

\begin{corollary} Assume \eqref{distr4}.  Suppose $m,n\to\infty$.  Define parameter 
$N$ by $n=\trigamf(\theta)N$, and assume that 
\[  N^{-\alpha}\bigl(m-\trigamf(\mu-\theta)N\bigr) \to c_1 >0
\quad\text{as $N\to\infty$}   
\]
for some $\alpha>2/3$.  Then as $N\to\infty$, 
\[ N^{-\alpha/2} \Bigl\{ \log Z_{m,n} 
-\E \bigl(\log Z_{m,n}\bigr) \Bigr\}\] converges in distribution
to a centered normal distribution  with variance 
$c_1\trigamf(\theta)$. 
\label{clt-cor}\end{corollary}

The quenched polymer measure $Q^\om_{m,n}$ is defined on paths 
$x_\centerdot\in\Pi_{m,n}$ by
\be  Q^\om_{m,n}(x_\centerdot)=  \frac1{Z_{m,n}}\prod_{k=1}^{m+n} Y_{x_k}  \label{Q}\ee
remembering convention \eqref{UV1}.  Integrating out the random 
environment $\w$ gives the annealed measure 
\[ \Pann_{m,n}(x_\centerdot) = \int Q^\w_{m,n}(x_\centerdot)\,\bP(d\w). \]
When the rectangle $\Lambda_{m,n}$  is understood, we drop the 
subscripts and write $\Pann=\E Q^\w$.  Notation will be further simplified
by writing $Q$ for $Q^\w$. 

We describe the fluctuations of the path $x_\centerdot$ under $\Pann$.
The next result shows that $N^{2/3}$ is the correct order
of magnitude of  the fluctuations of the path. 
Let $\vva(j)$ and $\vvb(j)$ denote the left- and rightmost points of the path
on the horizontal  line with ordinate $j$: 
 \be
\vva(j)=\min\{i\in\{0,\dotsc,m\}:  \exists k \ \text{such that} \  x_k=(i,j) \} \label{defv}\ee
and 
\be \vvb(j)=\max\{i\in\{0,\dotsc,m\}:  \exists k \ \text{such that} \    x_k=(i,j) \}. \label{defv1}\ee

\begin{theorem}  Assume \eqref{distr4} and let $(m,n)$ be as in \eqref{mnN}.
 Let $0\le \tau<1$.  
 Then there exist constants $C_1, C_2<\infty$ such that 
  for $N\ge 1$  and $b\ge C_1$,  
\be \Pann\bigl\{ \vva(\fl{\tau n})<\tau m-bN^{2/3}
\ \ \text{or}\ \  \vvb(\fl{\tau n})>\tau m+bN^{2/3}\bigr\}
\le C_2b^{-3}. 
\label{path-bd-1}\ee
Same bound holds for the  vertical counterparts of $\vva$ and $\vvb$. 

Let $0<\tau<1$.  Then given $\e>0$, there exists $\delta>0$ such that 
\be  \varlimsup_{N\to\infty} \Pann\{ 
\text{ $\exists k$ such that $\abs{\,x_k-(\tau m, \tau n)}
\le \delta N^{2/3}$ } \} 
\le \e. 
\label{path-bd-2}\ee
\label{path-bd-thm}\end{theorem}

Presently we do not have sharp quenched results.  
From  Lemma \ref{auxlm7} and the proof of Theorem \ref{path-bd-thm}
in Section \ref{sec:path-bd}  one can extract  estimates
on the $\bP$-tails of the quenched probabilities of the events
in \eqref{path-bd-1} and \eqref{path-bd-2}. 

\bigskip

We turn to results for the model without boundaries, by 
restricting ourselves to the positive quadrant $\bN^2$.  Define
the  partition function of a general rectangle 
$\{k,\dotsc,m\}\times\{\ell,\dotsc,n\}$ by 
\be
Z_{(k,\ell),(m,n)}=\sum_{x_\centerdot \in\Pi_{(k,\ell),(m,n)}} \prod_{i=1}^{m-k+n-\ell} Y_{x_i}
\label{Z3}\ee
where    $\Pi_{(k,\ell),(m,n)}$ is the collection
of up-right paths $x_\centerdot=(x_i)_{i=0}^{m-k+n-\ell } $ from $x_0=(k,\ell)$ to
$x_{m-k+n-\ell }=(m,n)$. Admissible steps are always 
$x_{i+1}-x_i=e_1=(1,0)$ or  $x_{i+1}-x_i=e_2=(0,1)$. 
We have chosen not to include the weight of the southwest corner $(k,\ell)$.  
The earlier definition \eqref{Z0.0} is the special case  
 $Z_{m,n}=Z_{(0,0),(m,n)}$.  Also we stipulate
that  when the rectangle reduces to a point,  $Z_{(k,\ell),(k,\ell)}=1$.

In particular, $Z_{(1,1),(m,n)}$ gives us partition functions that only involve
the bulk weights $\{Y_{i,j}: i,j\in\bN\}$.  The assumption on their distribution is as
before, with a fixed parameter $0<\mu<\infty$: 
\be
\text{ $\{Y_{i,j}: i,j\in\bN\}$ are i.i.d.\ with common distribution 
$Y_{i,j}^{-1}$ $\sim$ Gamma$(\mu, 1)$.} \label{distr4.1}\ee

We define the limiting free energy.  The identity
(see e.g. \cite[(2.11)]{arti-gamma} or \cite[Section 6.4]{handbook})
\[  \trigamf(x)=\sum_{k=0}^\infty \frac1{(x+k)^2} \]
shows that $\trigamf(0+)=\infty$.  Thus, given $0<s,t<\infty$,  there is a unique
$\theta=\theta_{s,t}\in(0,\mu)$ such that 
\be \frac{\trigamf(\mu-\theta)}{\trigamf(\theta)}=\frac{s}t. \label{trig-st}\ee
Define 
\be \freee_{s,t}(\mu)= -\,\bigl(  s \digamf(\theta_{s,t}) + t \digamf(\mu-\theta_{s,t}) \bigr).
\label{def-free}\ee
It can be verified that for a fixed $0<\mu<\infty$,  $\freee_{s,t}(\mu)$ is a 
continuous function of $(s,t)\in\bR_+^2$ with boundary values  
\[  \freee_{0,t}(\mu)=\freee_{t,0}(\mu)=-t\digamf(\mu). \]
Here is the result for the free energy of the polymer without boundary but still
with fixed endpoint.  
The constants in this theorem  depend on $(s,t,\mu)$.  

\begin{theorem}  Assume \eqref{distr4.1} and let $0<s,t<\infty$.  
We have the law of large numbers
\be
\lim_{N\to\infty} N^{-1}\log Z_{(1,1),(\fl{Ns},\fl{Nt})}  =\freee_{s,t}(\mu)
\quad \text{$\P$-a.s.}  
\label{Zlln}\ee
 There exist finite constants $N_0 $ and $C_0 $ such 
that, for  $b\ge 1$ and $N\ge N_0$, 
\be 
\P\bigl[  \; \abs{ \log Z_{(1,1),(\fl{Ns},\fl{Nt})} - N\freee_{s,t}(\mu) }\ge bN^{1/3} \,\bigr] 
\le C_0 b^{-3/2}. \label{Zmom2.9}\ee
\label{Zno-bd-thm}\end{theorem} 

In particular, we get the moment bound 
\be 
\E\biggl\{ \, \biggl\lvert \frac{\log Z_{(1,1),(\fl{Ns},\fl{Nt})} - N\freee_{s,t}(\mu) }
{N^{1/3}} \biggr\rvert^p \biggr\} \le C(s,t,\mu,p) <\infty \label{Zmom3}\ee
for $N\ge N_0(s,t,\mu)$ and   $1\le p<3/2$.   
The theorem is proved by relating $Z_{(1,1),(\fl{Ns},\fl{Nt})}$ to a polymer
with a boundary. Equation \eqref{trig-st} picks the correct 
boundary parameter $\theta$.  Presently we do not have a matching
lower bound for \eqref{Zmom2.9}. 

In a general rectangle the quenched polymer distribution of a path
$x_\centerdot \in\Pi_{(k,\ell),(m,n)}$ is 
\be
Q_{(k,\ell),(m,n)}(x_\centerdot)
=\frac1{Z_{(k,\ell),(m,n)}}{ \prod_{i=1}^{m-k+n-\ell} Y_{x_i}}.  \label{Q3}\ee
As before the annealed distribution is 
$\Pann_{(k,\ell),(m,n)}(\cdot)=\E Q_{(k,\ell),(m,n)}(\cdot)$.  The
upper  fluctuation
bounds for the path in the model  with boundaries can be 
extended to the model without boundaries. 
Here we can again allow  the endpoint  $(m,n)$ to deviate from the characteristic direction:
 \be \abs{ m-Ns }\le  \gamma N^{2/3} 
\quad\text{and}\quad 
 \abs{n- Nt}
\le \gamma N^{2/3}  \label{mnNgamma}\ee
for a constant $\gamma$. 
The constants in this theorem  depend on $(s,t,\mu, \gamma)$.  

\begin{theorem}   Assume \eqref{distr4.1}, fix $0<s,t<\infty$, and assume 
\eqref{mnNgamma}.
 Let $0\le \tau<1$.  
 Then there exist finite constants  
$C$, $C_0$ and $N_0$    such that  
  for $N\ge N_0$  and $b\ge C_0$,  
\be\begin{aligned}
 \Pann_{(1,1),(m,n)}\bigl\{ &
\vva(\fl{\tau n})<\tau m-bN^{2/3} \\
 &\qquad \qquad \text{or} \ \ 
 \vvb(\fl{\tau n})>\tau m+bN^{2/3}\bigr\} 
 \le Cb^{-3}. 
\end{aligned}\label{path-nobd-1}\ee
Same bound holds for the  vertical counterparts of $\vva$ and $\vvb$. 
\label{path-nobd-thm}\end{theorem}

Next we drop the restriction on the endpoint, and extend the upper bounds
to the point-to-line 
 polymer with unrestricted endpoint and no boundaries. Given the value
 of the parameter $N\in\bN$, the set of admissible paths   is 
 $\cup_{1\le k\le N-1} \Pi_{(1,1),(k,N-k)}$, namely the set of all up-right paths  
 $x_\centerdot=(x_k)_{0\le k\le N-2}$  that start at $x_0=(1,1)$ and 
 whose endpoint $x_{N-2}$ lies  on the line $x+y=N$. 
The quenched polymer probability of such a path is 
\[  Q_N^{\text{\rm p2l}}(x_\centerdot)= \frac1{Z_N^{\text{\rm p2l}}}\prod_{k=1}^{N-2}Y_{x_k} \]
with   the  
partition function  (superscript p2l stands for point-to-line)  
\[   Z_N^{\text{\rm p2l}} =\sum_{k=1}^{N-1} Z_{(1,1),(k, N-k)}.\]
The annealed measure is  $P_N^{\text{\rm p2l}}(\cdot)=\E Q_N^{\text{\rm p2l}}(\cdot)$. 
We collect all the results in one theorem, proved in Section \ref{sec:tot}. 
In particular,  \eqref{fa-5}  below shows that the fluctuations 
of the endpoint of the path are of order at most $N^{2/3}$.   
Statement \eqref{fa-3} in the proof gives bounds on the quenched probability 
of a deviation.  
 
\begin{theorem}
Fix $0<\mu<\infty$ and assume weight distribution \eqref{distr4.1}. 
We have the law of large numbers
\be
\lim_{N\to\infty} N^{-1}\log Z_N^{\text{\rm p2l}}
=\freee_{1/2,1/2}(\mu)=-\digamf(\mu/2)
\quad \text{$\P$-a.s.}  
\label{Ztotlln}\ee
There exist finite constants  
  $C(\mu)$ and $N_0(\mu)$  that depend on $\mu$ alone such that, for  
   $b\ge 1$,  
\be 
\sup_{N\ge N_0(\mu)} \P\bigl[  \; \abs{ \log Z_N^{\text{\rm p2l}}  - N\freee_{1/2,1/2}(\mu) }\ge bN^{1/3} \,\bigr] 
\le C(\mu)b^{-3/2} \label{fa-4}\ee
and 
\be \sup_{N\ge N_0(\mu)} P_N^{\text{\rm p2l}}\bigl\{ \,\bigl\lvert 
{x_{N-2}-(\tfrac{N}2, \tfrac{N}2)} \bigr\rvert \ge bN^{2/3}\,\bigr\}
  \le \; C(\mu)b^{-3}.  
\label{fa-5}\ee
\label{tot-f-thm}\end{theorem} 

The last case to address is the point-to-line polymer with boundaries.  
This case   is perhaps  of less interest than the others for the free energy scales diffusively, 
but we record it for the sake of completeness.  
Fix $0<\theta<\mu$ and  let assumption \eqref{distr4} on the weight distributions be in
force. The fixed-endpoint 
 partition function $Z_{m,n}=Z_{(0,0),(m,n)}$ is the one defined in \eqref{Z0.0}. 
Define  the   partition function of all paths from $(0,0)$ to the line 
$x+y=N$ by 
\[   Z_N^{\text{\rm p2l}}(\theta,\mu)= \sum_{\ell=0}^N Z_{\ell,N-\ell}.   \]
Define a  limiting free energy 
\[  g(\theta,\mu)= \max_{0\le s\le 1}\bigl(-s\digamf(\theta)-(1-s)\digamf(\mu-\theta)\bigr)
=\begin{cases}  -\digamf(\theta) &\theta\le \mu/2\\   -\digamf(\mu-\theta) &\theta\ge \mu/2.
\end{cases}  
\]
Set also
\[  \sigma^2(\theta,\mu)=\begin{cases}  \trigamf(\theta) &\theta\le \mu/2\\  
 \trigamf(\mu-\theta) &\theta\ge \mu/2, 
\end{cases}  
\]
and define random variables $\zeta(\theta,\mu)$ as follows:  
  for  $\theta\ne\mu/2$,  $\zeta(\theta,\mu)$ has  centered normal distribution with 
variance $\sigma^2(\theta,\mu)$, while 
\be \zeta(\mu/2,\mu)=\sqrt{2\trigamf(\mu/2)}\bigl(M_{1/2}\vee M_{1/2}'\bigr)
\label{defzeta2}\ee where 
$M_t=\sup_{0\le s\le t}B(s)$ is the running maximum of a standard Brownian motion 
and $M_t'$ is an independent copy of it.  

\begin{theorem}  Let $0<\theta<\mu$ and  assume \eqref{distr4}.  We have the law of large numbers 
\be     \lim_{N\to\infty} N^{-1}\log    Z_N^{\text{\rm p2l}}(\theta,\mu)= g(\theta,\mu)
\quad \text{$\P$-a.s.}    \label{llnZtotbd}\ee
and  the distributional limit 
\be
N^{-1/2}\bigl(  \log Z_N^{\text{\rm p2l}}(\theta,\mu) - Ng(\mu/2,\mu)  \bigr) 
\overset{d}\longrightarrow \zeta(\theta,\mu).  
\label{Ztotbdclt}\ee
\label{thm-Ztotbd}\end{theorem} 
When  $\theta\ne\mu/2$ the axis with the larger $-\digamf$ value 
completely dominates, while if   $\theta=\mu/2$ all
directions have the same limiting free energy.  This accounts for the results
in the theorem. 

\medskip

{\sl Organization of the paper.} 
Before we   begin  the proofs of the main theorems, 
 Section \ref{sec:basic} collects   basic
properties of the model, including the Burke-type property.  
 The upper and lower bounds of 
Theorem \ref{var-bd-thm} are proved in Sections 
\ref{sec:ub-bd} and \ref{sec:lb-bd}.  
Corollary \ref{clt-cor} is proved at the end of Section \ref{sec:ub-bd}.
The bounds for the fixed-endpoint path with boundaries 
are proved in Section \ref{sec:path-bd},  and   
the results for the fixed-endpoint polymer model without boundaries   
in Section \ref{sec:nobd}.  The results for the polymer with free endpoint 
are proved   in Section \ref{sec:tot}. 
 
\section{Basic properties of the polymer model with boundaries} 
\label{sec:basic}
This section sets the stage for the proofs  with some preliminaries.  
The main results of this section
 are the Burke property in Theorem \ref{burkethm} and  
identities that tie together the variance of the free energy and the exit
points from the axes in Theorem \ref{varxi-thm}. 
 
Occasionally we will use notation for the partition function that includes
the weight at the starting point, which we write as
\be
\Zalt_{(i,j),(k,\ell)}=\sum_{x_\centerdot \in\Pi_{(i,j),(k,\ell)}} 
\prod_{r=0}^{k-i+\ell-j} Y_{x_r} = Y_{i,j} Z_{(i,j),(k,\ell)}. 
\label{Z3.4}\ee  
  
Let the initial weights $\{U_{i,0}, V_{0,j}, Y_{i,j}:  i,j\in\bN\}$ be
given.   
  Starting from the lower left corner of $\bN^2$,
define inductively for $(i,j)\in\bN^2$
\be\begin{aligned} 
&U_{i,j}=Y_{i,j}\Bigl(1+\frac{U_{i,j-1}}{V_{i-1,j}}\Bigr), \quad 
V_{i,j}=Y_{i,j}\Bigl(1+\frac{V_{i-1,j}}{U_{i,j-1}}\Bigr) \\
&\qquad\qquad  \quad \text{and}\quad X_{i-1,j-1}=\Bigl(\frac1{U_{i,j-1}}+\frac1{V_{i-1,j}}\Bigr)^{-1}.  
\end{aligned}\label{UVY1}\ee
 The  partition function satisfies 
\be 
Z_{m,n}=Y_{m,n}\bigl( Z_{m-1,n} + Z_{m,n-1} \bigr)  \quad 
\text{for} \  (m,n)\in\bN^2
\label{Z1}\ee
and one  checks inductively that 
\be 
U_{m,n}=\frac{Z_{m,n}}{Z_{m-1,n}}
\quad\text{and}\quad 
V_{m,n}=\frac{Z_{m,n}}{Z_{m,n-1}}  \label{Z2}\ee
for $(m,n)\in \bZ_+^2\smallsetminus\{(0,0)\}$.
Equations \eqref{Z1} and \eqref{Z2} are also valid for $\Zalt_{m,n}$
because
the weight at the origin  cancels from the equations.  

It is also natural to associate the $U$- and $V$-variables to undirected edges of the lattice 
$\bZ_+^2$.  If $f=\{x-e_1,x\}$ is a horizontal edge then $T_f=U_x$, while if 
$f=\{x-e_2, x\}$ then $T_f=V_x$.  

The following  monotonicity property can be proved inductively:

\begin{lemma}  Consider two sets of positive  initial values 
$\{U_{i,0}, V_{0,j}, Y_{i,j}:  i,j\in\bN\}$ and $\{\wt U_{i,0}, \wt V_{0,j}, \wt Y_{i,j}:  i,j\in\bN\}$
that satisfy  $U_{i,0}\ge\wt U_{i,0}$, $V_{0,j} \le \wt V_{0,j}$, and  $Y_{i,j} =\wt Y_{i,j}$. 
From these define inductively the values $\{U_{i,j},V_{i,j}:(i,j)\in\bN^2\}$  and 
$\{\wt U_{i,j},\wt V_{i,j}:(i,j)\in\bN^2\}$ by equation \eqref{UVY1}.   Then 
$U_{i,j}\ge\wt U_{i,j}$ and $V_{i,j} \le \wt V_{i,j}$ for all   $(i,j)\in\bN^2$.  
\label{monolm}\end{lemma}

\subsection{Propagation of boundary conditions}
The next lemma gives a reversibility property that we can regard as an analogue 
of reversibility properties of  M/M/1 queues and their  last-passage 
versions.  (A basic reference for queues is \cite{kell}.
Related work
appears  in  \cite{bala-cato-sepp, cato-groe-05, cato-groe-06, oconn-yor-01}.)   

\begin{lemma} 
Let $U$, $V$ and $Y$ be independent positive random variables.  Define 
\be
U'=Y(1+UV^{-1}), \quad V'=Y(1+VU^{-1}) \quad \text{and}\quad Y'=(U^{-1}+V^{-1})^{-1}.  
\label{UVY}\ee
Then the triple $(U',V',Y')$ has the same distribution as $(U,V,Y)$ iff  there exist positive 
parameters $0<\theta<\mu$ and $r$ such that
 \be  
\text{$U^{-1}$ $\sim$ Gamma$(\theta, r)$,  $V^{-1}$ $\sim$ Gamma$(\mu-\theta, r)$,   and 
$Y^{-1}$ $\sim$ Gamma$(\mu, r)$. } 
\label{distr}\ee
\label{UVY-lm}\end{lemma} 
\begin{proof}  Assuming  \eqref{distr}, define independent gamma variables 
 $A=U^{-1}$, $B=V^{-1}$ and $Z=Y^{-1}$. Then set 
\[ A'=\frac{ZA}{A+B}\,, \quad  B'=\frac{ZB}{A+B}\,, \quad  \text{and}\quad  Z'=A+B. \]
 We need to show that
$(A',B',Z')\overset{d}=(A,B,Z)$.  Direct calculation with Laplace transforms 
is convenient.  Alternatively,  one can reason with basic properties
 of gamma
distributions  as follows.  The pair $(A/(A+B), B/(A+B))$ has distributions
 Beta$(\theta,\mu-\theta)$ and Beta$(\mu-\theta, \theta)$,  
  and is independent of the Gamma$(\mu,r)$-distributed sum
  $A+B=Z'$.  Hence $A'$ and $B'$ are a pair
of independent  variables with distributions Gamma$(\theta, r)$ 
and Gamma$(\mu-\theta,r)$, and by construction also independent of $Z'$.  

Assuming $(A',B',Z')\overset{d}=(A,B,Z)$,  
  $A'/B'=A/B$ is independent of $Z'=A+B$.  By 
Theorem 1 of \cite{luka-55}  $A$ and $B$ are independent gamma variables with the
same scale parameter $r$.  Then  $Z\overset{d}=Z'=A+B$ determines  the 
distribution of $Z$. 
\end{proof} 

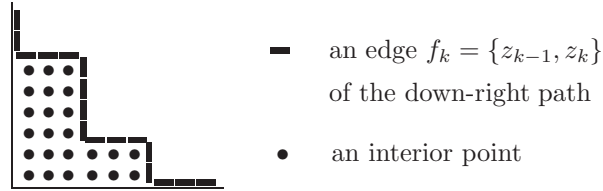
\begin{figure}[ht2]
\begin{center}
\begin{picture}(200,90)(60,20)
\put(40,20){\line(1,0){80}}
\put(40,20){\line(0,1){70}}
{\linethickness{2pt} 
\multiput(42,79)(0,8){2}{\line(0,-1){7}}   \multiput(42,70)(8,0){3}{\line(1,0){7}}
\multiput(67,45)(0,8){4}{\line(0,-1){7}}   \multiput(68,38)(8,0){3}{\line(1,0){7}}
\multiput(92,29)(0,8){2}{\line(0,-1){7}}   \multiput(94,22)(8,0){3}{\line(1,0){7}}
}
{
\multiput(44,22)(0,8){6}{\small $\bullet$}  \multiput(51.5,22)(0,8){6}{\small $\bullet$}
\multiput(59,22)(0,8){6}{\small $\bullet$}
\multiput(68,22)(0,8){2}{\small $\bullet$}  \multiput(76,22)(0,8){2}{\small $\bullet$}
 \multiput(84,22)(0,8){2}{\small $\bullet$}}
 {\linethickness{2pt}
 \multiput(138,72)(8,0){1}{\line(1,0){7}} }  
 \put(160,69){\small an edge $f_k=\{ z_{k-1}, z_k\}$ }
\put(160,53){\small  of the down-right path }
{ 
\put(140,30){\small $\bullet$} }
 \put(161,30){\small an interior point  }
\end{picture}
\end{center}  
\caption{ \small Illustration of a down-right path $(z_k)$ and its
set $\cI$ of  interior points. Interior point $(i,j)$ is represented by a dot
centered at $(i+1/2,\,j+1/2)$.}  \label{fig4}
\end{figure}

\bigskip

From this lemma we get a Burke-type theorem.  Let $z_\centerdot=(z_k)_{k\in\bZ}$ be a nearest-neighbor
down-right path in $\bZ_+^2$, that is,   $z_k\in\bZ_+^2$ and 
$z_k-z_{k-1}=e_1$ or $-e_2$.  Denote the undirected edges of the path by $f_k=\{z_{k-1},z_k\}$,
and let 
\[   T_{f_k}= \begin{cases}  U_{z_k}&\text{if $f_k$ is a horizontal edge}\\
V_{z_{k-1}}&\text{if $f_k$ is a vertical edge.}
\end{cases}  \]
Let the (lower left) {\sl interior} of the path be the vertex set 
$\cI=\{ (i,j)\in\bZ_+^2:  \exists m\in\bN:  \; (i+m,j+m)\in\{z_k\} \}$
(see Figure \ref{fig4}).  
$\cI$ is finite if the path $z_\centerdot$ coincides with the axes
 for all but finitely many edges.  Recall the definition of $X_{i,j}$
from \eqref{UVY1}. 

\begin{theorem}  Assume \eqref{distr4}.
For any down-right path $(z_k)_{k\in\bZ}$  in $\bZ_+^2$,  the  variables 
$\{ T_{f_k},  X_z:  k\in\bZ,\, z\in\cI\}$ 
  are mutually independent with marginal distributions 
\be  
\text{$U^{-1}$ $\sim$ Gamma$(\theta, 1)$,  $V^{-1}$ $\sim$ Gamma$(\mu-\theta, 1)$,   and 
$X^{-1}$ $\sim$ Gamma$(\mu, 1)$.} 
\label{distr2}\ee
\label{burkethm}\end{theorem} 

\begin{proof}  This is proved first by induction  for down-right paths with finite interior  $\cI$. 
If $z_\centerdot$ coincides with the $x$- and $y$-axes then $\cI$ is empty, and the statement
follows from  assumption \eqref{distr4}.   The inductive step consists of adding
a ``growth corner'' to $\cI$ and an application of Lemma \ref{UVY-lm}.  Namely, 
suppose $z_\centerdot$ goes through the three points 
$(i-1,j)$, $(i-1,j-1)$ and $(i,j-1)$.  Flip the corner over to create a
 new path $z'_\centerdot$ that goes through $(i-1,j)$, $(i,j)$ and $(i,j-1)$. The new
 interior is $\cI'=\cI\cup\{(i-1,j-1)\}$.  Apply Lemma \ref{UVY-lm} with 
\[ U=U_{i,j-1}, \ V=V_{i-1,j}, \ Y=Y_{i,j},\ U'=U_{i,j}, \ V'=V_{i,j}, \  \text{ and } \ 
Y'=X_{i-1,j-1} \] 
to see that the conclusion continues to hold for  $z'_\centerdot$ and  $\cI'$.

To prove the theorem for an arbitrary down-right path it suffices to consider a finite
portion of $z_\centerdot$ and  $\cI$ inside some large square $B=\{0,\dotsc,M\}^2$.  
Apply the first part of the proof  to the modified path that coincides with  $z_\centerdot$ inside $B$
but otherwise follows  the coordinate axes and connects up with $z_\centerdot$
on   the north and east boundaries of $B$. 
\end{proof}

To understand the sense in which Theorem \ref{burkethm}
 is a ``Burke property'',
note its similarity  with Lemma 4.2 in \cite{bala-cato-sepp} whose
connection with M/M/1 queues in series 
 is immediate through the last-passage
representation. 


\subsection{Reversal}

In a  fixed
rectangle $\Lambda=\{0,\dotsc, m\}\times\{0,\dotsc, n\}$ define 
the reversed  partition function 
\be  Z^*_{i,j}= \frac{Z_{m,n}}{Z_{m-i,n-j}} 
\quad \text{for } \  (i,j)\in\Lambda.  \label{Z*}\ee
Note that for the partition function of the entire rectangle, 
\[   Z^*_{m,n}=Z_{m,n}. \]
Recalling \eqref{UVY1}  make these further definitions:
\be \begin{aligned} 
U^*_{i,j}&=U_{m-i+1,n-j}\quad   &\text{for } \  (i,j)\in\{1,\dotsc, m\}\times\{0,\dotsc, n\}, \\
V^*_{i,j}&=V_{m-i, n-j+1} \quad   &\text{for } \  (i,j)\in\{0,\dotsc, m\}\times\{1,\dotsc, n\}, \\
Y^*_{i,j}&= X_{m-i,n-j} \quad   &\text{for } \  (i,j)\in\{1,\dotsc, m\}\times\{1,\dotsc, n\}.
\end{aligned}\label{U*}\ee

The mapping $*$ is an involution, that is, inside the rectangle
$\Lambda$, $Z^{**}_{i,j}=Z_{i,j}$ and similarly for 
$U$, $V$ and $Y$.  

\begin{proposition} Assume distributions \eqref{distr4}.    Then inside the rectangle 
$\Lambda$ the system $\{ Z^*_{i,j}, U^*_{i,j}, V^*_{i,j}, Y^*_{i,j}\}$ replicates
the properties of the  original system 
 $\{ Z_{i,j}, U_{i,j}, V_{i,j}, Y_{i,j}\}$.  Namely, we have these facts:

{\rm (a)}  $\{U^*_{i,0}, V^*_{0,j}, Y^*_{i,j}:  1\le i\le m, \, 1\le j\le n\}$ are independent 
with distributions 
\be \begin{aligned}
&\text{$(U^*_{i,0})^{-1}$ $\sim$ Gamma$(\theta, 1)$,  $(V^*_{0,j})^{-1}$ $\sim$ Gamma$(\mu-\theta, 1)$, } \\
&\qquad \text{  and 
$(Y^*_{i,j})^{-1}$ $\sim$ Gamma$(\mu, 1)$.}  \end{aligned} \label{distr7}\ee

{\rm (b)} These identities hold:  $Z^*_{0,0}=1$,  
$Z^*_{i,j}=Y^*_{i,j}\bigl( Z^*_{i-1,j} + Z^*_{i,j-1} \bigr)$,    
\begin{align*} 
&U^*_{i,j}=\frac{Z^*_{i,j}}{Z^*_{i-1,j}}\,,  
\quad V^*_{i,j}=\frac{Z^*_{i,j}}{Z^*_{i,j-1}} \,, \\
&U^*_{i,j}=Y^*_{i,j}\Bigl(1+\frac{U^*_{i,j-1}}{V^*_{i-1,j}}\Bigr), \quad \text{and }\ 
V^*_{i,j}=Y^*_{i,j}\Bigl(1+\frac{V^*_{i-1,j}}{U^*_{i,j-1}}\Bigr).  
 \end{align*} 
\label{pr-rev} \end{proposition} 
 
 \begin{proof}  Part (a) is a consequence of  Theorem \ref{burkethm}.   Part (b) follows
 from definitions \eqref{Z*} and \eqref{U*} of the reverse variables 
 and properties \eqref{UVY1},  \eqref{Z1} and 
 \eqref{Z2} of the original system. 
 \end{proof}

 Define a dual measure on paths $x_{0,m+n}\in\Pi_{m,n}$ by 
 \be  Q^{*,\w} (x_{0,m+n}) = \frac1{Z_{m,n}}
\prod_{k=0}^{m+n-1} X_{x_k} \label{Q*1}\ee
with the conventions $X_{i,n}=U_{i+1,n}$ for $0\le i<m$ and 
 $X_{m,j}=V_{m,j+1}$ for $0\le j<n$.  This convention is needed
because inside the fixed rectangle $\Lambda$,  \eqref{UVY1} defines
the $X$-weights only away from the north and east boundaries. 
The boundary weights are of the $U$- and $V$-type. 

  Define a reversed
 environment $\w^*$ as a function of $\w$ in $\Lambda$ by  
\[\w^*=(U^*_{i,0}, V^*_{0,j}, Y^*_{i,j}:
 (i,j)\in\{1,\dotsc,m\}\times\{1,\dotsc, n\} ). \]
Part (a) of Proposition \ref{pr-rev} says that $\w^*\overset{d}=\w$. 
  As before, utilize also the definitions
$Y^*_{i,0}=U^*_{i,0}$ and $Y^*_{0,j}=V^*_{0,j}$.
Write \[x^*_k=(m,n)-x_{m+n-k}\] for the reversed path.
For an event $A\subseteq\Pi_{m,n}$ on paths let $A^*=\{x_{0,m+n}: x^*_{0,m+n}\in A\}$.

\begin{lemma}   $Q^{*,\w}(A)$  and $Q^\w(A^*)$ 
 have the same distribution under $\P$. 
\label{lemQ*1}\end{lemma}

\begin{proof} By the definitions, 
\be\begin{aligned} 
Q^{*,\w}(A) = \frac1{Z_{m,n}} \sum_{x_{0,m+n}\in A} \prod_{k=0}^{m+n-1} X_{x_k}
=\frac1{Z^*_{m,n}}\sum_{x_{0,m+n}\in A} \prod_{j=1}^{m+n} Y^*_{x^*_j}
= Q^{\w^*}(A^*). 
\end{aligned}\label{Q*2}\ee
By   Proposition \ref{pr-rev},  $Q^{\w^*}(A^*) \overset{d}= Q^{\w}(A^*)$. 
\end{proof}

\begin{remark} $Q^{*,\w}(A)$  and $Q^\w(A)$ do not in general have
the same distribution because their boundary weights are 
different. 
\end{remark}

Under the dual measure the path $x_{0,m+n}$ is a Markov chain. This can be
seen by rewriting \eqref{Q*1} as 
 \be\begin{aligned} 
Q^{*,\w}(x_{0,m+n})=  \prod_{k=0}^{m+n-1} 
\frac{X_{x_k}Z_{x_k}}{Z_{x_{k+1}}} 
= \prod_{k=0}^{m+n-1} \pi^*_{x_k, x_{k+1}}
\end{aligned}\label{dualmarkov}\ee
where the last equality defines the Markov kernel $\pi^*_{x,y}$ on the state space
$\Lambda$. 
At points $x$ away from the north and east boundaries
we can write the  kernel as 
\be 
\pi^*_{x,x+e}= \frac{X_{x}Z_{x}}{Z_{x+e}}
=  \frac{Z_{x+e}^{-1}}{Z_{x+e_1}^{-1}
+ Z_{x+e_2}^{-1}}\,, \qquad e\in\{e_1,e_2\}. \label{dualpi}\ee
On the  north and east boundaries
(that is,  either $x=(i,n)$ for some $0\le i<m$ or  $x=(m,j)$ for some $0\le j<n$)  
the  kernel is degenerate because there is only one admissible step.

\subsection{Variance and exit point} 

Let 
 \be    \xexit=\max\{ k\ge 0:  \text{$x_i=(i,0)$ for $0\le i\le k$} \}  \label{defxix}\ee
 and 
 \be    \yexit=\max\{ k\ge 0:  \text{$x_j=(0,j)$ for $0\le j\le k$} \}  \label{defxiy}\ee 
denote  the exit points of a path from the $x$- and $y$-axes. For any given path
exactly one  of $\xexit$ and  $\yexit$ is zero. 
In terms of \eqref{defv1}, $\xexit=\vvb(0)$.

For $\theta, x>0$  define the function 
\be 
\functaa(\theta, x)= \int_0^x \bigl(\digamf(\theta)-\log y\bigr) x^{-\theta}y^{\theta-1}e^{x-y}\,dy.
\label{psi1}\ee
The observation 
\[  \functaa(\theta, x)=  -\,\Gamma(\theta)x^{-\theta}e^{x}\,\Cvv[\,\log A, \ind\{A\le x\}\,] 
\]
for  $A\sim\text{Gamma}(\theta, 1)$ shows that 
$ \functaa(\theta, x)>0.  $   Furthermore,  $\E\functaa(\theta, A)=\trigamf(\theta)$.  
 
\begin{theorem}   Assume \eqref{distr4}.  Then for
$m,n\in\bZ_+$ we have these identities: 
\be  \Vvv\bigl[\log Z_{m,n}\bigr] = n\trigamf(\mu-\theta)-m\trigamf(\theta) + 2 \, \Eann_{m,n}   \biggl[ \; \sum_{i=1}^\xexit   \functaa(\theta,  Y_{i,0}^{-1})   \biggr]
\label{var1}\ee
and 
\be  \Vvv\bigl[\log Z_{m,n}\bigr] = -n\trigamf(\mu-\theta)+m\trigamf(\theta) + 2 \, \Eann_{m,n}  \biggl[ \; \sum_{j=1}^\yexit   \functaa(\mu-\theta,  Y_{0,j}^{-1})   \biggr].  
\label{var1.1}\ee
When $\xexit=0$   the    sum  $\sum_{i=1}^\xexit $  is interpreted as $0$, and similarly for  
 $\yexit=0$.  
\label{varxi-thm}\end{theorem}

\begin{proof}  We prove \eqref{var1}.  Identity \eqref{var1.1} then follows by a 
reflection across the diagonal. 
Let us abbreviate temporarily, according to the compass 
directions of the rectangle $\Lambda_{m,n}$,
\begin{align*}
\north =\log Z_{m,n}-\log Z_{0,n},\quad \south =\log Z_{m,0},
\quad \east =\log Z_{m,n}-\log Z_{m,0},\quad \west =\log Z_{0,n}. 
\end{align*}
Then 
\be\begin{aligned}
\Vvv\bigl[\log Z_{m,n}\bigr]&=\Vvv(\west +\north )=\Vvv(\west )+\Vvv(\north )+2\Cvv(\west ,\north )\\
&=\Vvv(\west )+\Vvv(\north )+2\Cvv(\south +\east -\north ,\north )\\
 &= \Vvv(\west )-\Vvv(\north )+2\Cvv(\south ,\north ).\\
\end{aligned}\label{aux3}\ee
The last equality came from the independence of $\east $ and $\north $, 
 from Theorem \ref{burkethm} and \eqref{Z2}.
By  assumption \eqref{distr4}  $\Vvv(\west ) =n\trigamf(\mu-\theta)$,
 and by  Theorem \ref{burkethm} 
 $\Vvv(\north )= m\trigamf(\theta)$.

To prove \eqref{var1} it remains to work on $\Cvv(\south ,\north )$.  In the remaining part of the
proof we wish to differentiate with respect to the parameter $\theta$ of the weights
$Y_{i,0}$  on
the $x$-axis (term $\south $) without involving the other weights. Hence now think of a system
with three independent parameters $\theta$, $\rho$ and $\mu$ and with weight 
distributions  (for $i,j\in\bN$)  
\[  
\text{$Y_{i,0}^{-1}$ $\sim$ Gamma$(\theta, 1)$,  $Y_{0,j}^{-1}$ $\sim$ Gamma$(\rho, 1)$,   and 
$Y_{i,j}^{-1}$ $\sim$ Gamma$(\mu, 1)$.}  \]
We first show that 
\be  \Cvv(\south ,\north ) =-\,\frac{\partial}{\partial\theta} \E(\north ).  \label{cov3}\ee
The variable $\south $ is a sum
\[  \south =\sum_{i=1}^m \log U_{i,0}.  \]
The joint density of the vector of summands $(\log U_{1,0},\dotsc, \log U_{m,0})$ is 
\[  g_\theta(y_1,\dotsc,y_m)= \Gamma(\theta)^{-m} 
\exp\Bigl( -\theta\sum_{i=1}^m y_i -\sum_{i=1}^{m}e^{-y_i}  \Bigr)  \]
on $\bR^m$. This comes from the
product of  Gamma$(\theta, 1)$ distributions.  
The density of $\south $ is 
\[  
f_\theta(s)= \Gamma(\theta)^{-m}e^{-\theta s}\int_{\bR^{m-1}} 
\exp\Bigl( -\sum_{i=1}^{m-1}e^{-y_i} -e^{-s+y_1+\dotsm+y_{m-1}}\Bigr) \,dy_{1,m-1}.
\]
We also see that, given $\south $,  the joint distribution of $(\log U_{1,0},\dotsc, \log U_{m,0})$
does not depend on $\theta$.  Consequently in the calculation below the conditional
expectation does not depend on $\theta$.
\be\begin{aligned}
\frac{\partial}{\partial\theta} \E(\north )
&=\frac{\partial}{\partial\theta} \int_\bR \E(\north \,\vert\, \south =s) f_\theta(s)\,ds
=\int_\bR \E(\north \,\vert\, \south =s)  \frac{\partial   f_\theta(s)}{\partial\theta}\,ds\\
&=\int_\bR \E(\north \,\vert\, \south =s) \Bigl( -s-m\frac{\Gamma'(\theta)}{\Gamma(\theta)}\,\Bigr)f_\theta(s)\,ds\\
&=-\E(\north \south )+ \E(\north )m\E(\log U) = -\E(\north \south )+ \E(\north )\E(\south ) \\
&=-\Cvv(\north ,\south ). 
\end{aligned}\label{aux4}\ee
To justify taking ${\partial}/{\partial\theta}$ inside the integral  we check that for all
$0<\theta_0<\theta_1$, 
\be
\int_\bR \E(\,\abs{\north }\,\vert\, \south =s) \sup_{\theta\in[\theta_0,\theta_1]} 
\biggl\lvert  \frac{\partial   f_\theta(s)}{\partial\theta}\biggr\rvert \,ds <\infty.  \label{check2}\ee
Since 
\begin{align*}
\sup_{\theta\in[\theta_0,\theta_1]} 
\biggl\lvert  \frac{\partial   f_\theta(s)}{\partial\theta}\biggr\rvert
\le  C(1+\abs{s}) \bigl(   f_{\theta_0}(s) +  f_{\theta_1}(s) \bigr) 
\end{align*}
it suffices to get a bound for a fixed $\theta>0$:
\begin{align*}
\int_\bR \E(\,\abs{\north }\,\vert\, \south =s) (1+\abs{s})  f_\theta(s)\,ds
&= \E\bigl[\, \abs{\north } (1+\abs{\south })\,\bigr]\\
&\le \norm{\north }_{L^2(\P)}  \, \norm{1+\south }_{L^2(\P)} <\infty  
\end{align*}
because $\north $ and $\south $ are sums of i.i.d.\ random variables with all moments.  
Dominated convergence and this   integrability
bound \eqref{check2} also give the continuity of $\theta\mapsto \Cvv(\north ,\south )$.  

The next step is to calculate $({\partial}/{\partial\theta}) \E(\north )$ by a coupling. 
Sometimes we add a sub- or superscript $\theta$ to expectations and covariances  
  to emphasize their dependence on the parameter $\theta$ of the distribution
of the initial weights on the $x$-axis.  We also introduce a direct functional
dependence on $\theta$  in $Z_{m,n}$  by realizing the weights $U_{i,0}$ 
as functions of uniform random variables.   Let
\be  F_\theta(x)=\int_0^x \frac{y^{\theta-1}e^{-y}}{\Gamma(\theta)}\,dy,   
\quad x\ge 0, \label{gamma1}\ee
be the c.d.f.\ of the Gamma$(\theta,1)$ distribution and 
$H_\theta$ its inverse function, defined on $(0,1)$, that satisfies
$\eta=F_\theta(H_\theta(\eta))$ for $0<\eta<1$.   Then if $\eta$ is a 
Uniform$(0,1)$ random variable,  $U^{-1}=H_\theta(\eta)$  is  a
Gamma$(\theta,1)$ random variable.  Let  $\eta_{1,m}=(\eta_1,\dotsc,\eta_m)$
be a vector of Uniform$(0,1)$ random variables.  
We redefine $Z_{m,n}$ as a function of the random variables 
$\{ \eta_{1,m};  Y_{i,j}:  {(i,j)\in\bZ_+\times\bN}\}$ without changing its distribution: 
\be  Z_{m,n}(\theta)
 = \sum_{x_{\centerdot}\in\Pi_{m,n}} \prod_{i=1}^\xexit H_\theta(\eta_{i})^{-1} \cdot 
 \prod_{k=\xexit+1}^{m+n} Y_{x_k}. 
\label{defZmn3}\ee
 
Next we look for the derivative:  
\begin{align*}
 \frac{\partial}{\partial\theta}\log Z_{m,n}(\theta) 
= &\frac 1 { Z_{m,n}(\theta)}
   \sum_{x_\centerdot\in\Pi_{m,n}} 
 \biggl( - \sum_{i=1}^\xexit  \frac{\partial H_\theta(\eta_{i}) }{\partial\theta}  H_\theta(\eta_{i})^{-1}  \biggr) \\
 & \qquad \times  \prod_{i=1}^\xexit H_\theta(\eta_{i})^{-1} \cdot 
 \prod_{k=\xexit+1}^{m+n} Y_{x_k}.  
  \end{align*}
Differentiate implicitly
$\eta=F(\theta, H(\theta, \eta))$  to find
\begin{align}
 \frac{\partial H(\theta,\eta) }{\partial\theta}  =
-\, \frac{({\partial F }/{\partial\theta})(\theta, H(\theta, \eta))} 
{({\partial F }/{\partial x})(\theta, H(\theta, \eta))} \,.
\label{dH}\end{align} 
(We write $F(\theta,x)=F_\theta(x)$ and $H(\theta,\eta)=H_\theta(\eta)$
 when subscripts are not 
convenient.)  
If we define 
\be  \functaa(\theta, x)= -\,\frac1x\cdot \frac{{\partial F(\theta, x) }/{\partial\theta}} 
{{\partial F(\theta, x) }/{\partial x}}\, ,\quad \theta, x>0,  \label{psi}\ee
 we can write 
\be\begin{aligned}
 \frac{\partial}{\partial\theta}\log Z_{m,n}(\theta) 
= \frac 1 { Z_{m,n}(\theta)}
   \sum_{x_{\centerdot}\in\Pi_{m,n}} 
 \biggl\{ - \sum_{i=1}^\xexit   \functaa(\theta,  H_\theta(\eta_{i}))   \biggr\} 
  \prod_{i=1}^\xexit H_\theta(\eta_{i})^{-1} \cdot 
 \prod_{k=\xexit+1}^{m+n} Y_{x_k}.  
  \end{aligned} \label{dlogZ}\ee
Direct calculation  shows that  \eqref{psi} agrees with the earlier definition \eqref{psi1} of $\functaa$. 

Since
$  \digamf(\theta)= 
{\Gamma(\theta)}^{-1}\int_0^\infty   (\log y) {y^{\theta-1}e^{-y}} \,dy,
$
we also have 
\be 
\functaa(\theta, x)= \int_x^\infty \bigl(-\digamf(\theta)+\log y\bigr) x^{-\theta}y^{\theta-1}e^{x-y}\,dy.
\label{psi1.1}\ee
 For $x\le 1$
drop $e^{-y}$ and compute the integrals in \eqref{psi1}, while 
for  $x\ge 1$ apply 
   H\"older's inequality judiciously to \eqref{psi1.1}. This shows 
\be  0< \functaa(\theta, x) \le 
\begin{cases}   C(\theta)(1-\log x)    &\text{for $0<x\le 1$} \\[5pt]
  C(\theta) x^{-1/4} 
&\text{for $x\ge 1$.}  \end{cases}  \label{psi2}\ee
 In particular,    
$\functaa(\theta,  H_\theta(\eta))$  with  $\eta\sim\text{Uniform}(0,1)$
  has
an exponential moment: for small enough $t>0$,  
\be\begin{aligned} 
  \E\bigl[ e^{t\functaa(\theta,H_\theta(\eta) )} \,\bigr]=   
\int_0^\infty  e^{t\functaa(\theta, x)} \,
\frac{x^{\theta-1}e^{-x}}{\Gamma(\theta)}\,dx <\infty. 
\end{aligned}\label{psimom}\ee

Let $\wt\E$ denote expectation over the variables $\{Y_{i,j}\}_{ (i,j)\in\bZ_+\times\bN}$
(that is, excluding the weights on the $x$-axis). 
From \eqref{aux4} we get
\be\begin{aligned}
&-\int_{\theta_0}^{\theta_1} \Cvv^\theta(\north ,\south )\,d\theta = \E^{\theta_1}(\north )-  \E^{\theta_0}(\north )\\
&\quad = \wt\E \int_{(0,1)^m} d\eta_{1,m}  \bigl(  \log Z_{m,n}(\theta_1) -  \log Z_{m,n}(\theta_0) \bigr)\\
&\quad = \wt\E \int_{(0,1)^m} d\eta_{1,m}  \int_{\theta_0}^{\theta_1} 
 \frac\partial{\partial\theta}  \log Z_{m,n}(\theta) \, d\theta \\
&\quad =   \int_{\theta_0}^{\theta_1}   d\theta \; \wt\E \int_{(0,1)^m} d\eta_{1,m} 
\;  \frac\partial{\partial\theta}  \log Z_{m,n}(\theta). \\
\end{aligned}\label{aux4.05}\ee
The last equality above came from Tonelli's theorem,  justified 
by \eqref{dlogZ} which shows that $(\partial/{\partial\theta})  \log Z_{m,n}(\theta)$ is 
always negative.  

From \eqref{dlogZ} , upon replacing $H(\theta, \eta_i)$ with $Y_{i,0}^{-1}$, 
\be\begin{aligned}
 \frac{\partial}{\partial\theta}\log Z_{m,n}(\theta) 
&= \frac 1 { Z_{m,n}(\theta)}
   \sum_{x_{\centerdot}\in\Pi_{m,n}} 
 \biggl\{ - \sum_{i=1}^\xexit   \functaa(\theta,  Y_{i,0}^{-1})   \biggr\} 
 \prod_{k=1}^{m+n} Y_{x_k} \\
&= -\,  E^{Q^\om_{m,n}}   \biggl[ \; \sum_{i=1}^\xexit   \functaa(\theta,  Y_{i,0}^{-1})   \biggr].  
  \end{aligned} \label{dlogZ2}\ee
Consequently from \eqref{aux4.05} 
\[  \int_{\theta_0}^{\theta_1} \Cvv^\theta(\north ,\south )\,d\theta =
\int_{\theta_0}^{\theta_1} 
\E^\theta  E^{Q^\om_{m,n}}   \biggl[ \; \sum_{i=1}^\xexit   \functaa(\theta,  Y_{i,0}^{-1})   \biggr]\,d\theta. \]
Earlier we justified the continuity of $\Cvv^\theta(\north ,\south )$ as a function of $\theta>0$.
Same is true for the integrand on the right.   Hence we get 
\be    \Cvv^\theta(\north ,\south ) =
 \Eann^\theta_{m,n}   \biggl[ \; \sum_{i=1}^\xexit   \functaa(\theta,  Y_{i,0}^{-1})   \biggr].
\label{aux4.1}\ee
Putting this back into \eqref{aux3} completes the proof. 
\end{proof}

\section{Upper bound for the model with boundaries} 
\label{sec:ub-bd}
In this section we prove the upper bound of Theorem \ref{var-bd-thm}. 
Assumption \eqref{distr4} is in force, with $0<\theta<\mu$ fixed. 
While keeping $\mu$ fixed we shall also consider an alternative
value $\lambda\in(0,\mu)$ and then assumption \eqref{distr4}
is in force but with $\lambda$ replacing $\theta$. 
Since $\mu$ remains fixed we omit
 dependence on $\mu$ from all notation. At times dependence
on $\lambda$ and $\theta$ has to be made explicit,  as
for example in the next lemma where  $\Vvv^\lambda$ denotes
 variance computed under assumption \eqref{distr4}
with  $\lambda$ replacing $\theta$. 

\begin{lemma}   Consider $0<\delta_0<\theta<\mu$ fixed.  
Then there exists a constant $C<\infty$ such that for all  $\lambda\in[\delta_0,\theta]$,
\be
 \Vvv^\lambda\bigl[\log Z_{m,n}\bigr]\le   \Vvv^\theta\bigl[\log Z_{m,n}\bigr] 
 +C(m+n)(\theta-\lambda).  
  \label{varineq1}\ee
A single constant  $C$ works for all $\delta_0<\theta<\mu$ that vary  in a compact set.  
\end{lemma}
\begin{proof}
Identity  \eqref{var1.1} will be convenient for $\lambda<\theta$: 
\begin{align}
& \Vvv^\lambda\bigl[\log Z_{m,n}\bigr]-  \Vvv^\theta\bigl[\log Z_{m,n}\bigr] \nn\\[6pt]
&\qquad = 
 -n\trigamf(\mu-\lambda)+m\trigamf(\lambda) +n\trigamf(\mu-\theta)-m\trigamf(\theta) \label{aux5.97}\\[2pt]
 &\qquad\quad 
 + 2 \, \E^\lambda  E^{Q^\om_{m,n}}   \biggl[ \; \sum_{j=1}^\yexit   \functaa(\mu-\lambda,  Y_{0,j}^{-1})   
 \biggr]
-   2 \, \E^\theta  E^{Q^\om_{m,n}}   \biggl[ \; \sum_{j=1}^\yexit   \functaa(\mu-\theta,  Y_{0,j}^{-1})   
 \biggr].  \label{aux5.98}
 \end{align} 
$\trigamf$ is  continuously differentiable   and so 
\[  \text{line \eqref{aux5.97} } \le C(m+n)(\theta-\lambda). \]

We work on line \eqref{aux5.98}. 
As in the proof of Theorem \ref{varxi-thm} we replace the weights
on the $x$- and  $y$-axes with functions of uniform random 
variables. We need explicitly only the ones on the $y$-axis,
denote these by  $\eta_j$. 
Write $\wt\E$ for the expectation over the uniform  variables and
the bulk weights 
$\{Y_{i,j}: i, j\ge 1\}$. This expectation no longer depends
on $\lambda$ or $\theta$.  The quenched measure $Q^\w$ does carry
dependence on these parameters, and we express that by a superscript
$\theta$ or $\lambda$.
\begin{align}
& \text{line \eqref{aux5.98} without the factor $2$}  \nn\\
 &= \; \wt\E  E^{Q^{\lambda,\om}_{m,n}}   \biggl[ \; \sum_{j=1}^\yexit   \functaa({\mu-\lambda}, H_{\mu-\lambda}(\eta_j)) \biggr]
-
 \wt\E  E^{Q^{\theta,\om}_{m,n}}   \biggl[ \; \sum_{j=1}^\yexit   \functaa({\mu-\theta}, H_{\mu-\theta}(\eta_j)) \biggr]
 \nn\\
&
=  \; \wt\E  E^{Q^{\lambda,\om}_{m,n}}   \biggl[ \; \sum_{j=1}^\yexit   \functaa({\mu-\lambda}, H_{\mu-\lambda}(\eta_j)) \biggr]
-
 \wt\E  E^{Q^{\lambda,\om}_{m,n}}   \biggl[ \; \sum_{j=1}^\yexit   \functaa({\mu-\theta}, H_{\mu-\theta}(\eta_j)) \biggr]
 \label{aux6}\\
&\qquad \qquad + \; 
 \wt\E  E^{Q^{\lambda,\om}_{m,n}}   \biggl[ \; \sum_{j=1}^\yexit   \functaa({\mu-\theta}, H_{\mu-\theta}(\eta_j)) \biggr]
-
 \wt\E  E^{Q^{\theta,\om}_{m,n}}   \biggl[ \; \sum_{j=1}^\yexit   \functaa({\mu-\theta}, H_{\mu-\theta}(\eta_j)) \biggr].
 \label{aux7} 
 \end{align}

We first show that line \eqref{aux7} is $\le 0$, by showing that, as the parameter $\rho$ 
in $Q^{\om, \rho}_{m,n}$ increases, the random variable $\yexit$ increases stochastically. 
Write $B_j=H_{\mu-\rho}(\eta_j)$ for the Gamma($\mu-\rho, 1$) variable that gives
the weight $Y_{0,j}=B_j^{-1}$  
   in the
definition of $Q^{\om, \rho}_{m,n}$. For a given $\mu$,  $B_j$ decreases as $\rho$
increases.  Thus it suffices
to show that,  for
  $1\le k, \ell\le n$, 
  \be (\partial/{\partial B_\ell}) Q^\om\{\yexit\ge k\}\le 0. \label{aux7.01}\ee 
 Write $W=\prod_{j=1}^{\yexit} B_j^{-1}\cdot \prod_{k=\yexit+1}^{m+n} Y_{x_k}$
  for the total weight of a path $x_\centerdot$
(the numerator of  the quenched polymer probability of the path). 
\begin{align*}
 &\frac\partial{\partial B_\ell} Q^\om\{\yexit\ge k\}  
 =\frac\partial{\partial B_\ell} 
\biggl( \frac1{Z_{m,n}} \sum_{x_\centerdot}   
\ind\{\yexit\ge k\} W \biggr) \\
&= \frac1 { Z_{m,n}} 
 \sum_{x_\centerdot} \ind\{\yexit\ge k\} \ind\{\yexit\ge \ell\} (-B_\ell^{-1}) W   \\
&\qquad\qquad\qquad\qquad
 - \;\frac1 { Z_{m,n}^2} \Bigl(\;  \sum_{x_\centerdot} \ind\{\yexit\ge k\} W\Bigr)\cdot 
\Bigl( \; \sum_{x_\centerdot} \ind\{\yexit\ge \ell\} (-B_\ell^{-1}) W\Bigr)  \\
 &= -B_\ell^{-1}  \Cov^{Q^\om} \bigl[  \ind\{\yexit\ge k\} ,\ind\{\yexit\ge \ell\} \bigr] <0. 
 \end{align*}
 Thus we can bound line  \eqref{aux7} above by $0$. 

On  line \eqref{aux6}  inside the brackets only $\yexit$ is random under  $Q^{\om,\lambda}_{m,n}$.
We replace $\yexit$ with its upper bound  $n$ and then we are left with integrating over 
uniform variables $\eta_j$.  
\begin{align}
\abs{\text{ line \eqref{aux6} }} &\le  \; 
\wt\E  E^{Q^{\lambda,\om}_{m,n}}   \biggl[ \; \sum_{j=1}^\yexit   
\bigl\lvert \functaa({\mu-\lambda}, H_{\mu-\lambda}(\eta_j))  - 
 \functaa({\mu-\theta}, H_{\mu-\theta}(\eta_j))   \bigr\rvert \; \biggr]\nn\\
&\le n \int_0^1   \bigl\lvert \functaa({\mu-\lambda}, H_{\mu-\lambda}(\eta))  - 
 \functaa({\mu-\theta}, H_{\mu-\theta}(\eta))   \bigr\rvert\,d\eta \nn\\
 &=n \int_0^1  \int_{\mu-\theta}^{\mu-\lambda}  \, 
 \biggl\lvert \frac{d}{d \rho}  \functaa(\rho, H_{\rho}(\eta)) \biggr\rvert  \,d\rho\,d\eta
\label{aux7.1}\end{align} 
 From \eqref{dH} and \eqref{psi},  
 \begin{align*} \frac{d}{d \rho}  \functaa(\rho, H_{\rho}(\eta)) 
 &=\frac{\partial\functaa}{\partial\rho} + \frac{\partial\functaa}{\partial x}  \frac{\partial H_{\rho}(\eta)}{\partial \rho}\\
 &=\Bigl( \, \frac{\partial\functaa(\rho, x)}{\partial\rho} + 
x \functaa(\rho, x) \frac{\partial\functaa(\rho, x)}{\partial x} \,\Bigr)\Big\vert_{x=H_{\rho}(\eta)}. 
 \end{align*}
Utilizing \eqref{psi2} and explicit computations leads to bounds 
\be
\Bigl\lvert\, \frac{\partial\functaa(\rho, x)}{\partial\rho} + 
x \functaa(\rho, x) \frac{\partial\functaa(\rho, x)}{\partial x} \,\Bigr\rvert \le \begin{cases} 
C(\rho)(1+(\log x)^2)  &\text{for $0<x\le 1$}\\
C(\rho)x^{1/2}   &\text{for $x\ge 1$.}
\end{cases}    \label{psi3}\ee 
With  $\rho$ restricted to  a compact subinterval of $(0,\infty)$, 
these bounds are valid for a fixed constant $C$.
Continue from \eqref{aux7.1}, letting $B_\rho$ denote a Gamma$(\rho, 1)$ random variable: 
\begin{align}
\text{line \eqref{aux6} } &\le  
n   \int_{\mu-\theta}^{\mu-\lambda} \int_0^1 \, 
 \biggl\lvert \frac{d}{d \rho}  \functaa(\rho, H_{\rho}(\eta)) \biggr\rvert \,d\eta \,d\rho\nn\\
&\le Cn  \int_{\mu-\theta}^{\mu-\lambda}  \E \bigl[  1 + (\log B_\rho)^2
 + B_\rho^{1/2} \bigr] \,d\rho \nn\\
&\le Cn(\theta-\lambda).\nn
\end{align}

To summarize, we have shown that line \eqref{aux5.98} $\le$ $Cn(\theta-\lambda)$
and thereby completed the proof of the lemma.   
\end{proof}

The preliminaries are ready and we turn to the upper bound.  
Let the scaling parameter $N\ge 1$ be real valued.  
We assume that the dimensions $(m,n)\in\bN^2$ of the  
 rectangle satisfy  
\be \abs{\, m-{N\trigamf(\mu-\theta)}\, }\le \mnc_N
\quad\text{and}\quad 
 \abs{\, n- {N\trigamf(\theta)}\,}
\le \mnc_N \label{mnN1}\ee
for a sequence   $\mnc_N\le CN^{2/3}$ with a fixed constant $C<\infty$.

For a walk $x_\centerdot$ such that $\xexit>0$, weights at distinct parameter
values are related by 
\begin{align*}  W(\theta)
=\prod_{i=1}^{\xexit} H_\theta(\eta_i)^{-1} \cdot\prod_{k=\xexit+1}^{m+n} Y_{x_k}
=W(\lambda)\cdot \prod_{i=1}^{\xexit} \frac{H_\lambda(\eta_i)}{H_\theta(\eta_i)}.
\end{align*}
For $\lambda<\theta$,  ${H_\lambda(\eta)}\le {H_\theta(\eta)}$ and consequently 
\be\begin{aligned}
Q^{\theta,\om}\{\xexit\ge u\} 
&= \frac1{Z(\theta)} \sum_{x_\centerdot}   
\ind\{\xexit\ge u\} W(\theta) 
\le   \frac{Z(\lambda)}{Z(\theta)}
 \cdot \prod_{i=1}^{\fl u} \frac{H_\lambda(\eta_i)}{H_\theta(\eta_i)}. 
\end{aligned} \label{aux8.00}\ee
We bound the $\P$-tail of $Q^\om\{\xexit\ge u\}$ separately for two ranges 
of a positive real $u$.
Let $c, \delta>0$ be  constants.  
Their values  will be determined  in the course of the proof. 
For future use of the estimates developed here it is to
be  noted that $c$ and $\delta$, and the other constants introduced in this upper bound 
proof,  
are functions of $(\mu,\theta)$ and nothing else, and furthermore, fixed values
of the constants work for $0<\theta<\mu$ in a compact set.

\medskip

{\bf Case 1.}  $(1\vee c\mnc_N) \le u\le\delta N$. 

\medskip

Pick an auxiliary parameter value 
\be  \lambda =\theta-\frac{b u}{N}. \label{defla}\ee
We can assume $b>0$ and $\delta>0$ small enough so that $b\delta<\theta/2$ 
and then $\lambda\in(\theta/2, \theta)$.  
Let \be \alpha=\exp[u(\digamf(\lambda)-\digamf(\theta))+\delta u^2/N].\label{alpha}\ee
Consider $0<s<\delta$.  
  First a split into
two probabilities. 
\begin{align}
&\P\bigl[ Q^\om\{\xexit\ge u\}  \ge e^{-su^2/N} \,\bigr]
\le \P\biggl\{ \; \prod_{i=1}^{\fl u} \frac{H_\lambda(\eta_i)}{H_\theta(\eta_i)} \ge \alpha \biggr\}  
\label{aux8.01}\\[2pt]
&\qquad\qquad   + \P\biggl( \;   \frac{Z(\lambda) }{Z(\theta) } 
  \ge  \alpha^{-1}e^{-su^2/N} \,\biggr).
\label{aux8.02}\end{align}
Recall that $\E(\log H_\theta(\eta))=\digamf(\theta)$ and that overline
denotes a centered random variable. Then
for the second probability on line \eqref{aux8.01}, 
\be\begin{aligned}
&\P\biggl\{ \; \prod_{i=1}^{\fl u} \frac{H_\lambda(\eta_i)}{H_\theta(\eta_i)} \ge \alpha \biggr\}  \\
&=\P\biggl\{ \; \sum_{i=1}^{\fl u} \bigl(\;\overline{\log H_\lambda(\eta_i)}
- \overline{\log H_\theta(\eta_i)} \; \bigr) 
\ge  (u-\fl u) (\digamf(\lambda)-\digamf(\theta))+  \delta u^2/N  \biggr\} \\
&\qquad \le  \frac{4N^2}{\delta^2 u^3} \Vvv\bigl[ {\log H_\lambda(\eta)}
-  {\log H_\theta(\eta)}\bigr]  \le  C \frac{N^2}{ u^3}. 
\end{aligned}\label{aux8.05}\ee
The extra term with the integer part correction goes away  because
\[  \digamf(\lambda)-\digamf(\theta)\ge -C(\theta)(\theta-\lambda)
= -C(\theta) \frac{bu}N \ge - \frac{\delta u^2}{2N}\,,   \]
$u\ge 1$, and we  can choose $b$ small enough.  

Rewrite the probability from line \eqref{aux8.02} as
\be\begin{aligned}
&\P\Bigl( \;  \overline{ \log Z(\lambda)}-\overline{\log Z(\theta)} \, 
  \ge  -\E\bigl[\log Z(\lambda)\bigr] 
  +\E\bigl[\log Z(\theta) \bigr]  -\log \alpha  -su^2/N \,\Bigr).
\end{aligned}\label{aux8.06}\ee
Recall the mean from \eqref{ElogZ}. 
Rewrite the right-hand side of the inequality inside the probability above as follows: 
\begin{align}
& -\E\bigl[\log Z(\lambda)\bigr]  +\E\bigl[\log Z(\theta) \bigr]  -\log \alpha  -su^2/N\nn\\
&= \bigl( n\digamf(\mu-\lambda)+m\digamf(\lambda)\bigr) 
-\bigl( n\digamf(\mu-\theta)+m\digamf(\theta)\bigr) -\log \alpha  -su^2/N \nn\\
 &\ge \bigl(u-N\trigamf(\mu-\theta)\bigr)\bigl(\digamf(\theta)-\digamf(\lambda)\bigr) \nn\\
&\qquad\qquad -N\trigamf(\theta)\bigl(\digamf(\mu-\theta)- \digamf(\mu-\lambda)\bigr) 
 -(\delta+s)u^2/N \nn\\[3pt]
&\qquad\qquad - \mnc_N\abs{ \digamf(\lambda)-\digamf(\theta)}
 - \mnc_N\abs{ \digamf(\mu-\lambda)-\digamf(\mu-\theta)}   \nn\\[3pt]
&\ge u \trigamf(\theta)(\theta-\lambda)  
+\tfrac12{N} \bigl(  \trigamf(\mu-\theta)\trigamf'(\theta) +\trigamf(\theta)\trigamf'(\mu-\theta) \bigr) (\theta-\lambda)^2
\label{aux8.065}\\[2pt]
&\qquad\qquad-(\delta+s)u^2/N 
  -C_1(\theta,\mu) \bigl( u(\theta-\lambda)^2+N(\theta-\lambda)^3\bigr)\nn\\[2pt]
&\qquad\qquad- C_1(\theta,\mu)\mnc_N (\theta-\lambda) \nn\\
&\ge \bigl( b \trigamf(\theta) - C_2(\theta,\mu)b^2-2\delta -C_1(\theta,\mu)\delta(b^2+b^3)\bigr) \frac{u^2}{N}  -C_1(\theta,\mu) \mnc_N\frac{bu}N  
\label{aux8.066}\\
&\ge  \frac{c_1u^2}{N} .
\label{aux8.08}
\end{align}
Inequality \eqref{aux8.065}  with a  constant $C_1(\theta,\mu)>0$ came from the   expansions 
\begin{align*}
\digamf(\theta)-\digamf(\lambda) &= \trigamf(\theta)(\theta-\lambda)-\tfrac12 \trigamf'(\theta)(\theta-\lambda)^2
+\tfrac16 \trigamf''(\rho_0)(\theta-\lambda)^3
\intertext{and}
\digamf(\mu-\theta)- \digamf(\mu-\lambda) &= -\trigamf(\mu-\theta)(\theta-\lambda)-
\tfrac12 \trigamf'(\mu-\theta)(\theta-\lambda)^2
-\tfrac16\trigamf'' (\rho_1)(\theta-\lambda)^3, 
\end{align*}
for some $\rho_0, \rho_1\in(\lambda, \theta)$.  
For inequality \eqref{aux8.066} we defined 
\[   C_2(\theta,\mu) =  -\,\tfrac12 \bigl(  \trigamf(\mu-\theta)\trigamf'(\theta) +\trigamf(\theta)\trigamf'(\mu-\theta) \bigr)>0,  \] 
substituted in $\lambda=\theta-bu/N$ from \eqref{defla}, and recalled that $s<\delta$   and 
$u\le\delta N$.  To get \eqref{aux8.08} 
 we fixed $b>0$ small enough, then $\delta>0$
small enough,   defined a new constant $c_1>0$, and restricted 
$u$ to satisfy
\be
u\ge c\mnc_N \label{cond-u1}\ee
for another constant $c$.  We can also restrict to $u\ge 1$ if
the condition above does not enforce it. 

Substitute line \eqref{aux8.08} on the right-hand side inside probability  \eqref{aux8.06}.
This probability  came from line \eqref{aux8.02}.
Apply Chebyshev, then \eqref{varineq1}, and finally \eqref{var1}: 
\begin{align}
\text{line  \eqref{aux8.02}} 
&\le \P\Bigl( \;  \overline{ \log Z(\lambda)}-\overline{\log Z(\theta)} \, 
  \ge c_1u^2/N\Bigr) \\
&\le  \frac{CN^2}{u^4}\Vvv\bigl[ \log Z(\lambda)-\log Z(\theta)\bigr]\nn\\
&\le  \frac{CN^2}{u^4}\Bigl( \Vvv\bigl[ \log Z(\lambda)\bigr] +  \Vvv\bigl[\log Z(\theta)\bigr]\Bigr)\nn\\
&\le  \frac{CN^2}{u^4}\Bigl( \Vvv\bigl[ \log Z(\theta)\bigr] + N(\theta-\lambda)\Bigr)\nn\\
&\le  \frac{CN^2}{u^4}  \Eann    \biggl[ \; \sum_{i=1}^\xexit   \functaa(\theta,  Y_{i,0}^{-1})   \biggr]  + \frac{CN^2}{u^3}.  
\label{aux8.09}
\end{align}
Collecting \eqref{aux8.01}--\eqref{aux8.02}, \eqref{aux8.05}   and \eqref{aux8.09} gives
this intermediate result:   for $0<s<\delta$, $N\ge 1$, 
 and $1\vee c\mnc_N\le u\le \delta N$, 
\be
\P\bigl[ Q^\om\{\xexit\ge u\}  \ge e^{-su^2/N} \,\bigr] \le  \frac{CN^2}{u^4}  \Eann   \biggl[ \; \sum_{i=1}^\xexit   \functaa(\theta,  Y_{i,0}^{-1})   \biggr]  + \frac{CN^2}{u^3}.  
\label{aux8.11}\ee

\begin{lemma}  There exists a constant $0<C<\infty$ such that 
\be  \Eann   \biggl[ \; \sum_{i=1}^\xexit   \functaa(\theta,  Y_{i,0}^{-1})   \biggr] 
\le C  \bigl( \Eann (\xexit)  +1\bigr). \label{aux8.13}\ee
\end{lemma}
\begin{proof}  Write again $A_i= Y_{i,0}^{-1}$ for the Gamma$(\theta, 1)$ 
variables.  Abbreviate $\functaa_i=\functaa(\theta, A_i)$,  
$\functaabar_i= \functaa_i-\E\functaa_i$ and $S_k= \sum_{i=1}^k    \functaabar_i $.    
\begin{align*}
\Eann   \biggl[ \, \sum_{i=1}^\xexit   \functaa_i   \biggr] 
&= \E(\functaa_1) \Eann(\xexit) + \Eann   \biggl[ \, \sum_{i=1}^\xexit   \functaabar_i  \biggr]
=  \E(\functaa_1) \Eann(\xexit) +
\sum_{k=1}^m \E  \bigl[ \, Q^\om\{\xexit=k\} S_k \bigr] \\
&\le \bigl( \E(\functaa_1)+1\bigr)   \Eann (\xexit) + \sum_{k=1}^m \E  \bigl[   
\ind\bigl\{ S_k \ge k\bigr\}  S_k \, \bigr] 
\le C\Eann (\xexit) + C. 
\end{align*}
The last bound comes from the fact that $\{\functaabar_i\}$ are i.i.d.\ mean zero with
  all moments 
(recall \eqref{psimom}): 
\begin{align*} 
\E  \bigl[    \ind\bigl\{ S_k \ge k\bigr\}  S_k \, \bigr] &\le 
\bigl(k \E(\functaabar^2)\bigr)^{1/2} \bigl(\P\{ S_k \ge k\}\bigr)^{1/2} \\
&\le Ck^{1/2} \Bigl( k^{-8}  E(S_k^8)\Bigr)^{1/2}  \le Ck^{-3/2} 
\end{align*} 
and these are summable.  
\end{proof}

Since $u\ge 1$, we can combine \eqref{aux8.11} and \eqref{aux8.13} to give 
\be
\P\bigl[ Q^\om\{\xexit\ge u\}  \ge e^{-su^2/N} \,\bigr] \le\frac{CN^2}{u^4}  \Eann(\xexit) + \frac{CN^2} {u^3} 
\label{aux8.15}\ee
still for $0<s<\delta$ and $(1\vee c\mnc_N)\le u\le \delta N$.  

\medskip

{\bf Case 2.}  $(1\vee c\mnc_N\vee\delta N) \le u <\infty$. 

\medskip

The constant $\delta>0$ is now fixed small enough by Case 1.    Take new constants
$\nu>0$ and $\delta_1>0$ and set
 \[  \lambda=\theta-\nu  \]
and 
  \be \alpha=\exp[u(\digamf(\lambda)-\digamf(\theta))+\delta_1 u].\label{alpha2}\ee
Consider $0<s<\delta_1$.  First use again \eqref{aux8.00} to split the probability:
\begin{align}
&\P\bigl[ Q^\om\{\xexit\ge u\}  \ge e^{-su} \,\bigr]
\le \P\biggl\{ \; \prod_{i=1}^{\fl u} \frac{H_\lambda(\eta_i)}{H_\theta(\eta_i)} \ge \alpha \biggr\}  
+ \P\biggl( \;   \frac{Z(\lambda) }{Z(\theta) } 
  \ge  \alpha^{-1}e^{-su} \,\biggr) \nn\\
  &\qquad \le \; 
   \P\biggl\{ \; \sum_{i=1}^{\fl u} \bigl(\;\overline{\log H_\lambda(\eta_i)}
- \overline{\log H_\theta(\eta_i)} \; \bigr) \ge  \tfrac12\delta_1 u   \biggr\}   
\nn\\
&\qquad +     \P\Bigl( \;  \overline{ \log Z(\lambda)}-\overline{\log Z(\theta)} \, 
  \ge  -\E\bigl[\log Z(\lambda)\bigr] 
  +\E\bigl[\log Z(\theta) \bigr]  -\log \alpha  -su \,\Bigr).
\label{aux8.022}\end{align}
 
Logarithms of   gamma variables have an exponential moment:
\[   \E[e^{t\abs{\log H_\theta(\eta)}}]<\infty  
\quad \text{ if $t<\theta$.} \]    
Hence  
  standard large deviations apply,  and  for some constant $c_4>0$,
\be\begin{aligned}
 \P\biggl\{ \; \sum_{i=1}^{\fl u} \bigl(\;\overline{\log H_\lambda(\eta_i)}
- \overline{\log H_\theta(\eta_i)} \; \bigr) \ge  \tfrac12\delta_1 u   \biggr\} 
\le e^{-c_4 u}. 
\end{aligned}\label{aux8.055}\ee

Following the pattern that led to \eqref{aux8.08}, 
the right-hand side inside probability  \eqref{aux8.022} is bounded as follows: 
   \begin{align}
& -\E\bigl[\log Z(\lambda)\bigr]  +\E\bigl[\log Z(\theta) \bigr]  -\log \alpha  -su\nn\\
&\ge u \trigamf(\theta)(\theta-\lambda)  -NC_2(\theta)(\theta-\lambda)^2
 -(\delta_1+s)u
  -C_1(\theta) \bigl( u(\theta-\lambda)^2+N(\theta-\lambda)^3\bigr)\nn\\[2pt]
&\qquad\qquad- C_1(\theta)\mnc_N (\theta-\lambda) \nn\\
 &\ge u\Bigl[   \trigamf(\theta)\nu - \frac{C_2(\theta)\nu^2}{\delta}  -2\delta_1 -C_1(\theta)(\nu^2
 +\nu^3/\delta)  \Bigr] - C_1(\theta)\mnc_N\nu  \nn\\
 &\ge c_5 u  \nn
\end{align}  
for a constant $c_5>0$, when we fix $\nu$ and $\delta_1$ small enough
and again also enforce \eqref{cond-u1} 
 $u\ge c\mnc_N$ for a large enough $c$. 
By standard large deviations,  since $\log Z(\lambda)$ and $\log Z(\theta)$ can be
expressed as   sums of i.i.d.\ random variables with an exponential moment,  
and for $u\ge \delta N$, 
\be \begin{aligned}
\text{probability  \eqref{aux8.022}} \le 
 \P\Bigl( \;  \overline{ \log Z(\lambda)}-\overline{\log Z(\theta)} \, 
  \ge  c_5u\Bigr)  \le  e^{-c_6u}.  
  \end{aligned} \label{aux8.057}\ee
Combining \eqref{aux8.055} and \eqref{aux8.057} gives the bound 
 \be   \P\bigl[ Q^\om\{\xexit\ge u\}  \ge e^{-su} \,\bigr] \le  2e^{-c_7u}  
 \label{case2}\ee
 for $0<s<\delta_1$ and $u\ge \delta N$.    Integrate and use \eqref{case2}: 
\be\begin{aligned} 
&\int_{\delta N}^\infty \Pann(\xexit\ge u) \,du 
=  \int_{\delta N }^\infty  du \int_0^1 dt\, \P\bigl[  Q^\om(\xexit\ge u)\ge t \bigr] \\
&= \int_{\delta N }^\infty  du \int_0^\infty ds\,  ue^{-su}\,\P\bigl[  Q^\om(\xexit\ge u)\ge e^{-su} \bigr]\\
&\le   2c_7^{-1} e^{-c_7\delta N}  + \delta_1^{-1} e^{-\delta_1\delta N} \le  C. 
\end{aligned}\label{case2a}\ee
  
\medskip

Now we combine the two cases to finish the proof of the upper bound.  
Let $r\ge 1$ be large enough so that $c\mnc_N\le rN^{2/3}$ for all $N$
for the constant $c$ that appeared in \eqref{cond-u1}. 
\begin{align*}
\Eann(\xexit) &\le rN^{2/3} +\int_{rN^{2/3}}^{\delta N} \Pann(\xexit\ge u) \,du 
+\int_{\delta N}^\infty \Pann(\xexit\ge u) \,du \\
&\le C+rN^{2/3} +\int_{rN^{2/3}}^{\delta N} du \int_0^1  \P\bigl[  Q^\om(\xexit\ge u)\ge t \bigr] \,dt \\
&\le C+ rN^{2/3} +\int_{rN^{2/3}}^{\delta N} du \int_0^\delta 
\P\bigl[ Q^\om\{\xexit\ge u\}  \ge e^{-su^2/N} \,\bigr] \frac{u^2}N e^{-su^2/N}  
  \,ds \\
  \intertext{[substitute in \eqref{aux8.15} and integrate away the $s$-variable]}
&\le C+ rN^{2/3} +C\int_{rN^{2/3}}^\infty  \Bigl(\, \frac{N^2}{u^4}  \Eann(\xexit) 
+ \frac{N^2} {u^3} \, \Bigr) \,du \\  
&= C+ rN^{2/3} +\frac{C}{3r^3} \Eann(\xexit)  + \frac{CN^{2/3}}{2r^2}.
\end{align*}
If $r$ is fixed large enough relative to $C$, we obtain, with a new constant $C$  
\be    \Eann(\xexit)  \le  CN^{2/3}.   \label{goal8}\ee
This is valid   for all $N\ge 1$.  The constant $C$ depends on 
$(\mu, \theta)$ and the
other constants $\delta, \delta_1, b$ introduced along the way.  A single constant
works for $0<\theta<\mu$  that vary in  a compact set. 
 
Combining \eqref{var1},  \eqref{aux8.13}  and \eqref{goal8}  gives the upper variance bound
for the free energy:
\be  \Vvv[\log Z_{m,n}] \le CN^{2/3}.   \label{goal8.1}\ee

 Combining \eqref{aux8.15} and \eqref{case2} with  \eqref{goal8} gives this lemma:
 
 \begin{lemma} Assume weight distributions \eqref{distr4}
and rectangle dimensions \eqref{mnN1}.
  Then there are finite positive constants
 $\delta, \delta_1, c, c_1$ and $C$ such that 
 for $N\ge 1$ and  $(1\vee c\mnc_N)\le u\le \delta N$, 
 \be
\P\bigl[ Q^\om\{\xexit\ge u\}  \ge e^{-\delta u^2/N} \,\bigr] \le C\biggl( \frac{N^{8/3}}{u^4}    + 
\frac{N^2} {u^3} \biggr)
\label{aux8.6}\ee
while for  $N\ge 1$ and   $u\ge (1\vee c\mnc_N\vee \delta N)$, 
 \be   \P\bigl[ Q^\om\{\xexit\ge u\}  \ge e^{-\delta_1 u} \,\bigr]
 \le  e^{-c_1u}.  
 \label{case2c}\ee
Same bounds hold for $\yexit.$  The same constants work for 
$0<\theta<\mu$  that vary in  a compact set. 
\label{auxlm7}\end{lemma}

Integration gives these annealed bounds:

\begin{corollary} There are constants $0<\delta,  c, c_1, C <\infty$ 
such that 
for $N\ge 1$, 
  \be
\Pann\{\xexit\ge u\}   \le 
\begin{cases} C\Bigl( \frac{N^{8/3}}{u^4}    +  \frac{N^2} {u^3} \Bigr),
  &(1\vee c\mnc_N)\le u\le \delta N\\[5pt]
  2e^{-c_1u},  &u\ge (1\vee c\mnc_N\vee \delta N).  \end{cases}
  \label{upaux13}\ee
Same bounds hold for $\yexit.$
\label{coraux3}  \end{corollary} 

From the upper variance bound \eqref{goal8.1} and Theorem \ref{burkethm} 
we can easily 
deduce the central limit theorem for off-characteristic rectangles. 

\begin{proof}[Proof of Corollary \ref{clt-cor}]
Set $m_1=\fl{\trigamf(\mu-\theta)N}$.  
 Recall that overline means centering at the mean.
Since $Z_{m,n}=Z_{m_1,n}\cdot\prod_{i=m_1+1}^{m} U_{i,n}$, 
 \begin{align*}
 N^{-\alpha/2}\,  \overline{\log Z_{m,n}}
=  N^{-\alpha/2}\, \overline{\log Z_{m_1,n}}+ N^{-\alpha/2}\sum_{i=m_1+1}^{m} \overline{\log U_{i,n}}.
\end{align*}
Since $(m_1,n)$ is of characteristic shape,  \eqref{goal8.1} implies that  the first term
on the right is stochastically 
$O(N^{1/3-\alpha/2})$. Since $\alpha>2/3$ this term  converges to zero  in probability. 
 The second term is a sum of approximately $c_1N^\alpha$  i.i.d.\ terms
and hence satisfies a CLT.  
\end{proof}

\section{Lower bound for the model with boundaries}
\label{sec:lb-bd}

In this section we finish the proof of Theorem \ref{var-bd-thm} by providing the
lower bound.  
For subsets $A\subseteq\Pi_{(i,j),(k,\ell)}$ of paths, 
let us introduce the notation 
\be 
Z_{(i,j),(k,\ell)}(A)=\sum_{x_\centerdot \in A} \prod_{r=1}^{k-i+\ell-j} Y_{x_r} 
\label{Z6}\ee
for a restricted partition function. Then the quenched polymer probability
can be  written $Q_{m,n}(A)= Z_{m,n}(A)/Z_{m,n}$. 

\begin{lemma}  For $m\ge 2$ and $n\ge 1$ 
we have this comparison of partition
functions:
\be \frac{Z_{m,n}(\yexit>0)}{Z_{m-1,n}(\yexit>0)} \le 
\frac{Z_{(1,1), (m,n)}}{Z_{(1,1),(m-1,n)}}
\le \frac{Z_{m,n}(\xexit>0)}{Z_{m-1,n}(\xexit>0)}. \label{Zcomp2}\ee
\end{lemma}
\begin{proof} Ignore the original boundaries given by the
coordinate axes. 
Consider these partition functions 
on the positive quadrant
 $\bN^2$ with boundary $\{(i,1):i\in\bN\}\cup\{(1,j):j\in\bN\}$.
The boundary values for $Z_{(1,1), (m,n)}$ are 
 $\{Y_{i,1}:i\ge 2\}\cup\{Y_{1,j}:j\ge 2\}$.

From the definition of $Z_{m,n}(\yexit>0)$ 
\[  Z_{1,1}(\yexit>0)=V_{0,1}Y_{1,1} \quad\text{and}\quad
V_{1,2}= \frac{Z_{1,2}(\yexit>0)}{Z_{1,1}(\yexit>0)}
=  Y_{1,2}\biggl(1+\frac{V_{0,2}}{Y_{1,1}}\biggr).  \]
For $j\ge 3$ 
 apply \eqref{UVY1} inductively to compute the 
vertical boundary
values $V_{1,j}=Y_{1,j}(1+U_{1,j-1}^{-1}V_{0,j})$. 
 $V_{1,j}\ge Y_{1,j}$ for all $j\ge 2$.  
The  horizontal boundary values for $Z_{m,n}(\yexit>0)$ are 
simply $U_{i,1}=Y_{i,1}$ for $i\ge 2$.  
Lemma \ref{monolm} gives 
\[  \frac{Z_{m,n}(\yexit>0)}{Z_{m-1,n}(\yexit>0)} \le 
\frac{Z_{(1,1), (m,n)}}{Z_{(1,1),(m-1,n)}}
\quad\text{and}\quad
 \frac{Z_{m,n}(\yexit>0)}{Z_{m,n-1}(\yexit>0)} \ge 
\frac{Z_{(1,1), (m,n)}}{Z_{(1,1),(m,n-1)}}. \]
The second  inequality  of \eqref{Zcomp2} comes by transposing
the  second  inequality  above. 
\end{proof}

Relative to  a fixed rectangle $\Lambda_{m,n}=\{0,\dotsc,m\}\times\{0,\dotsc,n\}$,
define distances of  entrance points on the north and east boundaries
 from the corner $(m,n)$
 as  duals of the exit points \eqref{defxix}--\eqref{defxiy}:
 \be    \dxexit=\max\{ k\ge 0:  \text{$x_{m+n-i}=(m-i,n)$ for $0\le i\le k$} \}  \label{dualxix}\ee
 and 
 \be    \dyexit=\max\{ k\ge 0:  \text{$x_{m+n-j}=(m,n-j)$ for $0\le j\le k$} \}.  \label{dualxiy}\ee 

The next observation will not be used in the sequel, but it is curious to note
  the  following effect of the boundary conditions:  the chance that the last step
of the polymer
path  is along the $x$-axis does not depend on the endpoint $(m,n)$, but the chance
that the first step is along the $x$-axis increases strictly with $m$.  

\begin{proposition}  For all $m,n\ge 1$ these hold: 
\be  Q^\w_{m,n}\{\dxexit>0\}\overset{d}= \frac{A}{A+B} \label{xidist1}\ee
where $A\sim$ Gamma$(\theta,1)$ and  $B\sim$ Gamma$(\mu-\theta,1)$ are independent.
On the other hand,  
\be Q^\w_{m,n}\{\xexit>0\}\overset{d}=Q^\w_{m+1,n}\{\xexit>1\} < Q^\w_{m+1,n}\{\xexit>0\}. 
\label{xidist2}\ee
\end{proposition}
\begin{proof}  By the definitions, 
\begin{align*}  &Q^\w_{m,n}\{\dxexit>0\} = \frac{Z_{m-1,n}Y_{m,n}}{Z_{m,n}} 
=\frac{U_{m,n}^{-1}}{U_{m,n}^{-1}+V_{m,n}^{-1}}. 
\end{align*}
The distributional claim \eqref{xidist1}  follows from the Burke property Theorem \ref{burkethm}.

For the  distributional claim in \eqref{xidist2}  observe first  directly from  
definition \eqref{Q*1} that  $Q^{*,\w}_{m,n}\{\dxexit>0\}=Q^{*,\w}_{m+1,n}\{\dxexit>1\}$.
Note that in this equality we have dual measures defined in distinct rectangles 
$\Lambda_{m,n}$ and $\Lambda_{m+1,n}$.
Then appeal to Lemma \ref{lemQ*1}. The last inequality in  \eqref{xidist2}   is immediate.
\end{proof}

Recall the notations $\vva(j)$ and  $\vvb(j)$ defined in 
\eqref{defv}--\eqref{defv1}, and introduce their vertical
counterparts:  
\be
\wwa(i)=\min\{j\in\bZ_+:  \exists k: x_k=(i,j) \} \label{defw}\ee
and 
\be \wwb(i)=\max\{j\in\bZ_+:  \exists k: x_k=(i,j) \} \label{defw1}\ee
Implication $\vva(j)>k\Rightarrow \wwa(k)<j$ holds, and transposition 
(that is, reflection across the diagonal)
interchanges  $\vva$ and  $\wwa$. 
 Similar properties are valid for $\vvb$ and $\wwb$.

\begin{proposition}  Assume weight distributions \eqref{distr4}
and rectangle dimensions \eqref{mnN}. 
 Then  
\[  \lim_{\delta\searrow 0}\varlimsup_{N\to\infty}
\Pann\{1\le \xexit\le \delta N^{2/3}\} =0.\]
Same result holds for $\yexit$.  
\label{lb-prop1}\end{proposition}

\def\aaa{e^{\eta N^{1/3}}} 
\def\ttt{h} 

\begin{proof}  We prove the result for $\xexit$, and 
transposition gives it for $\yexit$. 
Take $\delta>0$ small and abbreviate $u=\fl{\delta N^{2/3}}$. 
By Fatou's lemma, it is enough to show that for all $0<\ttt<1$,
\be
\lim_{\delta\searrow 0}\varlimsup_{N\to\infty}
  \P\bigl[\, Q(0< \xexit\le u) >\ttt \, \bigr]=0. 
\label{lbaux1}\ee

 Fix a small $\eta>0$. Decompose the probability as follows. 
 \begin{align}
&\P\bigl[\, Q(0< \xexit\le u) >\ttt  \, \bigr]
=  \P\biggl[\;   Z_{m,n}(0<\xexit\le u)  
>\ttt \,Z_{m,n}   \;\biggr] \nn\\
&\le  \P\biggl[\;   Z_{m,n}(0<\xexit\le u)  
>\ttt \,Z_{m,n}(\xexit> u)   \;\biggr] \nn\\[4pt]
&= \P\biggl[\; \frac{ Z_{m,n}(0<\xexit\le u) }
{   \Zalt_{(1,1),(m,n)}}    >\ttt   \, \frac{ Z_{m,n}(\xexit> u)}{ \Zalt_{(1,1),(m,n)}} 
  \;\biggr] \nn\\[4pt]
&\le \P\biggl[\; \frac{Z_{m,n}(\xexit> u) }
{   \Zalt_{(1,1),(m,n)}} <\aaa \;\biggr]  \label{lbprob1}\\[4pt]
&\qquad 
+  \P\biggl[\; \frac{ Z_{m,n}(0<\xexit\le u)}{ \Zalt_{(1,1),(m,n)}} 
>  {\ttt\aaa}   \;\biggr]. 
 \label{lbprob2} \end{align}
We show separately that for small $\delta$, $\eta$ can be
chosen so that  probabilities 
\eqref{lbprob1} and \eqref{lbprob2} are asymptotically small. 

\bigskip
 
 \noindent
{\bf Step 1: Control of probability \eqref{lbprob1}. } 

\medskip

We begin with a general coupling lemma. Its proof shows 
 that it does not depend on any  
particular weight distribution.

\begin{lemma}
For each fixed $\w$, $  Q^{\w}_{m_1,n}(\xexit>0) \le   Q^{\w}_{m_2,n}(\xexit>0)  $  
for all $0<m_1<m_2$ and $n\ge 0$.  
\label{coupl-lm} \end{lemma}

\begin{proof} Fix $\w$.  We construct a coupling of polymer paths.  
 On the full lattice $\bZ_+^2$ define a backward Markov kernel 
\be 
\backpi_{x,x-e}= \frac{Y_{x}Z_{x-e}}{Z_{x}}
=  \frac{Z_{x-e}}{Z_{x-e_1}
+ Z_{x-e_2}}\,, \qquad x\in\bN^2,\; e\in\{e_1,e_2\},  \label{backpi}\ee
with the obvious degenerate transitions $\backpi_{(i,0),(i-1,0)}=\backpi_{(0,j),(0,j-1)}=1$ 
on the axes and absorption $\backpi_{(0,0),(0,0)}=1$ at the origin.  
For each $x\in\bZ_+^2\smallsetminus\{(0,0)\}$ pick a jump to $v(x)\in\{x-e_1,x-e_2\}$  according to these
transition probabilities.   Fix an endpoint $(m,n)$.  Construct a 
  path $x_{0,m+n}$  from the origin to  $(m,n)$ backwards,  beginning with $x_{m+n}=(m,n)$
 and then iterating   
$x_{k}=v(x_{k+1})$ for $k=m+n-1, m+n-2, \dotsc, 0$. The process ends at $x_0=0$. The probability of the path is 
\[  \prod_{k=1}^{m+n} \backpi_{x_k,x_{k-1}} =  \frac1{Z_{m,n}}\prod_{k=1}^{m+n} Y_{x_k}  =     
Q^\om_{m,n}(x_{0,m+n}). \] 
In other words, specifying the jumps $\{v(x)\}$ constructs a simultaneous 
realization of the polymer paths under all quenched measures $Q^\om_{m,n}$ for a fixed $\w$. 

Suppose $m_1<m_2$ and the path between the origin and $(m_1,n)$ 
goes through the point $(1,0)$.  Then the same is true for the path 
between the origin and $(m_2,n)$. This is because the path from $(m_2,n)$
cannot reach $(0,1)$ without intersecting the path from $(m_1,n)$, and 
once they intersect they   merge by the construction. \end{proof} 

 Turning to probability \eqref{lbprob1}, first decompose according
to the value of $\xexit$: 
\begin{align*}
&\frac{ Z_{m,n}(\xexit> u)}{ \Zalt_{(1,1),(m,n)}}
=\sum_{k=u+1}^m \biggl(\;\prod_{i=1}^k U_{i,0} \biggr)\cdot
 \frac{ \Zalt_{(k,1),(m,n)} }{ \Zalt_{(1,1),(m,n)}}. 
\end{align*} 
Construct a new system $\wt\w$  in the rectangle $\Lambda_{m,n}$.
Fix a parameter $a>0$ that we will take large in the end.   The interior
weights of $\wt\w$ are  $Y^{\wt\w}_{i,j}=Y_{m-i+1,n-j+1}$
for $(i,j)\in\{1,\dotsc,m\}\times\{1,\dotsc,n\}$.  
The boundary weights
$\{U_{i,0}^{\wt\w}, V_{0,j}^{\wt\w}\}$ obey the standard setting \eqref{distr4} 
with  a new parameter $\lambda=\theta-aN^{-1/3}$ (but $\mu$ stays 
fixed), 
and they are independent of 
  the old weights $\w$.  
Define new dimensions for a rectangle by  
\[ (\bar m, \bar n)= 
\bigl(m+ \fl{N\trigamf(\mu-\lambda)}-\fl{N\trigamf(\mu-\theta)}\, ,\, 
n+\fl{N\trigamf(\lambda)}-\fl{N\trigamf(\theta)}\bigr).\] 
We have the bounds 
\[  \bar n -n = \fl{N\trigamf(\lambda)} - \fl{N\trigamf(\theta)} \ge a\abs{\trigamf'(\theta)} N^{2/3}-1 
\ge c_1 aN^{2/3} \]
for a constant $c_1=c_1(\theta)$, and 
\[  \bar u=m-\bar m = \fl{N\trigamf(\mu-\theta)} - \fl{N\trigamf(\mu-\lambda)} 
\ge a\abs{\trigamf'(\mu-\lambda)} N^{2/3}-1 \ge bN^{2/3} 
 \]
for another constant $b$. By taking $a$ large enough we can
guarantee that $b>\delta$. (It is helpful to remember here
that $\trigamf'<0$ and  $\trigamf''>0$.)

By \eqref{Zcomp2}  and \eqref{Z2}, 
\begin{align*}
 &\frac{ \Zalt_{(k,1),(m,n)} }{ \Zalt_{(1,1),(m,n)}} =  
 \frac{ Z^{\, {\scriptscriptstyle\square}, \,\wt\w}_{(1,1),(m-k+1,n)} }
 { Z^{\, {\scriptscriptstyle\square},\, \wt\w}_{(1,1),(m,n)}}
 \ge \frac{ Z^{\wt\w}_{m-k+1,n}(\xexit>0) }{ Z^{\wt\w}_{m,n}(\xexit>0) }  \\
&= \frac{ Q^{\wt\w}_{m-k+1,n}(\xexit>0) Z^{\wt\w}_{m-k+1,n}}
 { Q^{\wt\w}_{m,n}(\xexit>0) Z^{\wt\w}_{m,n} } 
 \ge Q^{\wt\w}_{m-k+1,n}(\xexit>0) \biggl( \;\prod_{i=1}^{k-1} U^{\wt\w}_{m-i+1,n} \biggr)^{-1}.
\end{align*} 
After these transformations, 
\[
\text{\eqref{lbprob1}} \le \P\biggl[  U_{1,0}
\sum_{k=u+1}^m \biggl(\;\prod_{i=2}^k \frac{U_{i,0}}{U^{\wt\w}_{m-i+2,n}} \biggr) 
Q^{\wt\w}_{m-k+1,n}(\xexit>0) <\aaa \biggr]. 
\]
Inside this probability $\{U_{i,0}\}$ are independent 
of $\wt\w$.  Restrict the sum in the 
probability to $k\le \bar u$ and apply  Lemma \ref{coupl-lm}. 
This  turns the bound 
  above into 
  \begin{align}
\text{\eqref{lbprob1}} &\le
\P\biggl[ Q^{\wt\w}_{m-\bar u+1,n}(\xexit>0) \,  U_{1,0}
\sum_{k=u+1}^{\bar u} \biggl(\;\prod_{i=2}^k \frac{U_{i,0}}{U^{\wt\w}_{m-i+2,n}} \biggr) <\aaa \biggr]\nn\\[3pt]
&\le \P\Bigl[\, Q^{\wt\w}_{m-\bar u+1,n}(\xexit>0)\le \tfrac12 \,\Bigr] 
\label{lbprob4} \\[3pt]
&\qquad + 
\P\biggl[ \,  U_{1,0}
\sum_{k=u+1}^{\bar u} \biggl(\;\prod_{i=2}^k \frac{U_{i,0}}{U^{\wt\w}_{m-i+2,n}} \biggr) \le2{\aaa}\,\biggr].
\label{lbprob5}
\end{align}

 We treat first probability \eqref{lbprob4}.
 Apply the distribution-preserving reversal 
$\wt\w\mapsto \wt\w^*$, recall \eqref{Q*2}, and use the 
definition  \eqref{Q*1} of the dual measure to write 
\begin{align*}   Q^{\wt\w}_{m-\bar u+1,n}(\xexit>0)
\overset{d}=  Q^{*,\wt\w}_{m-\bar u+1,n}(\dxexit>0)
=Q^{*,\wt\w}_{m,n}(\dxexit\ge \bar u). 
\end{align*}
Going over to complements, 
\[
\eqref{lbprob4} = \P\Bigl[\, Q^{*,\wt\w}_{m,n}\{\dxexit< \bar u\}
> \tfrac12 \,\Bigr]. \]
We claim that 
\be  Q^{*,\wt\w}_{m,n}\{\dxexit\le  \bar u\} = Q^{*,\wt\w}_{\bar m, \bar n}\{\dyexit>\bar n -n\}. 
\label{lbaux2.9}\ee
Equality \eqref{lbaux2.9} comes from the next computation that utilizes the Markov property
\eqref{dualmarkov} of the dual measure. In the rectangle $\Lambda_{m,n}$
  event $\{\dxexit\le  \bar u\}$ says that
the path does not touch the segment $\{0,\dotsc, \bar m-1\}\times\{n\}$. 
Consequently the path uses one of the edges $((\bar m-1,\ell),  (\bar m,\ell))$ 
for $0\le \ell<n$. 
\begin{align*}
&Q^{*,\wt\w}_{m,n}\{\dxexit\le  \bar u\} = \sum_{\ell=0}^{n-1} 
Q^{*,\wt\w}_{m,n}\{  x_{\bar m+\ell-1}=(\bar m-1,\ell), \, x_{\bar m+\ell}=(\bar m,\ell)\}\\
&= \sum_{\ell=0}^{n-1} \sum_{x_\centerdot\in\Pi_{\bar m-1,\ell}} 
\biggl(\,\prod_{k=0}^{\bar m+\ell-1} X^{\wt\w}_{x_k} \biggr)\frac1{Z^{\wt\w}_{\bar m, \ell}}
= \sum_{\ell=0}^{n-1} \sum_{x_\centerdot\in\Pi_{\bar m-1,\ell}} 
\biggl(\,\prod_{k=0}^{\bar m+\ell-1} X^{\wt\w}_{x_k} \biggr)
\biggl(\,\prod_{j=\ell}^{\bar n-1} X^{\wt\w}_{\bar m, j} \biggr)\frac1{Z^{\wt\w}_{\bar m, \bar n}}\\
&= Q^{*,\wt\w}_{\bar m, \bar n}\{\dyexit>\bar n -n\}. 
\end{align*}  
The second-last equality above relied on the convention 
$X^{\wt\w}_{\bar m, j} =V^{\wt\w}_{\bar m, j+1}$ for the dual variables defined in the rectangle
$\Lambda_{\bar m, \bar n}$.  This checks \eqref{lbaux2.9}. 
Now appeal to Lemma \ref{auxlm7},
for $N\ge 1$ and  large enough $a$  to ensure 
$e^{-\delta(c_1a)^2N^{1/3}}\le 1/2$:  
\be\begin{aligned}  \eqref{lbprob4} &\le  \P\Bigl[\,  
Q^{*,\wt\w}_{\bar m, \bar n}\{\dyexit> c_1aN^{2/3}\}  \ge\tfrac12 \,\Bigr]\\
&= \P\Bigl[\,  
Q^{\wt\w}_{\bar m, \bar n}\{\yexit> c_1aN^{2/3}\}  \ge\tfrac12\,\Bigr]
\; \le \; C(\theta)a^{-3}.  
\end{aligned} \label{lbaux3}\ee

To treat probability \eqref{lbprob5},  let $A_i=U_{i+1,0}^{-1}\sim$ Gamma($\theta, 1$)
and $\wt A_i=(U^{\wt\w}_{m-i+1,n})^{-1}\sim$ Gamma($\lambda, 1$) so that we can
write 
\begin{align*}
\eqref{lbprob5} &= \P\biggl[ \;   
\sum_{k=u}^{\bar u-1} \biggl(\;\prod_{i=1}^k \frac{\wt A_{i}}{A_i} \biggr) 
\le2{\aaa A_0} \,\biggr]\\
&\le  \P\biggl[ \;   \sup_{u\le k<\bar u} \exp\Bigl\{ \, \sum_{i=1}^k (\log \wt A_i-\log A_i)\Bigr\}
\le 2{\aaa A_0}\,\biggr]. 
\end{align*}
We approximate the sum in the exponent by a Brownian motion.  Compute the mean: 
\[  \E (\log \wt A_i-\log A_i) = \digamf(\lambda)-\digamf(\theta) \ge -a_1N^{-1/3} \]
for a positive constant $a_1\approx \trigamf(\theta)a$.
(Recall that $\trigamf=\digamf'>0$.)  Define a continuous path 
$\{S_N(t): t\in\bR_+\}$ by 
\[  S_N(kN^{-2/3}) = N^{-1/3} \sum_{i=1}^{k} 
 (\log \wt A_i-\log A_i -\E\log \wt A_i+\E\log A_i ),
\quad k\in\bZ_+,   \]
and by linear interpolation.  
Then rewrite the probability from above:
 \begin{align*}
\eqref{lbprob5} &\le \P\Bigl[ \; \sup_{\delta\le t\le b} \bigl( S_N(t)-ta_1\bigr) 
\le \eta+ N^{-1/3}\log 2A_0 \,\Bigr]. 
\end{align*}
As $N\to\infty$, $S_N$ converges to a Brownian motion $B$ and so
\be
\varlimsup_{N\to\infty}  \eqref{lbprob5} 
\le \mP\Bigl[ \; \sup_{\delta\le t\le b} \bigl( B(t)-ta_1\bigr) 
\le \eta\Bigr] \searrow 0 \quad\text{as $\delta, \eta\searrow 0$.} 
\label{lbaux4}\ee

Combining \eqref{lbaux3} and \eqref{lbaux4} shows that, given $\e>0$,
we can first pick $a$ large enough to have 
$\varlimsup_{N\to\infty}\eqref{lbprob4}\le \e/2$.  Fixing $a$ fixes
$a_1$, and then we fix $\eta$ and $\delta$ small enough to 
have $\varlimsup_{N\to\infty}  \eqref{lbprob5} \le\e/2$. 
This is possible because $\sup_{0< t\le b}( B(t)-ta_1)$ is
a strictly positive random variable by the law of the iterated 
logarithm. 
Together these give $\varlimsup_{N\to\infty}  \eqref{lbprob1} \le \e$.

\bigskip
 
 \noindent
{\bf Step 2: Control of probability \eqref{lbprob2}. } 

\medskip

For later use we prove a  lemma that gives more than presently 
needed. 


\begin{lemma}  Assume weight distributions \eqref{distr4} with parameters $0<\theta<\mu$ 
and rectangle dimensions \eqref{mnN} with parameter $\gamma> 0$.  Let $a, b, s>0$. 

{\rm (i)} Let  $0<\e<1$. There 
exists a constant $C=C(\theta, \mu,  \gamma)<\infty$ such that, if
\be b\ge C\e^{-1/2}(a +\sqrt{a}\,),  
\label{lb-aux-condb}\ee
 then  
\be \varlimsup_{N\to\infty}
 \P\biggl[  \; \frac{ Z_{m,n}(0<\xexit\le aN^{2/3})}{ \Zalt_{(1,1),(m,n)}} 
\ge s e^{bN^{1/3}}\,\biggr]  
\le \e. \label{lbauxb10}\ee

{\rm (ii)}     There exist  finite positive  constants $N_0$, $b_0$ 
and $C$   that can depend on  $(\theta, \gamma, s)$, 
such that,  for $N\ge N_0$ and $b\ge b_0$, 
\be  \P\biggl[  \; \frac{ Z_{m,n}(0<\xexit\le \sqrt bN^{2/3})}{ \Zalt_{(1,1),(m,n)}} 
\ge s e^{bN^{1/3}}\,\biggr]  \le  Cb^{-3/2}.  
  \label{lbauxb10.1}\ee
 \label{lb-aux-lm3}\end{lemma}
 
  The key technical point of part (ii)  is that  $C$ and $N_0$ do not depend on $b$. Their  dependence on other constants is harmless.

 \begin{proof}   We begin with that segment  of the proof  that serves both parts (i) and (ii) of the lemma.   Let $u=\fl{aN^{2/3}}$. 
 First decompose.  
\begin{align}\label{lb-ratio4}
\frac{ Z_{m,n}(0<\xexit\le u)}{ \Zalt_{(1,1),(m,n)}}
=\sum_{k=1}^u \biggl(\;\prod_{i=1}^k U_{i,0} \biggr)\,
 \frac{ \Zalt_{(k,1),(m,n)} }{ \Zalt_{(1,1),(m,n)}}\, . 
\end{align} 
Construct a new environment $\wt\w$  in the rectangle $\Lambda_{m,n}$.  The interior
weights of $\wt\w$ are  $Y^{\wt\w}_{i,j}=Y_{m-i+1,n-j+1}$.  The new boundary weights
$\{U_{i,0}^{\wt\w}, V_{0,j}^{\wt\w}\}$ are independent of 
  the old weights $\w$
  and they obey 
  a new parameter $\lambda=\theta+rN^{-1/3}$ with $r>0$.   For  part (i)    $r$ can be arbitrarily large because we let $N\to\infty$.  For part (ii) we need to be careful about the acceptable pairs $(N,r)$ because an admissible parameter $\lambda$ must satisfy  $\lambda<\mu$.  
  
    By \eqref{Zcomp2}  and \eqref{Z2}, 
\begin{align*}
 &\frac{ \Zalt_{(k,1),(m,n)} }{ \Zalt_{(1,1),(m,n)}} =  
 \frac{ Z^{\, {\scriptscriptstyle\square}, \,\wt\w}_{(1,1),(m-k+1,n)} }
 { Z^{\, {\scriptscriptstyle\square},\, \wt\w}_{(1,1),(m,n)}}
 \le \frac{ Z^{\wt\w}_{m-k+1,n}(\yexit>0) }{ Z^{\wt\w}_{m,n}(\yexit>0) }  \\
&= \frac{ Q^{\wt\w}_{m-k+1,n}\{\yexit>0\} \, Z^{\wt\w}_{m-k+1,n}}
 { Q^{\wt\w}_{m,n}\{\yexit>0\}\, Z^{\wt\w}_{m,n} } 
 \le   \frac1{ Q^{\wt\w}_{m,n}\{\yexit>0\} } \,\biggl( \;\prod_{i=1}^{k-1} U^{\wt\w}_{m-i+1,n} \biggr)^{-1}. 
\end{align*} 
  Write  
$A_i=U_{i+1,0}^{-1}\sim$ Gamma($\theta, 1$)
and $\wt A_i=(U^{\wt\w}_{m-i+1,n})^{-1}\sim$ Gamma($\lambda, 1$).   
\begin{align}
\text{probability in \eqref{lbauxb10}} &\le  \P\biggl[ \; \frac{U_{1,0}} { Q^{\wt\w}_{m,n}\{\yexit>0\} } 
\sum_{k=1}^u \biggl(\;\prod_{i=2}^k \frac{U_{i,0}}{U^{\wt\w}_{m-i+2,n}} \biggr)  \ge
s e^{bN^{1/3}} \; \biggr]\nn\\[5pt]
&\le  
 \P\bigl[  \, Q^{\wt\w}_{m,n}\{\yexit>0\} <\tfrac12 \, \bigr] \label{lbprobb10} \\
&\qquad +  \P\biggl[ \;  A_0^{-1} 
\sum_{k=1}^u \biggl(\;\prod_{i=1}^{k-1} \frac{\wt A_i}{A_i}\, \biggr)  
\ge \tfrac12 s e^{bN^{1/3}} \; \biggr].  \label{lbprobb11}
\end{align} 

To treat  the probability in \eqref{lbprobb10}, define 
a new scaling parameter $M=n/\trigamf(\lambda)$
and  new rectangle dimensions  
\[  (\bar m,\bar n) 
= \bigl( \fl{M\trigamf(\mu-\lambda)} \,,\, n\bigr) = \bigl( \fl{M\trigamf(\mu-\lambda)} \,,\, M\trigamf(\lambda)\bigr). \]
The upper bound Lemma \ref{auxlm7} is valid for $\lambda$ and $(\bar m,\bar n)$ with $\mnc_M=1$.    $1/2\le M/N\le 2$ for $N\ge N_1(\theta,\gamma)$. 
\begin{align*}
\bar m -m &=  \fl{M\trigamf(\mu-\lambda)} - N\trigamf(\mu-\theta) -\gamma N^{2/3}  \\
&\ge  M\trigamf(\mu-\lambda)   - M \frac{\trigamf(\lambda)\trigamf(\mu-\theta)}{ \trigamf(\theta)}  \;-\;  \frac{\trigamf(\mu-\theta)}{\trigamf(\theta)} \gamma N^{2/3}  -\gamma N^{2/3}  -1  \\
&=\frac{M}{\trigamf(\theta)} \bigl[   \trigamf(\theta) \trigamf(\mu-\lambda) - 
\trigamf(\lambda)\trigamf(\mu-\theta)  \bigr]   - C_1(\theta,\mu,\gamma) M^{2/3}\\
&=\frac{M}{\trigamf(\theta)} \bigl[ -  \trigamf(\theta)\Psi_2(\rho_1) - 
 \trigamf(\mu-\theta) \Psi_2(\rho_2) \bigr] (\lambda-\theta)   - C_1(\theta,\mu,\gamma) M^{2/3}\\
 &\ge M^{2/3}\bigl[  C_2(\theta,\mu)r  -  C_1(\theta,\mu,\gamma)  \bigr]  .  
\end{align*}

Thus there exists a constant $c_2=c_2(\theta,\mu,\gamma)>0$ such that 
\[ \bar m -m   \ge c_2 r M^{2/3} \]
provided 
\[ r\ge  r_0(\theta,\mu,\gamma) = 2C_1(\theta,\mu,\gamma)/C_2(\theta,\mu), \]   $N\ge N_1(\theta,\gamma)$,  and $\lambda$ is restricted to (say)  $[\theta, (\theta+\mu)/2]$  (which requires $N$ large enough relative to $r$). 

Consider the complement $\{\xexit>0\} $ of the inside event in 
\eqref{lbprobb10}.  Apply
$\wt\w\mapsto\wt\w^*$, and use the definition \eqref{Q*1} of the dual measure
 to go from $\Lambda_{m,n}$
to the larger rectangle $\Lambda_{\bar m,n}=\Lambda_{\bar m,\bar n}$
\begin{align*}
Q^{\wt\w^*}_{m,n}\{\xexit>0\} &= Q^{*, \wt\w}_{m,n}\{\dxexit>0\} 
= Q^{*, \wt\w}_{\bar m,n}\{\dxexit> \bar m-m\}
\le  Q^{*, \wt\w}_{\bar m,\bar n}\{\dxexit> c_2 r M^{2/3}\}. 
\end{align*}
We can fix the lower bounds on $r$ and $N$ large enough so that 
\be  e^{-\delta(c_2r)^2M^{1/3}}\le \tfrac12 
\label{lbauxb10.5}\ee
Then by Lemma \ref{lemQ*1} and Lemma \ref{auxlm7},  
\be\begin{aligned} 
\eqref{lbprobb10}&=  \P\bigl[  \, Q^{\wt\w}_{m,n}\{\xexit>0\} >\tfrac12 \, \bigr]
\le   \P\bigl[  \, Q^{*, \wt\w}_{\bar m,\bar n}\{\dxexit> c_2r M^{2/3}\}
  >\tfrac12 \, \bigr]\\
&= \P\bigl[  \, Q^{\wt\w}_{\bar m,\bar n}\{\xexit> c_2r M^{2/3}\}  
> \tfrac12\, \bigr]
\le Cr^{-3}. 
\end{aligned} \label{lbauxb11}\ee

For probability \eqref{lbprobb11} we 
rewrite the event in terms  of mean zero i.i.d's.   Compute the mean: 
\[  \E (\log \wt A_i-\log A_i) = \digamf(\lambda)-\digamf(\theta) \le r_1N^{-1/3} \]
for a positive constant $r_1= \trigamf(\theta)r$.  Let 
\[  S_k =   \sum_{i=1}^{k} 
 (\log \wt A_i-\log A_i -\E\log \wt A_i+\E\log A_i ).
   \]
By Kolmogorov's  inequality,  
\begin{align}
\eqref{lbprobb11} &\le 
\P\Bigl[ \; \sup_{0\le k\le u} S_k 
\ge bN^{1/3} -  r_1aN^{1/3} +
\log\frac{s A_0}{2aN^{2/3}}   \, \Bigr]\nn\\
&\le \P\Bigl[ \; \sup_{0\le k\le u} S_k 
\ge bN^{1/3} -  r_1aN^{1/3} +
\log\frac{s b_1}{2aN^{2/3}} \, \Bigr]  +\P(A_0<b_1) \nn\\
&\le\frac{ \E(S_u^2)}
{\bigl(bN^{1/3} -  r_1aN^{1/3} +
\log\frac{s b_1}{2aN^{2/3}}\bigr)^2}   
+ \int_0^{b_1} \frac{x^{\theta-1} e^{-x}}{\Gamma(\theta)}\,dx\nn\\
&\le  \frac{C a}
{\bigl(b -  r_1a +
N^{-1/3}\log\frac{s b_1}{2aN^{2/3}}\bigr)^2}   
+  Cb_1^\theta, 
\nn\end{align}
assuming that the quantity inside the parenthesis in the
denominator is positive. 
Collecting the bounds from \eqref{lbauxb11} 
and above  we have, provided \eqref{lbauxb10.5} holds,  
\begin{align}
& \P\biggl[  \; \frac{ Z_{m,n}(0<\xexit\le aN^{2/3})}{ \Zalt_{(1,1),(m,n)}} 
\ge c e^{bN^{1/3}}\,\biggr]   \nn\\
 &\quad \le 
\frac{C}{r^3} +  \frac{C a}
{\bigl(b -  r_1a +
N^{-1/3}\log\frac{s b_1}{2aN^{2/3}}\bigr)^2}   
+  Cb_1^\theta. \label{lbauxb15} 
\end{align}

We prove  statement (i) of the lemma.    Choose
  $r=(3C/\e)^{1/3}$ and 
 $b_1=(\e/(3C))^{1/\theta}$ for a large enough constant $C$. 
Then by assumption \eqref{lb-aux-condb},
\[
\varlimsup_{N\to\infty} \P\biggl[  \; \frac{ Z_{m,n}(0<\xexit\le aN^{2/3})}{ \Zalt_{(1,1),(m,n)}} 
\ge s e^{bN^{1/3}}\,\biggr] 
\le \frac{2\e}3 + \frac{Ca}{(b-r_1a)^2}\; \le\; \e. 
\]

\medskip

We turn to  
  statement (ii).      Fix $\e_0(\theta)>0$ small enough so that $\Psi_1(\theta)\e_0(\theta)<1/4$.
  Recall that   $r_1= \trigamf(\theta)r$.   Set  $a=\sqrt b$,   $r=\e_0(\theta)b^{1/2}$
and $b_1=b^{-3/(2\theta)}$.  Then    $b\ge b_0(\theta, \mu, \gamma)$ guarantees that  $r\ge r_0(\theta, \mu, \gamma)$ as required   for \eqref{lbauxb11} above.   This and a large enough lower bound 
  $N\ge N_0(\theta, \mu, \gamma,  s)$ guarantee that  the long denominator on line 
 \eqref{lbauxb15}  is $\ge (b/2)^2$ and the entire bound becomes 
\be \P\biggl[  \; \frac{ Z_{m,n}(0<\xexit\le \sqrt bN^{2/3})}{ \Zalt_{(1,1),(m,n)}}  \biggr] 
\le   Cb^{-3/2} 
\label{lbauxb16}\ee   
which is exactly the goal \eqref{lbauxb10.1}.   

The arguments that brought us to this point are valid as long as the perturbed parameter $\lambda=\theta+rN^{-1/3}=\theta+\e_0(\theta)b^{1/2}N^{-1/3}$ satisfies $\lambda  \le (\theta+\mu)/2$ (that  is, stays bounded away and below $\mu$).   Hence   bound \eqref{lbauxb10.1} has been proved for 
\[
b_0(\theta, \mu, \gamma)\le  b\le  \tfrac14\e_0(\theta)^{-2} (\mu-\theta)^2N^{2/3} \]   
and  $N\ge N_0(\theta, \mu, \gamma,  s)$.
 

To finish the proof of   statement (ii) we give a separate argument for \eqref{lbauxb10.1} for the case 
$ b\ge \tfrac14 \e_0(\theta)^{-2} (\mu-\theta)^2N^{2/3}$. 
 Let $C_0(\theta)=\Psi_0(\mu)-\Psi_0(\theta)$.  
 Previously when $\e_0(\theta)$ was fixed, we can fix  it  
  small enough to guarantee  $ \e_0(\theta)^{-1} (\mu-\theta)\ge 8C_0(\theta)$.   Then we are in the case 
  \be\label{lbaux16.5} b\ge 16C_0(\theta)^2N^{2/3}. \ee
  
  We return to the beginning to treat   ratio \eqref{lb-ratio4} differently.   Transposing the first inequality of \eqref{Zcomp2}  gives the inequality 
\be \frac{Z_{m,n}(\xexit>0)}{Z_{m,n-1}(\xexit>0)} \le 
\frac{\Zalt_{(1,1), (m,n)}}{\Zalt_{(1,1),(m,n-1)}}
 \nn\ee
which, with the left-hand side partly expanded,  reads as   
\be \frac{\sum_{k=1}^m \bigl( \;\prod_{i=1}^k  U_{i,0} \bigr)\Zalt_{(k,1),(m,n)}}{\sum_{k=1}^m \bigl( \;\prod_{i=1}^k  U_{i,0} \bigr)\Zalt_{(k,1),(m,n-1)}} \le 
\frac{Z_{(1,1), (m,n)}}{Z_{(1,1),(m,n-1)}}.
 \nn\ee
This statement is a consequence of algebraic relations and hence valid for all positive weights. Consequently we can let $U_{i,0}\to 0$ for $i>u$ to obtain the inequality 
\be \frac{Z_{m,n}(0<\xexit\le u)}{Z_{m,n-1}(0<\xexit\le u)} \le 
\frac{\Zalt_{(1,1), (m,n)}}{\Zalt_{(1,1),(m,n-1)}}.
 \nn\ee
 Rewrite it,   iterate it to drive the $n$-coordinate all the way down to 1, and expand in terms of weights: 
 \begin{align*}
 &\frac{Z_{m,n}(0<\xexit\le u)}{\Zalt_{(1,1), (m,n)}} \le 
\frac{Z_{m,n-1}(0<\xexit\le u)}{\Zalt_{(1,1),(m,n-1)}} \le 
\frac{Z_{m,n-2}(0<\xexit\le u)}{\Zalt_{(1,1),(m,n-2)}} \\
&\qquad \le \dotsm \le 
\frac{Z_{m,1}(0<\xexit\le u)}{\Zalt_{(1,1),(m,1)}} 
=    \sum_{k=1}^u \;   U_{k,0} \prod_{i=1}^{k-1} \frac{U_{i,0}}{Y_{i,1}} 
=    \sum_{k=1}^u   e^{S_k} \\
 &\qquad \le   u \exp\Bigl[ \max_{1\le k\le u}S_k \Bigr]
 \end{align*}
where we defined  
 \[   S_k=\log  U_{k,0}+ \sum _{i=1}^{k-1}  (\log U_{i,0} -\log Y_{i,1}),\]  
 a sum of  independent terms with   mean   $\E S_k=\digamf(\theta)+(k-1)C_0(\theta)$.
 
With these preliminaries,   with   $u=\sqrt bN^{2/3}$, 
\begin{align*}
\text{probability in \eqref{lbauxb10.1}} &\le 
\P\Bigl[ \; \max_{0\le k\le u} S_k 
\ge bN^{1/3}   +    \log\frac{s }{u}   \, \Bigr] \\
&\le   \P\Bigl[ \; \max_{0\le k\le u} (S_k-\E S_k)   \ge \tfrac34 bN^{1/3}  -  (u-1)C_0(\theta) -\abs{\digamf(\mu)} \, \Bigr]   \\
&\le \frac{Cu}{(\tfrac14 bN^{1/3})^2 }  \le Cb^{-3/2}.   
\end{align*}
Above   $\abs{  \log{s}/{u}\, }\le \tfrac14 bN^{1/3}$ 
  for $N\ge N_2(s)$ because this forces  $b$ also large due to  \eqref{lbaux16.5}.  
 The second last inequality is  Kolmogorov's inequality   as was done above, together
with $uC_0(\theta)   \le \tfrac14 bN^{1/3}$ 
which is equivalent to   \eqref{lbaux16.5} and $ \abs{\digamf(\mu)} \le \tfrac14 bN^{1/3}$ which is true for large enough $N$. 
We have proved \eqref{lbauxb10.1} for the case $b\ge 16C_0(\theta)^2N^{2/3}$
and $N\ge N_0(\theta, \mu, \gamma, s)$.  
This concludes the proof of  Lemma \ref{lb-aux-lm3}. 
\end{proof}


Now apply part (i) of Lemma \ref{lb-aux-lm3} 
 with $a=\delta$ and $b=\eta$  to show 
\[  \varlimsup_{N\to\infty} 
 \P\biggl[\; \frac{ Z_{m,n}(0<\xexit\le \delta N^{2/3})}{ \Zalt_{(1,1),(m,n)}} 
>{\ttt\aaa}  \;\biggr] \le \e. \]
Step 1 already fixed $b=\eta>0$ small. Given $\e>0$, we 
can then take $a=\delta$ small enough to satisfy 
\eqref{lb-aux-condb}. 
 Shrinking
$\delta$ does not harm  the conclusion from Step 1 because
the bound in \eqref{lbaux4} becomes stronger.  This concludes
Step 2.  

\medskip

To summarize, we have shown that if $\delta$ is small enough, then
\[  \varlimsup_{N\to\infty}
  \P\bigl[\, Q(0< \xexit\le \delta N^{2/3}) >\ttt \, \bigr]\le 2\e. \]
This proves \eqref{lbaux1} and thereby Proposition \ref{lb-prop1}. 
\end{proof}

From  Proposition \ref{lb-prop1} we extract the lower bound on 
the variance of $\log Z_{m,n}$. 

\begin{corollary}  Assume weight distributions \eqref{distr4}
and rectangle dimensions \eqref{mnN}.  Then there exists a constant 
$c$ such that  for large enough $N$, 
$ \Vvv^\theta[\log Z_{m,n}] \ge cN^{2/3} $.  
\end{corollary}

\begin{proof}  Adding equations \eqref{var1} and \eqref{var1.1} gives 
\[  \Vvv\bigl[\log Z_{m,n}\bigr] =   \Eann_{m,n}   \biggl[ \; \sum_{i=1}^\xexit   \functaa(\theta,  Y_{i,0}^{-1})   \biggr]
 +   \Eann_{m,n}  \biggl[ \; \sum_{j=1}^\yexit   
\functaa(\mu-\theta,  Y_{0,j}^{-1})   \biggr]. 
\] 
 Fix $\delta>0$ so that 
\[\Pann\{0<\xexit<\delta N^{2/3}\}+ \Pann\{0<\yexit<\delta N^{2/3}\}
<1/2 \]
for large $N$.  Then for a particular $N$ either 
$\Pann\{\xexit\ge\delta N^{2/3}\}\ge 1/4$ or
$\Pann\{\yexit\ge\delta N^{2/3}\}\ge 1/4$. 
Suppose it is $\xexit$. (Same argument for the other case.)
Abbreviate $\functaa_i=\functaa(\theta,  Y_{i,0}^{-1})$ and 
pick $a>0$ small enough so that for some constant $b>0$,
\[  \P\biggl[ \;\sum_{i=1}^{\fl{\delta N^{2/3}}} \functaa_i
<aN^{2/3} \;\biggr] \le e^{-bN^{2/3}} \quad\text{for $N\ge 1$.} \]
This is possible because $\{\functaa_i\}$ are 
strictly positive, i.i.d.\ random variables. 

It suffices now to prove that  for large $N$, 
\[ \Eann\biggl[ \;\sum_{i=1}^{\xexit} \functaa_i
\,\biggr]\ge \frac{a}8 N^{2/3}.  \]
This follows now readily: 
\begin{align*}
&\Eann\biggl[ \;\sum_{i=1}^{\xexit} \functaa_i
\,\biggr]\ge  \Eann\biggl[ \; \ind\{\xexit\ge\delta N^{2/3}\}
 \sum_{i=1}^{\fl{\delta N^{2/3}}} \functaa_i
\,\biggr]\\
&\ge aN^{2/3} \cdot  \Pann\biggl\{ \xexit\ge\delta N^{2/3} \,, \,
 \sum_{i=1}^{\fl{\delta N^{2/3}}} \functaa_i
\ge aN^{2/3} \biggr\} \\
&\ge aN^{2/3} \bigl(\tfrac14 -  e^{-bN^{2/3}} \bigr) 
\ge \frac{a}8 N^{2/3}. \qedhere
\end{align*}
\end{proof}

The corollary above concludes the proof of  Theorem \ref{var-bd-thm}.

\section{Fluctuations of the path in the model with boundaries} 
\label{sec:path-bd}
Fix two rectangles $\Lambda_{(k,\ell),(m,n)}\subseteq \Lambda_{(k_0,\ell_0),(m,n)}$,
with $0\le k_0\le k\le m$ and  $0\le \ell_0\le\ell\le n$.  As before define the 
partition function  $Z_{(k_0,\ell_0), (m,n)}$ and quenched polymer measure 
$Q_{(k_0,\ell_0), (m,n)}$ in the larger rectangle.  In the smaller rectangle 
$\Lambda_{(k,\ell),(m,n)} $
impose boundary conditions on the south and west boundaries,
given by the quantities $\{U_{i,\ell}, V_{k,j}: i\in\{k+1,\dotsc,m\},\, j\in\{ \ell+1,\dotsc, n\}\}$ 
computed in the larger rectangle as in \eqref{Z2}:
\be   U_{i,\ell}=\frac{Z_{(k_0,\ell_0), (i,\ell)}}{Z_{(k_0,\ell_0), (i-1,\ell)}}
\quad\text{and}\quad    V_{k,j}=\frac{Z_{(k_0,\ell_0), (k,j)}}{Z_{(k_0,\ell_0), (k,j-1)}}. 
\label{UVkl}\ee
Let $Z^{(k,\ell)}_{m,n}$ and $Q^{(k,\ell)}_{m,n}$ denote the partition function and 
quenched polymer measure in $\Lambda_{(k,\ell),(m,n)}$ under these boundary
conditions.  Then 
\be\begin{aligned}
 Z^{(k,\ell)}_{m,n} &=\sum_{s=k+1}^m\biggl(\, \prod_{i=k+1}^s U_{i,\ell} \biggr)\Zalt_{(s,\ell+1),(m,n)}
+ \sum_{t=\ell+1}^n\biggl(\, \prod_{j=\ell+1}^t V_{k,j} \biggr) \Zalt_{(k+1, t),(m,n)} \\
&= \frac{Z_{(k_0,\ell_0), (m,n)}}{Z_{(k_0,\ell_0), (k,\ell)}}.
\end{aligned}\label{Zkl1}\ee
For a  path  
$x_\centerdot \in\Pi_{(k,\ell),(m,n)}$ with  $x_1=(k+1,\ell)$, in other
words  $x_\centerdot$ takes off horizontally, 
\[  Q^{(k,\ell)}_{m,n}(x_\centerdot)  = \frac1{Z^{(k,\ell)}_{m,n}} 
\prod_{i=1}^{\xexitk{k}{\ell}} U_{k+i,\ell} \; \cdot \prod_{i=\xexitk{k}{\ell}+1}^{m-k+n-\ell} 
Y_{x_i}. \]
We wrote $\xexitk{k}{\ell}$ for the distance $x_\centerdot$ travels on the $x$-axis 
from the perspective of the new origin $(k,\ell)$: for $x_\centerdot \in\Pi_{(k,\ell),(m,n)}$
\be  \xexitk{k}{\ell}=
\max\{ r\ge 0:  \text{$x_i=(k+i,\ell)$ for $0\le i\le r$} \}.  \label{defxixk}\ee

Consider the distribution of $\xexitk{k}{\ell}$ under $Q^{(k,\ell)}_{m,n}$: 
adding up all the possible path segments  from $(k+r,\ell+1)$ to $(m,n)$ 
and utilizing \eqref{UVkl} and  \eqref{Zkl1}  gives
\be\begin{aligned}
&Q^{(k,\ell)}_{m,n}\{\xexitk{k}{\ell}=r\}   = \frac1{Z^{(k,\ell)}_{m,n}} 
\biggl(\, \prod_{i=k+1}^{k+r} U_{i,\ell} \biggr)\Zalt_{(k+r,\,\ell+1),(m,n)} \\
&\qquad =  \frac{Z_{(k_0,\ell_0), (k+r,\ell)} \Zalt_{(k+r,\,\ell+1),(m,n)} } {Z_{(k_0,\ell_0), (m,n)}}\\[3pt]
&\qquad =Q_{(k_0,\ell_0), (m,n)}\{ \text{$x_\centerdot$ goes through $(k+r,\ell)$ and 
 $(k+r,\,\ell+1)$} \}\\[3pt]
&\qquad =Q_{(k_0,\ell_0), (m,n)}\{ \vvb(\ell)=k+r\}.  
\end{aligned}\label{Zkl4}\ee
Thus $\xexitk{k}{\ell}$ under $Q^{(k,\ell)}_{m,n}$ has the same distribution as
$\vvb(\ell)-k$ under $Q_{(k_0,\ell_0), (m,n)}$.  We can now give the proof of 
Theorem \ref{path-bd-thm}. 

\begin{proof}[Proof of Theorem \ref{path-bd-thm}]
 If $\tau=0$ then the results 
are already contained in Corollary \ref{coraux3}
 and Proposition \ref{lb-prop1}.  Let us assume $0<\tau<1$. 

Set $u=\fl{bN^{2/3}}$.
Take $(k_0,\ell_0)=(0,0)$ and $(k,\ell)=(\fl{\tau m},\fl{\tau n})$ above. 
The system in the smaller rectangle $\Lambda_{(k,\ell),(m,n)}$ is   a  
system with boundary  distributions  \eqref{distr4} and dimensions
$(m-k,n-\ell)$ that satisfy  \eqref{mnN}
 for a new scaling parameter $(1-\tau)N$.
By  \eqref{Zkl4}, 
\be\begin{aligned}
  Q_{m,n}\{  \vvb(\fl{\tau n})\ge  \fl{\tau m}+u \}
&= Q^{(k,\ell)}_{m,n}\{\xexitk{k}{\ell} \ge u \} \\
&\overset{d}= Q_{m-k,n-\ell}\{\xexit \ge u \}. 
\end{aligned} \label{Zkl6}\ee
  Hence   bounds \eqref{aux8.6} and 
\eqref{case2c}  of Lemma \ref{auxlm7}  are valid as they stand for the 
quenched probability above. 
The part of  \eqref{path-bd-1} that pertains to 
  $\vvb(\fl{\tau n})$ now follows from Corollary \ref{coraux3}.  

Recall definition \eqref{defw1} of $\wwb$.  
To get control of the left tail of $\vva$, first note the
implication 
\begin{align*}
Q_{m,n}\{  \vva(\fl{\tau n})< \fl{\tau m}-u \} \le Q_{m,n}\{  \wwb(\fl{\tau m}-u)\ge \fl{\tau n} \}. 
\end{align*}
Let $ k=\fl{\tau m}-u$ and $\ell=\fl{\tau n}-\fl{nu/m}$.  
Then up to integer-part
corrections, $ k/\ell=m/n$.  For a constant $C(\theta)>0$,  
$\fl{\tau n}\ge \ell+C(\theta)bN^{2/3}$.
By \eqref{Zkl4}, applied to the vertical counterpart $\wwb$ of $\vvb$, 
\begin{align*}
Q_{m,n}\{  \wwb(\fl{\tau m}-u)\ge \fl{\tau n} \}
= Q^{(k,\ell)}_{m,n}\{\yexitk{k}{\ell}\ge b_1N^{2/3}\}
 \overset{d}= Q_{m-k,n-\ell}\{\yexit\ge C(\theta)bN^{2/3}\}. 
\end{align*}
The part of  \eqref{path-bd-1} that pertains to 
  $\vva(\fl{\tau n})$ now follows from Corollary \ref{coraux3},
applied to $\yexit$.  

\medskip

Last we   prove \eqref{path-bd-2}.   
  By a calculation similar to \eqref{Zkl4},    the event of passing through
a given edge at least one of whose endpoints lies in the interior  of $\Lambda_{(k,\ell),(m,n)}$ 
has the same probability under $Q^{(k,\ell)}_{m,n}$ and under $Q_{m,n}$. 
 Put   $(k,\ell)
=(\fl{\tau m}-2\fl{\delta N^{2/3}}, \, \fl{\tau n}-2\fl{c\delta N^{2/3}})$
 where the constant $c$ is picked so that 
$c>m/n$ for large enough $N$. If the path $x_\centerdot$ 
comes within distance   $\delta N^{2/3}$ of $(\tau m, \tau n)$, then 
it necessarily enters the rectangle 
$\Lambda_{(k+1,\ell+1), (k+4\fl{\delta N^{2/3}}  ,\,    \ell+ 4\fl{c\delta N^{2/3}})}$  
 through the south or the west side.
This event of entering  decomposes into a disjoint union according to the
unique edge that is used to enter the rectangle, and consequently the
probabilities under $Q^{(k,\ell)}_{m,n}$ and $Q_{m,n}$ are again the same. 
From the perspective of the polymer model $Q^{(k,\ell)}_{m,n}$, this event 
implies that either $0<\xexitk{k}{\ell}\le 4\delta N^{2/3}$ or 
$0<\yexitk{k}{\ell}\le 4c\delta N^{2/3}$. 
The following bound arises:
\begin{align*}
&Q_{m,n}\{ 
\text{ $\exists k$ such that $\abs{\,x_k-(\tau m, \tau n)} \le \delta N^{2/3}$ } \} \\[3pt]
&\le Q^{(k,\ell)}_{m,n}\bigl\{ \text{$0<\xexitk{k}{\ell}\le 4\delta N^{2/3}$ or 
$0<\yexitk{k}{\ell}\le 4c\delta N^{2/3}$}  \bigr\} \\[3pt]
&\overset{d}= Q_{m-k,n-\ell} \bigl\{ \text{$0<\xexit\le 4\delta N^{2/3}$ or 
$0<\yexit\le 4c\delta N^{2/3}$}  \bigr\}.  
\end{align*} 
  Proposition \ref{lb-prop1} now gives \eqref{path-bd-2}.
\end{proof}

\section{Polymer with fixed endpoint but without boundaries}
\label{sec:nobd} 
 Throughout this section, for  given 
 $0<s,t<\infty$, let $\theta=\theta_{s,t}$ as determined by \eqref{trig-st} and
 $(m,n)$ satisfy \eqref{mnNgamma}.   Up to corrections from integer parts, \eqref{ElogZ} and definition \eqref{def-free} give
\[  N\freee_{s,t}(\mu) = \E\log Z_{\fl{Ns}, \fl{Nt}}.   \]
Define the scaling parameter $M$ by 
\be  M=\frac{Ns}{\trigamf(\mu-\theta)} =\frac{Nt}{\trigamf(\theta)} . \label{defMN}\ee
Then 
$(Ns,Nt)=(M\trigamf(\mu-\theta), M\trigamf(\theta))$ is the   characteristic
direction for  parameters $M$ and $\theta$. 

\begin{lemma}  Let $\P$ satisfy assumption \eqref{distr4} and $(m,n)$ satisfy
\eqref{mnNgamma}. 
 There exist finite constants $N_0, C, C_0$ such 
that, for  $b\ge C_0$ and $N\ge N_0$, 
\[
\P\bigl[  \; \abs{ \log { Z_{m,n}} -\log \Zalt_{(1,1),(m,n)} }\ge bN^{1/3} \,\bigr] 
\le C b^{-3/2}. \]  
\label{ZZlm1} \end{lemma} 

\begin{proof} 
Separating the paths that go through the point $(1,1)$ gives 
\be 
Z_{m,n}= (U_{1,0}+V_{0,1}) \Zalt_{(1,1),(m,n)}
+  Z_{m,n}(\xexit>1) +  Z_{m,n}(\yexit>1). \label{ZZaux1}\ee
Consequently 
\begin{align*}
&\P\biggl[  \; \frac{ Z_{m,n}}{ \Zalt_{(1,1),(m,n)}}
\le  e^{-bN^{1/3}}\,\biggr] \le 
\P( U_{1,0}+V_{0,1} \le  e^{-bN^{1/3}}) 
\le C(\theta) e^{-bN^{1/3}}. 
\end{align*}

For the other direction abbreviate  $u=\sqrt b \bigl(\trigamf(\theta)/t\bigr)^{1/6}   M^{2/3}$. 
\begin{align}  &\P\biggl[  \; \frac{ Z_{m,n}}{ \Zalt_{(1,1),(m,n)}} 
\ge  e^{bN^{1/3}}\,\biggr]\nn\\[4pt]
&=   \P\biggl[  \; 
\frac{ Z_{m,n}(\{0<\xexit\le u\}\cup\{0<\yexit\le u\}) }
{ \Zalt_{(1,1),(m,n)}\, Q_{m,n}(\{0<\xexit\le u\}\cup\{0<\yexit\le u\})} 
\ge  e^{bN^{1/3}}\,\biggr]\nn\\[5pt]
&\le   \P\biggl[  \; 
\frac{ Z_{m,n}(0<\xexit\le u) }
{ \Zalt_{(1,1),(m,n)}} 
\ge  \tfrac14 e^{bN^{1/3}}\,\biggr]
+  \P\biggl[  \; 
\frac{ Z_{m,n}(0<\yexit\le u) }
{ \Zalt_{(1,1),(m,n)}} 
\ge  \tfrac14 e^{bN^{1/3}}\,\biggr]
\label{Zprob3}\\[3pt]
&\qquad +\P\bigl[ \, Q_{m,n}(\{0<\xexit\le u\}\cup\{0<\yexit\le u\}) 
\le\tfrac12 \,\bigr]. \label{Zprob4}
\end{align}
By part (ii) of Lemma
\ref{lb-aux-lm3}, 
  line \eqref{Zprob3} is bounded by $Cb^{-3/2}$.   By Lemma 
 \ref{auxlm7} 
\[ \text{line \eqref{Zprob4}}\le   \P\bigl[ \, Q_{m,n}\{\xexit> u\}>\tfrac14\,\bigr]
+ \P\bigl[ \, Q_{m,n}\{\yexit> u\}>\tfrac14\,\bigr] \le  Cb^{-3/2} \]
provided $e^{-\delta b (\trigamf(\theta)/t)^{1/3}   M^{1/3}}\le 1/4$
and $u\ge c\kappa_M$. $M$ is now the scaling parameter and comparison 
of \eqref{mnN1}  and \eqref{mnNgamma} shows $\kappa_M=\gamma N^{2/3}$.  
The requirements  are satisfied  with  $N\ge N_0$
and  $b\ge C_0$.  

 To summarize, we have for  $b\ge C_0$ and $N\ge N_0$, and for a 
 finite constant $C$, 
\be
\P\biggl[  \; \frac{ Z_{m,n}}{ \Zalt_{(1,1),(m,n)}} 
\ge  e^{bN^{1/3}}\,\biggr] \le C b^{-3/2}  \label{Zprob5}\ee
This furnishes the remaining part of the conclusion. 
\end{proof} 

\begin{proof}[Proof of Theorem \ref{Zno-bd-thm}]
By Chebyshev, variance bound \eqref{goal8.1}  
 and Lemma \ref{ZZlm1}, and with a little correction to
take care of the difference between $Z_{(1,1),(\fl{Ns},\fl{Nt})}$ and 
$\Zalt_{(1,1),(\fl{Ns},\fl{Nt})}$, 
\begin{align*}
&\P\bigl[  \; \abs{ \log Z_{(1,1),(\fl{Ns},\fl{Nt})} - N\freee_{s,t}(\mu) }\ge bN^{1/3} \,\bigr] 
\le \P( \,\abs{\log Y_{1,1}}\ge   \tfrac14bN^{1/3} ) \\[3pt]
&\qquad \qquad 
+  \P\bigl[  \; \abs{  \log \Zalt_{(1,1),(\fl{Ns},\fl{Nt})} -\log { Z_{\fl{Ns},\fl{Nt}}}}\ge \tfrac12bN^{1/3} \,\bigr] \\[3pt]
&\qquad \qquad 
+ \P\bigl[  \; \abs{ \log Z_{\fl{Ns},\fl{Nt}}  - N\freee_{s,t}(\mu) }\ge \tfrac14bN^{1/3} \,\bigr] \\[3pt]
&\quad \le  Ce^{-\frac14bN^{1/3}} + Cb^{-3/2} + Cb^{-2} \le Cb^{-3/2}.  
\end{align*}
This bound implies convergence in probability in \eqref{Zlln}. 
One can apply the subadditive ergodic theorem to upgrade the 
statement to a.s.\ convergence. We omit the details. 
\end{proof}

\begin{proof}[Proof of Theorem \ref{path-nobd-thm}]
Let  $(k,\ell)=(\fl{\tau m},\fl{\tau n})$ and 
$u={bN^{2/3}} =b(\trigamf(\theta)/t)^{2/3}M^{2/3}$. 
Decompose the event $\{ \vvb(\ell)\ge k+u\}$ according to the vertical
edge  $\{ (i,\ell), (i,\ell+1)\}$, $k+u\le i\le m$, taken by the path, 
and utilize \eqref{ZZaux1}: 
\begin{align*}
&Q_{(1,1),(m,n)}\{\vvb(\ell)\ge k+u\}
=\sum_{i:k+u\le i\le m} \frac{ \Zalt_{(1,1),(i,\ell)} \Zalt_{(i,\ell+1),(m,n)}}{ \Zalt_{(1,1),(m,n)}}\\
&\le \sum_{i:k+u\le i\le m}
\frac{Z_{i,\ell} \Zalt_{(i,\ell+1),(m,n)}}{ (U_{1,0}+V_{0,1}) \Zalt_{(1,1),(m,n)}}
= \frac{ Q_{m,n}\{\vvb(\ell)\ge k+u\} }{ U_{1,0}+V_{0,1}} \cdot 
\frac{ Z_{ m,n}}{  \Zalt_{(1,1),(m,n)}}. 
\end{align*}  
As explained  in the paragraph of \eqref{Zkl6} above, 
$Q_{m,n}\{\vvb(\ell)\ge k+u\}\overset{d}=Q_{m-k,n-\ell}\{\xexit\ge u\}$.  
Let $b^{-3}<h<1$.  From above, remembering \eqref{defMN}, 
\begin{align*}
& \P \bigl[  Q_{(1,1),(m,n)}\{\vvb(\ell)\ge k+u\}>h  \bigr]  
\; \le \;  \P( U_{1,0}+V_{0,1} \le b^{-3}) \\[5pt]
 &\qquad \qquad + \; \P\biggl[  \; \frac{ Z_{m,n}}{ \Zalt_{(1,1),(m,n)}} 
\ge  \exp\Bigl( \frac{\delta b^2 \trigamf(\theta) N^{1/3}}{2(1-\tau)t}\Bigr) \,\biggr] \\[4pt]
&\qquad + \;  \P \Bigl[ Q_{m-k,n-\ell}\{\xexit\ge u\}  >hb^{-3} 
\exp\bigl( -\tfrac12\delta u^2/(1-\tau)M \bigr)  \Bigr] \\
&\qquad \le Cb^{-3}.  
\end{align*}
The justification for the last inequality is as follows.  
With a new scaling parameter $(1-\tau)M$, bound \eqref{aux8.6}  applies
to the last probability above and bounds it by $Cb^{-3}$ for all
 $h>b^{-3}$ and $b\ge 1$,    provided $N\ge N_0$.
Apply \eqref{Zprob5} to the second last probability,  valid
if  $ b\ge C_0$ and 
$N\ge N_0$.
We obtain
\begin{align*}
\Pann_{(1,1),(m,n)}\{\vvb(\ell)\ge k+u\} &\le b^{-3} 
+ \int_{b^{-3}}^1  \P \bigl[  Q_{(1,1),(m,n)}\{\vvb(\ell)\ge k+u\}>h  \bigr] \,dh \\
&\le Cb^{-3}. 
\end{align*}

The corresponding bound from below on $\vva(\ell)$ comes by reversal.
If $\wt Y_{i,j}=Y_{m-i+1,n-j+1}$ for $(i,j)\in\Lambda_{(1,1),(m,n)}$, then 
$Q^{\wt\w}_{(1,1),(m,n)}(x_\centerdot)=Q^{\w}_{(1,1),(m,n)}(\wt x_\centerdot)$
where $\wt x_{j}=(m+1,n+1)-x_{m+n-2-j}$ for $0\le j\le m+n-2$. 
This mapping of paths has the property 
$\vva(\ell,x_\centerdot)-k=m+1-k-\vvb(n+1-\ell, \wt x_\centerdot)$, and it converts an 
upper  bound on $\vvb$ into a lower   bound on $\vva$.  
\end{proof}

\section{Point-to-line polymer}
\label{sec:tot} 

In this final section we prove Theorems \ref{tot-f-thm} and \ref{thm-Ztotbd}, beginning with 
the three parts of  Theorem \ref{tot-f-thm}. 

\begin{proof}[Proof of limit \eqref{Ztotlln}]
The claimed limit is the maximum over directions in the first quadrant: 
\[ -\digamf(\mu/2)=f_{1/2,1/2}(\mu)\ge f_{s,1-s}(\mu)
\quad\text{for $0\le s\le 1$.} \]
  One bound for  the limit comes from  
$Z_N^{\text{\rm p2l}}\ge  Z_{(1,1),(\fl{N/2}, N-\fl{N/2})}$.
To bound $\log  Z_N^{\text{\rm p2l}}$ from above, fix $K\in\bN$ 
and let $\delta=1/K$.  For $1\le k\le K$ set
$(s_k,t_k)=(k\delta, (K-k+1)\delta)$.  Partition the
indices $m\in\{1,\dotsc,N-1\}$ into sets 
\[ I_k=\{ m\in\{1,\dotsc,N-1\}:  (m,N-m)\in\Lambda_{\fl{Ns_k}, \fl{Nt_k}} \}. \]
The $\{I_k\}$ cover the entire set of $m$'s because
$N(k-1)\delta\le m\le Nk\delta$ implies $m\in I_k$. Overlap among the
$I_k$'s is not harmful.   
\begin{align*}
 Z_N^{\text{\rm p2l}} &\le 
\sum_{k=1}^K \sum_{m\in I_k} Z_{(1,1),(m, N-m)} 
\frac{Z_{(m, N-m), (\fl{Ns_k}, \fl{Nt_k})}} {Z_{(m, N-m), (\fl{Ns_k}, \fl{Nt_k})}}\\
 &\le 
 \Bigl\{\min_{1\le k\le K, \, m\in I_k}Z_{(m, N-m), (\fl{Ns_k}, \fl{Nt_k})}\Big\}^{-1}
 \sum_{k=1}^K  Z_{(1,1),  (\fl{Ns_k}, \fl{Nt_k})} . 
\end{align*} 
For each $m\in I_k$ fix a specific path 
$x^{(m)}_\centerdot \in\Pi_{(m, N-m), (\fl{Ns_k}, \fl{Nt_k})}$. Since 
\[ Z_{(m, N-m), (\fl{Ns_k}, \fl{Nt_k})} \ge  \prod_{i=1}^{\fl{Ns_k}+ \fl{Nt_k}-N} Y_{x^{(m)}_i}\,, \]
we get the bound  
\be  \begin{aligned}
N^{-1} \log  Z_N^{\text{\rm p2l}} \;  &\le \; 
\max_{1\le k\le K, \, m\in I_k}   N^{-1} \sum_i  \log Y_{x^{(m)}_i}^{-1}  \; + \; N^{-1} \log K   \\
&\qquad  +\; \max_{1\le k\le K} N^{-1} \log  Z_{(1,1),  (\fl{Ns_k}, \fl{Nt_k})} .
  \end{aligned} \label{Ztot5}\ee
The sum $\sum_i  \log Y_{x^{(m)}_i}^{-1} $ has $\fl{Ns_k}+ \fl{Nt_k}-N\le  N\delta $ i.i.d.\ terms. 
Given $\e>0$,  we can choose $\delta=K^{-1}$ small enough to guarantee that 
    $\P\{\sum_i  \log Y_{x^{(m)}_i}^{-1}\ge N\e\}$ decays exponentially with $N$.
Thus $\P$-a.s.\ the entire first term after the inequality in \eqref{Ztot5} is $\le \e$ for large $N$.  
In the limit we get, utilizing law of large numbers \eqref{Zlln}, 
\[  \varlimsup_{N\to\infty}  N^{-1} \log  Z_N^{\text{\rm p2l}} \le  \e + 
\max_{1\le k\le K} \freee_{s_k, t_k}(\mu) \le  \e+  \sup_{0\le s\le 1} \freee_{s,1-s+\delta}(\mu). \]
Let $\delta\searrow 0$ utilizing the  continuity of $\freee_{s,t}(\mu)$ in $(s,t)$, and 
then let $\e\searrow 0$.  This gives $\varlimsup  N^{-1} \log  Z_N^{\text{\rm p2l}} \le
-\digamf(\mu/2)$ and completes the proof of the limit \eqref{Ztotlln}. 
\end{proof}

\begin{proof}[Proof of bound \eqref{fa-4}]  
Let 
\be (m, n)=( N-\fl{N/2}, \fl{N/2} ). \label{fa-2}\ee
 An upper bound on the left tail in \eqref{fa-4} comes immediately 
from \eqref{Zmom2.9}: 
\begin{align*}
\bP\bigl\{  \log  Z_N^{\text{\rm p2l}}  \le  N\freee_{1/2,1/2}(\mu) - bN^{1/3}  \bigr\}
&\le 
\bP\bigl\{  \log  Z_{(1,1),(m, n)}   \le  N\freee_{1/2,1/2}(\mu) - bN^{1/3}  \bigr\}\\
&\le Cb^{-3/2}.
\end{align*}


The  bound on the right tail is proved in two parts.  We start with the easy case.

\medskip 

\noindent 
{\bf Case 1.}  Assume $b\ge c_0N^{2/3}$ for some constant $c_0>0$.  

\medskip

Continue with $(m,n)$ as in \eqref{fa-2} and 
let $\theta=\mu/2$ be the boundary parameter for  partition functions $Z_{m,n}$.  Since we have the fluctuation bounds for $Z_{m,n}$ and 
$\E(\log Z_{m,n})=-N\digamf(\mu/2)=Nf_{1/2,1/2}(\mu)$, it suffices to prove that 
\be\label{dd-goal1}
\P\Bigl[  \; \frac{Z_N^{\text{\rm p2l}} } {Z_{m,n}}  \ge e^{bN^{1/3}} \,\Bigr] 
\le C b^{-3/2}.
\ee
Using   ratio variables, 
\begin{align}
 \frac{Z_N^{\text{\rm p2l}} } {Z_{m,n}} &= \sum_{\ell=1}^{N-1}\frac{ Z_{(1,1),(\ell, N-\ell)}} {Z_{m,n}} \nn\\
 &=   \sum_{\ell=1}^{n}\frac{ Z_{(1,1),(\ell, N-\ell)}} {Z_{\ell, N-\ell}} 
  \prod_{i=\ell}^{m-1} \frac{V_{i,N-i}}{U_{i+1, N-i-1}}
\; + \; \sum_{\ell=m\vee(n+1)}^{N-1}\!\!\! \frac{ Z_{(1,1),(\ell, N-\ell)}} {Z_{\ell, N-\ell}} 
  \prod_{i=m+1}^{\ell} \!\!\frac{U_{i,N-i}}{V_{i-1, N-i+1}}\nn \\
 &\le   \frac1{U_{1,0}Y_{1,1}}  \sum_{\ell=1}^{n}
  \prod_{i=\ell}^{m-1} \frac{V_{i,N-i}}{U_{i+1, N-i-1}}
\; + \;  \frac1{U_{1,0}Y_{1,1}}  \sum_{\ell=m}^{N-1}
  \prod_{i=m+1}^{\ell} \frac{U_{i,N-i}}{V_{i-1, N-i+1}}  .  \label{dd:row9} 
\end{align}
In the last step we used \eqref{ZZaux1}.   The two terms on the last line above are similar so let us see how to handle the first one.    
 By the Burke property (Theorem \ref{burkethm}) and because now $\theta=\mu-\theta$,  the $V$ and $U$ variables  in the products are i.i.d. With a simplifying change of indices  we can rewrite the first  term as 
\begin{align*}
 \frac1{U_{1,0}Y_{1,1}}  \sum_{\ell=1}^{n}
  \prod_{i=\ell}^{m-1} \frac{V_{i,N-i}}{U_{i+1, N-i-1}} 
  \le    \frac1{U_{1,0}Y_{1,1}}  \sum_{k=0}^{m-1} e^{S_k}  
\end{align*} 
where  $S_k=\sum_{i=1}^k \xi_i$ is a sum of mean zero i.i.d.\ terms. 
 (The only reason there is an inequality above is that if $m>n$ then we have introduced an extra $k=0$ term.)    The desired bound comes from applying Kolmogorov's inequality to the random walk $S_k$.  
\begin{align*}
&\P\Bigl\{ \;   \frac1{U_{1,0}Y_{1,1}}  \sum_{\ell=1}^{n}
  \prod_{i=\ell}^{m-1} \frac{V_{i,N-i}}{U_{i+1, N-i-1}}      \ge e^{bN^{1/3}} \,\Bigr\} \\
  &\le  \P\Bigl\{  \; \max_{0\le k<m} S_k 
\ge bN^{1/3}   -  \log N + \log U_{1,0} +\log Y_{1,1}  \Bigr\}\\
&\le  \P\Bigl\{  \; \max_{0\le k<m} S_k 
\ge \tfrac12bN^{1/3}      \Bigr\}   +  \P( U_{1,0}^{-1}>b\,)  +  \P( Y_{1,1}^{-1}>b\,) \\
&\le  \frac{Cm}{(bN^{1/3})^2}  + Ce^{-cb}  \le  Cb^{-3/2}.  
\end{align*}
The  steps above  came from  Kolmogorov's inequality,  $m\le N\le Cb^{3/2}$, lower tail bounds for gamma variables, and from taking $N$ large enough and $b\ge 1$.  

\medskip
\noindent
{\bf Case 2.}  For some constant $c_0>0$,  $1\le b\le  c_0N^{2/3}$.   

\medskip

Begin with  
\be\begin{aligned}
  &Z_N^{\text{\rm p2l}}  = \sum_{\ell=1}^{N-1} Z_{(1,1),(\ell, N-\ell)}\\
  &\le N \biggl( {Z_{(1,1),(m, n)}} \cdot \max_{0\le k< n} \frac{Z_{(1,1),(m+k, n-k)}}{Z_{(1,1),(m, n)}}      \biggr) \bigvee  \biggl( {Z_{(1,1),(n, m)}} \cdot  \max_{0\le \ell< m} 
  \frac{Z_{(1,1),(n-\ell, m+\ell)}}{Z_{(1,1),( n,  m)}}  \biggr) 
    \end{aligned}\label{fa0}\ee
 The terms  in the large parentheses   are transposes of each other, so 
 we spell out the details only for the first case.   In one spot below it is convenient
 to have $m\ge n$, hence the choice in \eqref{fa-2}.  
 Thus, considering $b\ge 2$, and once
  $N$ is  large enough so that $\log N<N^{1/3}/3$,   bounding 
  \[  \bP\bigl\{  \log  Z_N^{\text{\rm p2l}}  \ge  N\freee_{1/2,1/2}(\mu) + bN^{1/3}  \bigr\}  \]
boils down to bounding the sum 
\begin{align}
&\bP\Bigl\{  \log  Z_{(1,1),(m, n)}   \ge  N\freee_{1/2,1/2}(\mu) + \tfrac13bN^{1/3}  \Bigr\}
\label{fa1}\\[4pt]
&\qquad \qquad  +  
\bP\Bigl\{  \log    \max_{0< k< n} \frac{Z_{(1,1),(m+k, n-k)}}{Z_{(1,1),(m, n)}}  
   \ge  \tfrac13bN^{1/3}  \Bigr\}.  
\label{fa2}\end{align}
The probability on line \eqref{fa1} is again taken care of with \eqref{Zmom2.9}. 
  Utilizing both inequalities in \eqref{Zcomp2}, the first one transposed, 
we deduce for $1\le k< n$, 
\be\begin{aligned}
&\frac{Z_{(1,1),(m+k, n-k)}}{Z_{(1,1),(m, n)}}
=\prod_{j=1}^k \frac{Z_{(1,1),(m+j, n-j)}}{Z_{(1,1),(m+j-1, n-j)}}
\cdot  \frac{Z_{(1,1),(m+j-1, n-j)}}{Z_{(1,1),(m+j-1, n-j+1)}}\\
&\le  \prod_{j=1}^k \frac{Z_{m+j,\, n-j}(\xexit>0)}{Z_{m+j-1,\, n-j}(\xexit>0)}
\cdot  \frac{Z_{m+j-1,\, n-j}(\xexit>0)}{Z_{m+j-1,\, n-j+1}(\xexit>0)}\\
&=  \frac{Z_{m+k,\, n-k}(\xexit>0)}{Z_{m, n}(\xexit>0)}
\le    \frac1{Q_{m, n}(\xexit>0)} \cdot \frac{Z_{m+k,\, n-k}}{Z_{m, n}}\\
&= \frac1{Q_{m, n}(\xexit>0)} \cdot \prod_{j=1}^k \frac{U_{m+j,n-j}}{V_{m+j-1,n-j+1}}. 
\end{aligned}\label{fa3}\ee   
The last equality used \eqref{Z2}.  In the  calculation above we  switched
from  partition functions $Z_{(1,1),(i,j)}$ that use only bulk weights to 
 partition functions $Z_{i,j}= Z_{(0,0),(i,j)}$  that use both bulk and 
  boundary weights, distributed as in assumption \eqref{distr4}.  The
parameter $\theta$ is at our disposal.  We take $\theta=\mu/2+rN^{-1/3}$ with $r>0$ 
and link $r$ to $b$ in the next lemma.  
The choice $\theta>\mu/2$ makes the $U/V$  ratios small which is good for 
bounding the last line of \eqref{fa3}.  However, this choice also makes 
$Q_{m, n}(\xexit>0)$ small which works against us.  To bound $Q_{m, n}(\xexit>0)$
from below we switch from $\theta=\mu/2+rN^{-1/3}$ to $\lambda=\mu/2-rN^{-1/3}$
and pay for this by bounding the Radon-Nikodym derivative.  Under parameter
$\lambda$ 
the event $\{\xexit>0\}$ is favored at the expense of $\{\yexit>0\}$, and we can get
a lower bound.  

Utilizing \eqref{fa3}, the probability in  \eqref{fa2}   is  
bounded as follows: 
\begin{align}
&\bP\Bigl\{  \log   \max_{1\le k\le n} \frac{Z_{(1,1),(m+k, n-k)}}{Z_{(1,1),(m, n)}}
   \ge  \tfrac13bN^{1/3}  \Bigr\}
  \; \le \; \bP\bigl\{  Q_{m, n}(\xexit>0)  \le  e^{-bN^{1/3}/6} \bigr\} \label{fa4}\\[5pt]
 &\qquad +  \; \bP\Bigl\{  \,  \max_{1\le k\le n}
 \sum_{j=1}^k \bigl( \log {U_{m+j,n-j}} -\log {V_{m+j-1,n-j+1}} \bigr) 
 \ge {bN^{1/3}/6}  \Bigr\} . \label{fa5} 
\end{align} 
We  treat first the right-hand side probability on line \eqref{fa4}.

\begin{lemma}  Let $(m,n)$ be as in \eqref{fa-2}.  Let $\theta=\mu/2+rN^{-1/3}$ be the parameter of boundary weights  as specified in  \eqref{distr4}. Given   $c_0\ge 1$,  we can choose positive constants $\e_0(\mu)<\e_1(\mu)$,  $N_0(\mu)$ and $C(\mu)$  such that,  under conditions 
\be\label{fa5.5} 
1\le b\le c_0N^{2/3}, \ \ \frac{\e_0(\mu)}{\sqrt{c_0}}\sqrt{b}\le r\le\frac{\e_1(\mu)}{\sqrt{c_0}}\sqrt{b}\; , 
\ \ \text{ and } \ N\ge c_0^{3}N_0(\mu), 
\ee
we have 
\be  \bP\bigl\{  Q_{m, n}(\xexit>0) \le   e^{-bN^{1/3}/6} \bigr\} \le  C(\mu)c_0^{3/2}b^{-3/2}.  
\label{fa6}\ee
\label{lmfa6} 
\end{lemma}
\begin{proof}    
Let $U_{i,0}, V_{0,j}$ be the boundary weights with parameter
$\theta=\mu/2+rN^{-1/3}$ as specified in  \eqref{distr4}. Let  $\Util_{i,0}, \Vtil_{0,j}$ denote
boundary weights  with 
parameter $\lambda=\mu/2-rN^{-1/3}$ in place of $\theta$. 
We ensure $\mu/4\le \lambda<\theta\le 3\mu/4$ by taking  $\e_1(\mu)\le\mu/4$.  

We also use 
a scaling parameter $M$   determined by $n=M\trigamf(\lambda)$.
Set $\bar m=\fl{M\trigamf(\mu-\lambda)}$ which satisfies $m-C_1(\mu)rN^{2/3}\le\bar m\le m$
for a constant $C_1(\mu)$, as long as $N\ge c_0^{3/4}N_0(\mu)$ to take care of the constant error from integer parts.    Set $t=2C_1(\mu)r$ and   $u=\fl{tN^{2/3}}$.

All along   bulk weights have distribution $Y_{i,j}^{-1}\sim$ Gamma($\mu,1$).
  The coupling of the boundary weights $\{U_{i,0}, V_{0,j}\}$  with $\{\Util_{i,0}, \Vtil_{0,j}\}$
   is such that $U_{i,0}\le \Util_{i,0}$. 
    Tildes mark quantities that use $\Util_{i,0}, \Vtil_{0,j}$.
  Recall that $\digamf$ is
   strictly increasing and $\trigamf$ strictly decreasing. 
\begin{align}
Q_{m, n}(\xexit>0)&\ge Q_{m, n}(0<\xexit\le u)
= \frac1{Z_{m,n}}\sum_{k=1}^u \biggl(\;\prod_{i=1}^k U_{i,0} \biggr) 
  \Zalt_{(k,1),(m,n)}\nn\\[4pt]
  &= \frac1{\wt Z_{m,n}}\sum_{k=1}^u 
  \biggl(\;\prod_{i=1}^k \Util_{i,0} \cdot \prod_{i=1}^k \frac{U_{i,0}}{\Util_{i,0}} \biggr) 
  \Zalt_{(k,1),(m,n)}\cdot  \frac{\wt Z_{m,n}}{Z_{m,n}}\nn\\[5pt]
  &\ge \wt Q_{m, n}(0<\xexit\le u)  \biggl(\;  \prod_{i=1}^u \frac{U_{i,0}}{\Util_{i,0}} \biggr) 
 \frac{\wt Z_{m,n}}{Z_{m,n}}.  \label{fa8} \end{align}  
We derive tail bounds for each of the three factors on line  \eqref{fa8}, working our
way from right to left. $C(\mu)$   denotes a constant that depends on $\mu$ and 
can change from one line to the next, while $C_i(\mu)$ denote constants specific
to the cases.  

Since $\theta>\lambda$ sit symmetrically around $\mu/2$ and $m\ge n$, 
\begin{align*}  \E(\log \wt Z_{m,n})-\E(\log  Z_{m,n})=m\bigl(-\digamf(\lambda)+\digamf(\theta)\bigr) 
+n\bigl(-\digamf(\mu-\lambda)+\digamf(\mu-\theta)\bigr)\ge 0 
\end{align*}
 and in fact vanishes for even $N$.  By Chebyshev and   the variance bound of 
 Theorem \ref{var-bd-thm}, 
\be\begin{aligned}
&\P\Bigl[ \; \frac{\wt Z_{m,n}}{Z_{m,n}} \le e^{-bN^{1/3}/18} \Bigr] 
\le\P\Bigl[ \;  \overline{ \log \wt Z_{m,n}}-\overline{\log Z_{m,n}} \,\le -bN^{1/3}/18  \Bigr]\\
&\le   \frac{2\cdot18^2}{N^{2/3}b^2} \bigl( \Vvv(\log \wt Z_{m,n}) + \Vvv(\log Z_{m,n}) \bigr) 
\le C(\mu)(1+r) b^{-2}\le C(\mu)b^{-3/2}. 
\end{aligned}\label{fa10}\ee
To understand the last inequality above for the first variance,  use the scaling parameter $M$  from above.  
Since $(\bar m, n)$ is the characteristic direction for $\lambda$ and $M$, 
\begin{align*}  \Vvv(\log \wt Z_{m,n}) &= \Vvv\Bigl(\log \wt Z_{\bar m,n} 
+ \sum_{i=\bar m+1}^{m}\log \wt U_{i,n}\Bigr)\\
&\le 2 \Vvv\bigl(\log \wt Z_{\bar m,n} \bigr) 
+ 2 \Vvv\Bigl(\; \sum_{i=\bar m+1}^{m}\log \wt U_{i,n}\Bigr)\\
&\le C(\mu)(M^{2/3} +  m-\bar m)  \le C(\mu)(1+r)N^{2/3}.  \end{align*}
We used above the variance bound of Theorem \ref{var-bd-thm} together with
the feature that  fixed constants work for parameters varying in a compact set.
This is now valid because we have constrained $\lambda$ and $\theta$ to lie
in $[\mu/4,3\mu/4]$. 
Similar argument works for the second variance in \eqref{fa10}.

Next, 
\[  \E(\log U_{1,0}-\log\Util_{1,0})=-\digamf(\theta)+\digamf(\lambda) 
\ge -C_2(\mu)rN^{-1/3}. \] 
Since $c_0\ge 1$, we can ensure $b>36C_2(\mu)rt$ by choosing $\e_1(\mu)$ small enough.  
Then by Chebyshev,   
\be\begin{aligned}
\P\Bigl[ \; \prod_{i=1}^u \frac{U_{i,0}}{\Util_{i,0}} \le e^{-bN^{1/3}/18} \Bigr] 
&= \P\Bigl[ \;\sum_{i=1}^u (\overline{\log U_{i,0}}-\overline{\log\Util_{i,0}})  
\le  -\bigl(\tfrac1{18}b-C_2(\mu)rt\bigr)N^{1/3} \Bigr]\\
&\le C(\mu)tb^{-2}\le C(\mu)b^{-3/2}. 
\end{aligned}\label{fa12}\ee

Now choose $\e_0(\mu) \in (0, \e_1(\mu))$.  
For the probability on  line  \eqref{fa8} write 
\be  \wt Q_{m, n}\{0<\xexit\le tN^{2/3}\} = 1- \wt Q_{m, n}\{\xexit> tN^{2/3}\} 
- \wt Q_{m, n}\{\yexit> 0\}.  \label{fa14}\ee
To both probabilities on the right we apply Lemma \ref{auxlm7} after adjusting the
parameters. Let $M$ and $\bar m$ be as above so that $(\bar m,n)$ is the 
characteristic direction for $\lambda$.  
   Reasoning as for the distributional equality in
\eqref{xidist2} and with $t= 2 C_1(\mu)r$, 
\[   \wt Q_{m, n}\{\xexit> tN^{2/3}\} \overset{d}=  
\wt Q_{\bar m, n}\{\xexit> tN^{2/3}-(m-\bar m)\}  \le  \wt Q_{\bar m, n}\{\xexit> tN^{2/3}/2\}. \] 
Consequently  by \eqref{aux8.6} 
\begin{align*}
\P\bigl[ \, \wt Q_{\bar m, n}\{\xexit> tN^{2/3}/2\} \ge e^{-\delta t^2 N^{4/3}/(4M) } \,\bigr] 
\le  C(\mu)t^{-3}\le C(\mu)c_0^{3/2}b^{-3/2}.  
\end{align*}
For the last probability on line \eqref{fa14} we get the same kind of bound by
defining $K$ through $m=K\trigamf(\mu-\lambda)$, and  
$\bar n=\fl{K\trigamf(\lambda)}\ge n+C_4(\mu)rN^{2/3}$.  Then
\[   \wt Q_{m, n}\{\yexit> 0\}  \overset{d}= \wt Q_{m, \bar n}\{\yexit> \bar n-n\} 
\le \wt Q_{m, \bar n}\{\yexit>C_4(\mu)rN^{2/3}\},  \]
and again by \eqref{aux8.6} 
\begin{align*}
\P\bigl[ \,  \wt Q_{m, \bar n}\{\yexit>C_4(\mu)rN^{2/3}\} \ge e^{-\delta C_4(\mu)^2 r^2 N^{4/3}/K } \, \bigr] 
\le  C(\mu)r^{-3}\le C(\mu)c_0^{3/2}b^{-3/2}.  
\end{align*}
The lower bound $b\ge 1$ implies lower  bounds $r\wedge t\ge \e_2(\mu)c_0^{-1/2}$ with $\e_2(\mu)>0$.    Hence we can ensure that 
\[   e^{-\delta C_4(\mu)^2 r^2 N^{4/3}/K } \le 1/4 \quad\text{and}\quad 
e^{-\delta t^2 N^{4/3}/(4M) } \le 1/4 \]  
by enforcing  $N\ge c_0^3N_0(\mu)$ for a large enough constant $N_0(\mu)$.  
The upshot of this  paragraph is that if $N\ge c_0^3N_0(\mu)$  then 
\be 
\P\bigl[ \,  \wt Q_{m, n}\{0<\xexit\le u\}  \le  \tfrac12 \, \bigr] 
 \le C(\mu)c_0^{3/2}b^{-3/2}.  \label{fa16}\ee

Put bounds \eqref{fa10}, \eqref{fa12}  and \eqref{fa16}  back into \eqref{fa8}.  
 Adding up the bounds gives 
\[  \P\bigl[ \,  Q_{m, n}\{\xexit>0\}  \le  e^{-bN^{1/3}/6} \, \bigr] 
\le   C(\mu)c_0^{3/2}b^{-3/2}.  \qedhere  \]
 \end{proof}

We turn to probability \eqref{fa5}. By the Burke property Theorem \ref{burkethm} 
inside the probability we have a sum of i.i.d.\ terms with mean 
\be   \E( \log {U_{m+1,n-1}} -\log {V_{m,n}})=-\digamf(\theta)+\digamf(\mu-\theta) \le 
-C_5(\mu)rN^{-1/3}. \label{fa17}\ee
Consequently, if we let 
\be 
\eta_j= \log {U_{m+j,n-j}} -\log {V_{m+j-1,n-j+1}}+\digamf(\theta)-\digamf(\mu-\theta),
\label{faeta}\ee
 then
\be
\text{\eqref{fa5}} \le 
\bP\Bigl\{    \max_{1\le k\le n}
 \sum_{j=1}^k \bigl(\eta_j -C_5(\mu)rN^{-1/3} \bigr) 
 \ge bN^{1/3}/6 \Bigr\}. 
 \label{fa18}\ee
 The variables $\eta_j$ have all moments.  Apply 
 part (a) of Lemma \ref{rwlm}  below to the probability
above with  $t=N^{1/3}$,  $\alpha =C_5(\mu)r$ and  $\beta=b/6$.  With 
$r=\kappa(\mu)b^{1/2}$, $b^{1/2}\ge 6C_5(\mu)\kappa(\mu)$, and $p$
large enough,
this gives 
\be  
\text{\eqref{fa5}} \le  C(\mu) b^{-3/2}.  \label{fa20}\ee 

Insert bounds \eqref{fa6} and  \eqref{fa20} into   \eqref{fa4}--\eqref{fa5}, and this
in turn back into  \eqref{fa2}.  This completes the proof of \eqref{fa-4}. 
\end{proof}


Before the third and last part of the proof of Theorem \ref{tot-f-thm}    
we state and prove the random walk lemma used to derive \eqref{fa20} above. 
It includes a part (b)   for subsequent use. 
 
\begin{lemma}  Let  $Z, Z_1, Z_2, \dotsc$ be i.i.d.\ random
variables that satisfy  $\mE(Z)=0$  and  $\mE(\abs{Z}^p)<\infty$ for
some $p>2$.  Set 
$S_k=Z_1+\dotsm+Z_k$.  Below $C=C(p)$     is a constant that depends only on $p$.

{\rm (a)} For all  $\beta\ge\alpha>0$ and  $ t>0$, 
\[  \mP\Bigl\{  \, \sup_{k\ge 0} \bigl( S_k- k\alpha t^{-1}\bigr) \ge \beta t\Bigr\} \le 
C\mE(\abs{Z}^p) \alpha^{-\,\frac{p^2}{2(p-1)}} \beta^{-\,\frac{p(p-2)}{2(p-1)}}. 
  \]
  
  {\rm (b)} For all  $\alpha,\beta, t>0$
  and $M\in\bN$ such that $2\beta\le M\alpha$, 
\[  \mP\Bigl\{ \, \sup_{k>Mt^2} \bigl( S_k- k\alpha t^{-1}\bigr) \ge -\beta t\Bigr\} \le 
C\mE(\abs{Z}^p) \alpha^{-p}M^{-(p/2)+1}. \]  
  \label{rwlm}\end{lemma}
\begin{proof}  Part (a). 
Pick an integer $m>0$ and split the probability:
\be\begin{aligned}
&\mP\Bigl\{ \, \sup_{k\ge 0} \bigl( S_k- k\alpha t^{-1}\bigr) \ge \beta t\Bigr\} \le 
\mP\Bigl\{ \, \max_{0<k\le mt^2}  S_k  \ge \beta t\Bigr\} \\
&\qquad + \sum_{j\ge m} 
\mP\Bigl\{ \, \max_{jt^2<  k \le  (j+1)t^2} \bigl( S_k- k\alpha t^{-1}\bigr) \ge \beta t\Bigr\}. 
 \end{aligned}\label{fe1}\ee
 Recall that the Burkholder-Davis-Gundy inequality \cite[Thm 3.2]{burk-73} gives 
$\mE \abs{S_k}^p \le C_p \mE\abs{Z}^p k^{p/2}$. 
  Doob's inequality together with BDG gives
\[   \mP\Bigl\{ \, \max_{0<k\le mt^2}   S_k  \ge \beta t\Bigr\}   \le  
C \mE\abs{Z}^p m^{p/2}\beta^{-p}  
  \]
  where we now write $C$ for a constant that depends only on $p$. 
For the last  probability in \eqref{fe1},  
\begin{align*}
&\mP\Bigl\{ \, \max_{jt^2<  k \le  (j+1)t^2} \bigl( S_k- k\alpha t^{-1}\bigr) \ge \beta t\Bigr\}  \le 
\mP\Bigl\{ \, \max_{0<  k \le  (j+1)t^2}  S_k \ge  j\alpha t  \Bigr\} \\[4pt]
 &\qquad\qquad 
\le C \mE\abs{Z}^p j^{-p/2}\alpha^{-p}. 
\end{align*}
Putting the bounds back into \eqref{fe1} gives
\begin{align*}
\mP\Bigl\{  \sup_{k\ge 0} \bigl( S_k- k\alpha t^{-1}\bigr) \ge \beta t\Bigr\} &\le 
C \mE\abs{Z}^p\Bigl(  \frac{m^{p/2}}{\beta^{p}}  + \alpha^{-p} \sum_{j\ge m} j^{-p/2}  \Bigr) \\
&\le C \mE\abs{Z}^p \bigl(   {m^{p/2}}{\beta^{-p}}  + \alpha^{-p}{m^{-(p/2)+1}}\bigr). 
\end{align*}
Choosing $m$ a constant multiple of $(\beta/\alpha)^{p/(p-1)}$ gives the conclusion
for part (a). 

\medskip

Part (b).  Proceeding as above:
\begin{align*}
& \mP\Bigl\{  \sup_{k>Mt^2} \bigl( S_k- k\alpha t^{-1}\bigr) \ge -\beta t\Bigr\} \le
 \sum_{j\ge M} 
\mP\Bigl\{ \, \max_{jt^2<  k \le  (j+1)t^2} \bigl( S_k- k\alpha t^{-1}\bigr) \ge -\beta t\Bigr\}\\
 &\qquad \le \sum_{j\ge M} \mP\Bigl\{ \, \max_{0<  k \le  (j+1)t^2}  S_k \ge \tfrac12 j\alpha t  \Bigr\}
\le  C \mE\abs{Z}^p \alpha^{-p} \sum_{j\ge M} j^{-p/2}\\
&\qquad \le C\mE(\abs{Z}^p) \alpha^{-p}M^{-(p/2)+1}.  \qedhere
\end{align*}
\end{proof}

Next  the   last part of the proof of Theorem \ref{tot-f-thm}.   

\begin{proof}[Proof of bound   \eqref{fa-5}]
  We shall show the existence of 
 constants $c_1=c_1(\mu)>0$ and  $C(\mu), N_0(\mu)<\infty$ such that, 
for $s\ge (6c_1)^{-1/2}$ and $N\ge N_0(\mu)$, 
\be \P\Bigl[\;  Q_N^{\text{\rm p2l}}\bigl\{ \,\bigl\lvert 
{x_{N-2}-(\tfrac{N}2, \tfrac{N}2)} \bigr\rvert \ge 2sN^{2/3}\,\bigr\}
\ge e^{-c_1s^2N^{1/3}} \,\Bigr]  \le \; C(\mu)s^{-3}.  
\label{fa-3}\ee
Abbreviating $A_N=  \{ \, \lvert 
{x_{N-2}-(\tfrac{N}2, \tfrac{N}2)}  \rvert \ge 2sN^{2/3}\}$,
we have   
 \begin{align*}
P_N^{\text{\rm p2l}}(A_N) &= \E Q_N^{\text{\rm p2l}}(A_N) \le e^{-c_1s^2N^{1/3}} + 
\P[\;  Q_N^{\text{\rm p2l}}(A_N)   \ge e^{-c_1s^2N^{1/3}} \,]
\le C(\mu)s^{-3}.  
\end{align*}
To cover all $s\ge 1/2$   increase the constant $C(\mu)$,  and then 
  \eqref{fa-5} follows for $b=2s$.   

To show \eqref{fa-3} we control sums of ratios of partition functions: 
\begin{align*}
&Q_N^{\text{\rm p2l}}\{ \,\abs{x_{N-2}-(\tfrac{N}2, \tfrac{N}2)}\ge 2sN^{2/3}\,\} \\[4pt]
&\quad \le    \sum_{0<\ell<N/2-sN^{2/3}}\frac{ Z_{(1,1),(\ell,N-\ell)}} {Z_N^{\text{\rm p2l}}}
+ \sum_{N/2+sN^{2/3}<\ell<N} \frac{ Z_{(1,1),(\ell,N-\ell)} } {Z_N^{\text{\rm p2l}}} . 
\end{align*}
We treat the second sum from above. The first one develops the same way.
  With $(m,n)$ as in \eqref{fa-2} and utilizing
\eqref{fa3} write 
\begin{align*}
 &\sum_{N/2+sN^{2/3}<\ell<N} \frac{ Z_{(1,1),(\ell,N-\ell)} } {Z_N^{\text{\rm p2l}}}  \  \ 
 \le \sum_{sN^{2/3}\le k <N/2} \frac{ Z_{(1,1),(m+k,n-k)} } { Z_{(1,1),(m,n)} }   \\
&\qquad\qquad \le \; \frac1{Q_{m, n}(\xexit>0)} \sum_{sN^{2/3}\le k <N/2} \; 
 \prod_{j=1}^k \frac{U_{m+j,n-j}}{V_{m+j-1,n-j+1}}\\
 &\qquad\qquad \le \; \frac{N}{Q_{m, n}(\xexit>0)} \cdot \max_{sN^{2/3}\le k <N/2} \;
 \prod_{j=1}^k \frac{U_{m+j,n-j}}{V_{m+j-1,n-j+1}} .
 \end{align*} 
As in   \eqref{fa-2} we introduced again boundary weights with parameter 
$\theta=\mu/2+rN^{-1/3}$.     
The value of $c_1=c_1(\mu)$  will be determined below. 
Consider $N$ large enough so that $N\le e^{c_1N^{1/3}}$ and take  $s\ge 1$.
Define $\eta_j$ as in \eqref{faeta} and let $C_5(\mu)$ be as in \eqref{fa17}.  Then 
\begin{align}
&\P\biggl[\; \sum_{N/2+sN^{2/3}<\ell<N} \frac{ Z_{(1,1),(\ell,N-\ell)} } {Z_N^{\text{\rm p2l}}}
\ge e^{-c_1s^2N^{1/3}} \,\biggr] \nn \\
&\quad  \le \; \P\bigl[ \, Q_{m, n}(\xexit>0) \le e^{-c_1s^2N^{1/3}} \, \bigr] \label{fa30} \\
&\qquad\qquad 
+\;  \bP\Bigl[  \;  \max_{sN^{2/3}\le k\le N/2}
 \sum_{j=1}^k \bigl(\eta_j -C_5(\mu)rN^{-1/3} \bigr) 
 \ge -3c_1s^2N^{1/3} \,\Bigr]   \label{fa32}\\
 &\quad  \le \; C(\mu)s^{-3}. \nn
\end{align} 

The justification for the last inequality is in the previous lemmas.   
We apply Lemma \ref{lmfa6} with $b=6c_1s^2$ to
the probability on line \eqref{fa30}.    Since $s\le \tfrac12N^{1/3}$ (otherwise the probability in \eqref{fa-3} vanishes and there is nothing to prove),  in Lemma \ref{lmfa6} we take $c_0=1$ and to satisfy  $b\le c_0N^{2/3}$  we consider only $c_1\le \tfrac23$.   Pick $\kappa(\mu)\in (\e_0(\mu), \e_1(\mu))$ and set   $r=\kappa(\mu)b^{1/2}=\kappa(\mu)s\sqrt{6c_1}$.   Now Lemma \ref{lmfa6}  can be applied to bound 
  probability   \eqref{fa30}   by $C(\mu)s^{-3}$. 
 
Apply Lemma \ref{rwlm}(b) with $M=s$, 
$t=N^{1/3}$, $\alpha=C_5(\mu)r$ and  $\beta=3c_1s^2$,   to  bound the probability on line \eqref{fa32}   
also by $C(\mu)s^{-3}$.  The condition $2\beta\le M\alpha$ of that lemma is 
equivalent to $\sqrt{6c_1}\le C_5(\mu)\kappa(\mu)$, and we can fix $c_1$
small enough to satisfy
this.   This completes the proof of \eqref{fa-3} and thereby the proof of Theorem
\ref{tot-f-thm}.  \end{proof}

\begin{proof}[Proof of Theorem \ref{thm-Ztotbd}]
{\bf Case 1:}  $\theta\ne\mu/2$.  We do the subcase $0<\theta<\mu/2$.
By \eqref{Z2}, 
\be\begin{aligned}
\log Z_N^{\text{\rm p2l}}(\theta,\mu)  &= \log Z_{N,0} + 
 \log \biggl(1+  \sum_{k=1}^N  \prod_{i=1}^{k}\frac{V_{N-i+1,i}}{U_{N-i+1,i}}\biggr). 
\end{aligned}\label{fo3.9}\ee
Since 
\[  \E(\log {V_{N-i+1,i}} -\log {U_{N-i+1,i}}) = -\digamf(\mu-\theta)+\digamf(\theta)<0\]
the random variable 
\[   \log \biggl(1+  \sum_{k=1}^\infty  \prod_{i=1}^{k}\frac{V_{N-i+1,i}}{U_{N-i+1,i}}\biggr)  \]
is positive and  finite.  Since $\log Z_{N,0}$ is a sum of i.i.d.\  variables
$\log U_{i,0}$ with $U_{i,0}^{-1}\sim$ Gamma($\theta, 1$), the conclusions follow for the case $0<\theta<\mu/2$.

\bigskip

{\bf Case 2:}  $\theta=\mu/2$. 
 Let 
$ (m, n)=( N-\fl{N/2}, \fl{N/2} )$. Separate the partition function in the
characteristic direction and use \eqref{Z2}: 
\be\begin{aligned}
\log Z_N^{\text{\rm p2l}}(\mu/2,\mu)  
&= \log Z_{m,n} + 
\log \biggl( \, \sum_{k=0}^m \prod_{i=1}^{k}\frac{V_{m-i+1,n+i}}{U_{m-i+1,n+i}}\;+\; 
\sum_{k=1}^n  \prod_{i=1}^{k}\frac{U_{m+i,n-i+1}}{V_{m+i,n-i+1}}\biggr). 
\end{aligned}\label{fo4}\ee
 By the Burke property the mean zero random variables 
 $\eta_i=\log {U_{m+i,n-i+1}} -\log {V_{m+i,n-i+1}}$ for $i\in\bZ$ 
  are  i.i.d.  
 For $k\ge 1$ define  sums 
\[  S_k=\sum_{i=1}^k \eta_i\,, \ \  S_0=0    \quad\text{and}\quad  
S_{-k}=-\sum_{i=1}^{k} \eta_{-i+1}.   \]
At $\theta=\mu/2$,  $\E(\log Z_{m,n})=Ng(\mu/2,\mu)$.  Consequently \eqref{fo4} gives 
\be\begin{aligned}   \log Z_N^{\text{\rm p2l}}(\mu/2,\mu) - Ng(\mu/2,\mu) 
\; = \; \overline{ \log Z_{m,n}}  +  O({\log N}) +  \max_{-m\le k\le n} S_k. 
 \end{aligned} \label{fo6}\ee
 By the usual strong law of large numbers $N^{-1} \max_{-m\le k\le n} S_k\to 0 $ a.s.\ and 
 so together with 
  \eqref{Zllnbd},  \eqref{fo6}  gives the law of large numbers \eqref{llnZtotbd}
   in the case $\theta=\mu/2$. 
Second,  since $\overline{ \log Z_{m,n}}$ is 
stochastically of order $O(N^{1/3})$  by Theorem \ref{var-bd-thm}  and since 
$N^{-1/2}  \max_{-m\le k\le n} S_k$ converges weakly to $\zeta(\mu/2,\mu)$ defined in \eqref{defzeta2},
\eqref{fo6} implies also the weak limit \eqref{Ztotbdclt}.  
\end{proof}

\bibliography{growthrefs}
\bibliographystyle{plain}
\end{document}